\documentclass[12pt]{amsart}
\usepackage{graphicx}
\usepackage{amsmath, amscd, amssymb, amsthm}
\usepackage[subnum]{cases}
\usepackage{changepage}
\usepackage{xcolor}
\usepackage{marginnote}
\usepackage{hyperref}
\usepackage{cleveref}
\usepackage[all]{xy}
\usepackage{tabularx}
\usepackage{ltablex}
\usepackage[T1]{fontenc}
\usepackage[utf8]{inputenc}

\newtheorem{theorem}{Theorem}[section]
\newtheorem{lemma}[theorem]{Lemma}

\newtheorem{proposition}[theorem]{Proposition}
\newtheorem{corollary}[theorem]{Corollary}
\newtheorem{conjecture}[theorem]{Conjecture}

\theoremstyle{definition}
\newtheorem{definition}[theorem]{Definition}
\newtheorem{remark}[theorem]{Remark}
\numberwithin{equation}{section}
\newtheorem{example}[theorem]{Example}
\newtheorem{assumption}[theorem]{Assumption}
\newtheorem{setting}[theorem]{Setting}

\usepackage{fancyhdr}
\begin{document}

\normalfont

\title{Hodge-Iwasawa Theory II}
\author{Xin Tong}

\maketitle

\begin{abstract}
\rm We continue our study on the Hodge-Iwasawa theory which is a continuation of our previous work on Hodge-Iwasawa theory, which is aimed at higher dimensional deformation of higher dimensional Hodge structures over general analytic spaces or adic spaces. We still follow closely the approaches of Kedlaya-Liu to study our Frobenius modules over the different kinds of period rings including more generalized perfect period rings and the corresponding imperfect period rings. It is desirable that one can globalize the deformed information in order to organize the deformed period rings into the corresponding sheaves with respect to the coefficient sheaves, which will build up some foundations on further application to moduli of Frobenius sheaves.
\end{abstract}

\newpage

\tableofcontents

\newpage

\section{Introduction}

\subsection{Introduction}

\indent In our previous paper \cite{T1} on the Hodge-Iwasawa theory we studied some deformation of $p$-adic Hodge structures over some general perfectoid spaces and adic spaces. We mainly focused on the corresponding perfect setting which to some extent gives us some suitable and desired functoriality with respect to some general adic spaces. This reflects the corresponding motivation in some geometric and categorical way since our initial goal on the corresponding Hodge-Iwasawa consideration is to establish some possibility to deform the functorial geometric structures.\\

\indent In our current situation we still emphasize higher dimensional geometric structures. We will generalize the corresponding context of \cite{KL16} to our context which on one hand produces more general tools to study the deformation of Hodge structures and on the other hand we will be more flexible over some specific towers both coming from geometry and Iwasawa theory. The context of \cite{KL16} is already over some deep generalization. First off the towers considered in \cite{KL16} are axiomatic before some more concrete study. Also the corresponding context is within the consideration of generalized Witt vectors. Also the context considers the corresponding pseudocoherent modules and sheaves instead of the objects initially considered in the context of \cite{KL15}. We consider the corresponding deformations of these generalized structures over some affinoid algebras.\\

\indent In our current consideration we will do some kind of globalization to some extent which will generalize the corresponding story in some more rigid framework. To be more precise we will consider the situation where the deformation happens over some affinoid rigid analytic space. The sheafiness will be non-trivial here but will be true in this specific context.\\

\indent We will work in more general setting than the settings of our previous work and \cite{KL16}. Although the setting of \cite{KL16} is already more general than \cite{KL15}, we choose to work in the setting where the coefficient field $E$ is ultrametric with perfect residue field of characteristic $p>0$ (even in the nondiscrete valued situation). When $E$ is not discrete valued this is actually quite hard to deal with, where we only consider the situation where we do have some stable unique expression in the variable $\pi$ (for elements in the ring of Witt vectors), which we hope go back to this later to work within more general setting as long as one can define the corresponding power-multiplicative norms to do various completion under the hypothesis that $R$ is at least uniform Banach. Therefore our work will provide some additional reference to the study on this level of generality. Please be careful enough on our notations.

\subsection{Main Results}

\indent We have managed to do some generalization of our previous work \cite{T1}, and the corresponding context in \cite{KL15}, \cite{KL16}, \cite{KP}. As mentioned in the previous subsection the corresponding imperfect setting of the Hodge-Iwasawa theory is one main goal of our development. Also we would like to discuss to some extent the corresponding globalization of the deformed information which the rings in our consideration are carrying. To summarize we have:\\

I. We have now considered the corresponding equal characteristic analog of the corresponding comparison between the vector bundles, pseudocoherent sheaves over the deformed schematic Fargues-Fontaine curves and the corresponding Frobenius modules over the deformed Robba rings, generalizing \cite{T1}, \cite{KL16} and \cite{KP}. The corresponding motivation comes from certainly not only \cite{KP}, but also the corresponding $t$-motivic Hodge theory in the relative setting such as in Hartl-Kim \cite{HK}.\\

II. We have now considered the corresponding equal characteristic analog of the corresponding comparison between the vector bundles, pseudocoherent sheaves over the deformed adic Fargues-Fontaine curves and the corresponding Frobenius modules over the deformed Robba rings \cite{T1}, \cite{KL15}, \cite{KL16}. Again the corresponding motivation comes from certainly not only \cite{KP}, but also the corresponding $t$-motivic Hodge theory in the relative setting such as in Hartl-Kim \cite{HK}. This will lead to some sort of globalization in the category of adic spaces.\\ 


III. We have discussed the corresponding relationship between Frobenius modules over the perfect period rings and the Frobenius modules over the corresponding imperfection of the period rings in our deformed setting. This is to some extent important since the corresponding imperfection might lead to the corresponding noetherian objects which might be easy to control. This will have further application to the study of our deformation of the \'etale local systems.\\

IV. We have discussed the corresponding relationship between Frobenius modules over the perfect period rings and the Frobenius modules over the corresponding imperfection of the period rings in our deformed setting, but over the corresponding noncommutative Banach rings. This is to some extent important since the corresponding imperfection might lead to the corresponding noncommutative noetherian objects which might be easy to control. This will have further application to the study of our noncommutative deformation of the \'etale local systems. It is very natural to consider noncommutative setting in our development since we are considering the corresponding deformation into the noncommutative Fr\'echet-Stein algebras.\\

\subsection{Further Consideration}

We should mention that actually the corresponding ideas we mentioned within the globalization are not quite new since this bears some similarity to Pappas-Rapoport \cite{PR1} and Hellmann's work \cite{Hel1} on the arithmetic stacks of Frobenius modules. The corresponding comparison in our context could be morally conveyed in the following diagram:
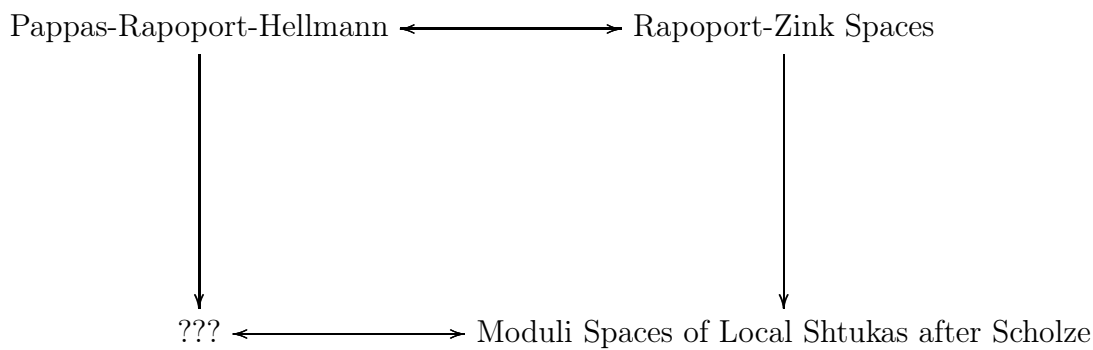
\[
\xymatrix@R+6pc@C+0pc{
\text{Pappas-Rapoport-Hellmann}
\ar[r]\ar[r]\ar[r]\ar[d]\ar[d]\ar[d] &\text{Rapoport-Zink Spaces}\ar[l]\ar[l]\ar[l] \ar[d]\ar[d]\ar[d]\\
\text{???} \ar[r]\ar[r]\ar[r] &\text{Moduli Spaces of Local Shtukas after Scholze} \ar[l]\ar[l]\ar[l] .\\
}
\]

For instance in our consideration we will consider some arithmetic deformation of vector bundles over the arithmetic deformation of Fargues-Fontaine curves. It is very natural to consider the corresponding moduli problem around the corresponding vector bundles over these deformed spaces. Certainly these are not directly diamonds, but at least they should be stacks over the corresponding categories of certain adic spaces.\\


In this paper, we have built some well-posed equivariant versions of many results in the relative $p$-adic Hodge theory in \cite{KL15} and \cite{KL16} which have already reached beyond our previous consideration. We expect everything could be further applied to the corresponding study on the cohomologies of all the corresponding objects we defined and studied here. For instance if one applies these to the $t$-motivic setting, then immediately one has the chance to construct the corresponding $t$-adic local Tamagawa number conjecture after Nakamura \cite{Nakamura1} (also inspired by the recent work \cite{FGHP}), which certainly is a relative version in rigid family (again after Nakamura relying on the results from Kedlaya-Pottharst-Xiao on the finiteness of the corresponding cohomology of $(\varphi,\Gamma)$-modules over relative Robba rings).\\

We have reached some noncommutative Iwasawa deformation of the relative $p$-adic Hodge structures. At least the corresponding hope is directly targeted at the corresponding noncommutative Tamagawa number conjectures after Burns-Flach-Fukaya-Kato \cite{BF1}, \cite{BF2} and \cite{FK1},  Nakamura and Z\"ahringer \cite{Zah1}. Our belief is that although we have not shown that coherent sheaves and modules over noncommutative deformation of full Robba rings are equivalent, but at least for coherent sheaves one could have a coherent Iwasawa theory in the noncommutative setting. \\


\subsection{Lists of Notations}

\indent The arrangement of the corresponding notations getting involved in this paper is bit complicated, so we have made some indication to indicate where they mainly emerge into the corresponding discussion.\\

\begin{center}
\begin{longtable}{p{7.8cm}p{8cm}}
Notation & Description (mainly in section 2, section 3, setion 4, section 6.1, section 6.2, setion 6.3) \\
\hline
$E$ & A complete discrete valued field which is nonarchimedean with residue field $k$ perfect in characteristic $p>0$.\\
$A$ & A reduced affinoid algebra in rigid analytic geometry in the sense of Tate, or a noncommutative Banach affinoid algebra.\\
$W_\pi$ & The ring of generalized Witt vectors with respect to some finite extension $E$ of $\mathbb{Q}_p$ with pseudo-uniformizer $\pi$. \\
$(R,R^+)$   & Adic perfect uniform algebra over $\mathbb{F}_{p^h}$ for some $h\geq 1$. \\
$\widetilde{\Omega}_{R,A}^\mathrm{int}$ & Deformed version of the period rings in the style of \cite{KL15} and \cite{KL16}.\\
$\widetilde{\Omega}_{R,A}$ & Deformed version of the period rings in the style of \cite{KL15} and \cite{KL16}.\\
$\widetilde{\Pi}_{R,A}^{\mathrm{int},r}$ & Deformed version of the period rings in the style of \cite{KL15} and \cite{KL16}.\\
$\widetilde{\Pi}_{R,A}^\mathrm{int}$ & Deformed version of the period rings in the style of \cite{KL15} and \cite{KL16}.\\
$\widetilde{\Pi}_{R,A}^{\mathrm{bd},r}$ & Deformed version of the period rings in the style of \cite{KL15} and \cite{KL16}.\\
$\widetilde{\Pi}_{R,A}^\mathrm{bd}$ & Deformed version of the period rings in the style of \cite{KL15} and \cite{KL16}.\\
$\widetilde{\Pi}_{R,A}^{r}$ & Deformed version of the period rings in the style of \cite{KL15} and \cite{KL16}.\\
$\widetilde{\Pi}_{R,A}$ & Deformed version of the period rings in the style of \cite{KL15} and \cite{KL16}.\\
$\widetilde{\Pi}_{R,A}^{\infty}$ & Deformed version of the period rings in the style of \cite{KL15} and \cite{KL16}.\\
$\widetilde{\Pi}_{R,A}^I$ & Deformed version of the period rings in the style of \cite{KL15} and \cite{KL16}.\\

$\widetilde{\Omega}_{*,A}^\mathrm{int}$ & Deformed version of the period sheaves in the style of \cite{KL15} and \cite{KL16}.\\
$\widetilde{\Omega}_{*,A}$ & Deformed version of the period sheaves in the style of \cite{KL15} and \cite{KL16}.\\
$\widetilde{\Pi}_{*,A}^{\mathrm{int},r}$ & Deformed version of the period sheaves in the style of \cite{KL15} and \cite{KL16}.\\
$\widetilde{\Pi}_{*,A}^\mathrm{int}$ & Deformed version of the period sheaves in the style of \cite{KL15} and \cite{KL16}.\\
$\widetilde{\Pi}_{*,A}^{\mathrm{bd},r}$ & Deformed version of the period sheaves in the style of \cite{KL15} and \cite{KL16}.\\
$\widetilde{\Pi}_{*,A}^\mathrm{bd}$ & Deformed version of the period sheaves in the style of \cite{KL15} and \cite{KL16}.\\
$\widetilde{\Pi}_{*,A}^{r}$ & Deformed version of the period sheaves in the style of \cite{KL15} and \cite{KL16}.\\
$\widetilde{\Pi}_{*,A}$ & Deformed version of the period sheaves in the style of \cite{KL15} and \cite{KL16}.\\
$\widetilde{\Pi}_{*,A}^{\infty}$ & Deformed version of the period sheaves in the style of \cite{KL15} and \cite{KL16}.\\
$\widetilde{\Pi}_{*,A}^I$ & Deformed version of the period sheaves in the style of \cite{KL15} and \cite{KL16}.\\

$W_{\pi,\infty}$ & The ring of generalized Witt vectors with respect to some finite extension $E$ of $\mathbb{Q}_p$ with pseudo-uniformizer $\pi$, but with base change to $E_\infty$. In this case each element admits unique expression $\sum_{n\in \mathbb{Z}[1/p]_{\geq 0}}\pi^n[\overline{x}_n]$. \\

$\widetilde{\Omega}_{R,\infty}^\mathrm{int},\widetilde{\Omega}_{*,\infty}^\mathrm{int}$ & Period rings or sheaves in the style of \cite{KL15} and \cite{KL16}, with base change to $E_\infty$.\\
$\widetilde{\Omega}_{R,\infty},\widetilde{\Omega}_{*,\infty}$ & Period rings or sheaves in the style of \cite{KL15} and \cite{KL16}, with base change to $E_\infty$.\\
$\widetilde{\Pi}_{R,\infty}^{\mathrm{int},r},\widetilde{\Pi}_{*,\infty}^{\mathrm{int},r}$ & Period rings or sheaves in the style of \cite{KL15} and \cite{KL16}, with base change to $E_\infty$.\\
$\widetilde{\Pi}_{R,\infty}^\mathrm{int},\widetilde{\Pi}_{*,\infty}^\mathrm{int}$ & Period rings or sheaves in the style of \cite{KL15} and \cite{KL16}, with base change to $E_\infty$.\\
$\widetilde{\Pi}_{R,\infty}^{\mathrm{bd},r},\widetilde{\Pi}_{*,\infty}^{\mathrm{bd},r}$ & Period rings or sheaves in the style of \cite{KL15} and \cite{KL16}, with base change to $E_\infty$.\\
$\widetilde{\Pi}_{R,\infty}^\mathrm{bd},\widetilde{\Pi}_{*,\infty}^\mathrm{bd}$ & Period rings or sheaves in the style of \cite{KL15} and \cite{KL16}, with base change to $E_\infty$.\\
$\widetilde{\Pi}_{R,\infty}^{r},\widetilde{\Pi}_{*,\infty}^{r}$ & Period rings or sheaves in the style of \cite{KL15} and \cite{KL16}, with base change to $E_\infty$.\\
$\widetilde{\Pi}_{R,\infty},\widetilde{\Pi}_{*,\infty}$ & Period rings or sheaves in the style of \cite{KL15} and \cite{KL16}, with base change to $E_\infty$.\\
$\widetilde{\Pi}_{R,\infty}^{\infty},\widetilde{\Pi}_{*,\infty}^{\infty}$ & Period rings or sheaves in the style of \cite{KL15} and \cite{KL16}, with base change to $E_\infty$.\\
$\widetilde{\Pi}_{R,\infty}^I,\widetilde{\Pi}_{*,\infty}^I$ & Period rings or sheaves in the style of \cite{KL15} and \cite{KL16}, with base change to $E_\infty$.\\

$\widetilde{\Omega}_{R,\infty,A}^\mathrm{int},\widetilde{\Omega}_{*,\infty,A}^\mathrm{int}$ & Deformed period rings or sheaves in the style of \cite{KL15} and \cite{KL16}, with base change to $E_\infty$.\\
$\widetilde{\Omega}_{R,\infty,A},\widetilde{\Omega}_{*,\infty,A}$ & Deformed period rings or sheaves in the style of \cite{KL15} and \cite{KL16}, with base change to $E_\infty$.\\
$\widetilde{\Pi}_{R,\infty,A}^{\mathrm{int},r},\widetilde{\Pi}_{*,\infty,A}^{\mathrm{int},r}$ & Deformed period rings or sheaves in the style of \cite{KL15} and \cite{KL16}, with base change to $E_\infty$.\\
$\widetilde{\Pi}_{R,\infty,A}^\mathrm{int},\widetilde{\Pi}_{*,\infty,A}^\mathrm{int}$ & Deformed period rings or sheaves in the style of \cite{KL15} and \cite{KL16}, with base change to $E_\infty$.\\
$\widetilde{\Pi}_{R,\infty,A}^{\mathrm{bd},r},\widetilde{\Pi}_{*,\infty,A}^{\mathrm{bd},r}$ & Deformed period rings or sheaves in the style of \cite{KL15} and \cite{KL16}, with base change to $E_\infty$.\\
$\widetilde{\Pi}_{R,\infty,A}^\mathrm{bd},\widetilde{\Pi}_{*,\infty,A}^\mathrm{bd}$ & Deformed period rings or sheaves in the style of \cite{KL15} and \cite{KL16}, with base change to $E_\infty$.\\
$\widetilde{\Pi}_{R,\infty,A}^{r},\widetilde{\Pi}_{*,\infty,A}^{r}$ & Deformed period rings or sheaves in the style of \cite{KL15} and \cite{KL16}, with base change to $E_\infty$.\\
$\widetilde{\Pi}_{R,\infty,A},\widetilde{\Pi}_{*,\infty,A}$ & Deformed period rings or sheaves in the style of \cite{KL15} and \cite{KL16}, with base change to $E_\infty$.\\
$\widetilde{\Pi}_{R,\infty,A}^{\infty},\widetilde{\Pi}_{*,\infty,A}^{\infty}$ & Deformed period rings or sheaves in the style of \cite{KL15} and \cite{KL16}, with base change to $E_\infty$.\\
$\widetilde{\Pi}_{R,\infty,A}^I,\widetilde{\Pi}_{*,\infty,A}^I$ & Deformed period rings or sheaves in the style of \cite{KL15} and \cite{KL16}, with base change to $E_\infty$.\\

%
%
\end{longtable}
\end{center}

\begin{center}
\begin{longtable}{p{7.8cm}p{8cm}}
Notation & Description (mainly in section 5, section 6.4, section 6.5, section 6.6, section 6.7) \\
\hline
$(H_\bullet,H^+_\bullet)$ & A tower as in \cite[Chapter 5]{KL16}.\\
$\Pi^{\mathrm{int},r}_{H,A}$ & Imperfect relative period ring corresponding to $(H_\bullet,H^+_\bullet)$ which is deformed analog of Kedlaya-Liu's imperfect ring.\\
$\Pi^{\mathrm{int},\dagger}_{H,A}$ & Imperfect relative period ring corresponding to $(H_\bullet,H^+_\bullet)$ which is deformed analog of Kedlaya-Liu's imperfect ring.\\
$\Omega^\mathrm{int}_{H,A}$ & Imperfect relative period ring corresponding to $(H_\bullet,H^+_\bullet)$ which is deformed analog of Kedlaya-Liu's imperfect ring.\\
$\Omega_{H,A}$ & Imperfect relative period ring corresponding to $(H_\bullet,H^+_\bullet)$ which is deformed analog of Kedlaya-Liu's imperfect ring.\\
$\Pi^{\mathrm{bd},r}_{H,A}$ & Imperfect relative period ring corresponding to $(H_\bullet,H^+_\bullet)$ which is deformed analog of Kedlaya-Liu's imperfect ring.\\
$\Pi^{\mathrm{bd},\dagger}_{H,A}$ & Imperfect relative period ring corresponding to $(H_\bullet,H^+_\bullet)$ which is deformed analog of Kedlaya-Liu's imperfect ring.\\
$\Pi^{[s,r]}_{H,A}$ & Imperfect relative period ring corresponding to $(H_\bullet,H^+_\bullet)$ which is deformed analog of Kedlaya-Liu's imperfect ring.\\
$\Pi^r_{H,A}$ & Imperfect relative period ring corresponding to $(H_\bullet,H^+_\bullet)$ which is deformed analog of Kedlaya-Liu's imperfect ring.\\
$\Pi_{H,A}$ & Imperfect relative period ring corresponding to $(H_\bullet,H^+_\bullet)$ which is deformed analog of Kedlaya-Liu's imperfect ring.\\

$\breve{\Pi}^{\mathrm{int},r}_{H,A}$ & Imperfect relative period ring corresponding to $(H_\bullet,H^+_\bullet)$ which is deformed analog of Kedlaya-Liu's imperfect ring.\\
$\breve{\Pi}^{\mathrm{int},\dagger}_{H,A}$ & Imperfect relative period ring corresponding to $(H_\bullet,H^+_\bullet)$ which is deformed analog of Kedlaya-Liu's imperfect ring.\\
$\breve{\Omega}^\mathrm{int}_{H,A}$ & Imperfect relative period ring corresponding to $(H_\bullet,H^+_\bullet)$ which is deformed analog of Kedlaya-Liu's imperfect ring.\\
$\breve{\Omega}_{H,A}$ & Imperfect relative period ring corresponding to $(H_\bullet,H^+_\bullet)$ which is deformed analog of Kedlaya-Liu's imperfect ring.\\
$\breve{\Pi}^{\mathrm{bd},r}_{H,A}$ & Imperfect relative period ring corresponding to $(H_\bullet,H^+_\bullet)$ which is deformed analog of Kedlaya-Liu's imperfect ring.\\
$\breve{\Pi}^{\mathrm{bd},\dagger}_{H,A}$ & Imperfect relative period ring corresponding to $(H_\bullet,H^+_\bullet)$ which is deformed analog of Kedlaya-Liu's imperfect ring.\\
$\breve{\Pi}^{[s,r]}_{H,A}$ & Imperfect relative period ring corresponding to $(H_\bullet,H^+_\bullet)$ which is deformed analog of Kedlaya-Liu's imperfect ring.\\
$\breve{\Pi}^r_{H,A}$ & Imperfect relative period ring corresponding to $(H_\bullet,H^+_\bullet)$ which is deformed analog of Kedlaya-Liu's imperfect ring.\\
$\breve{\Pi}_{H,A}$ & Imperfect relative period ring corresponding to $(H_\bullet,H^+_\bullet)$ which is deformed analog of Kedlaya-Liu's imperfect ring.\\

$\widehat{\Pi}^{\mathrm{int},r}_{H,A}$ & Imperfect relative period ring corresponding to $(H_\bullet,H^+_\bullet)$ which is deformed analog of Kedlaya-Liu's imperfect ring.\\
$\widehat{\Pi}^{\mathrm{int},\dagger}_{H,A}$ & Imperfect relative period ring corresponding to $(H_\bullet,H^+_\bullet)$ which is deformed analog of Kedlaya-Liu's imperfect ring.\\
$\widetilde{\Omega}^\mathrm{int}_{H,A}$ & Imperfect relative period ring corresponding to $(H_\bullet,H^+_\bullet)$ which is deformed analog of Kedlaya-Liu's imperfect ring.\\
$\widetilde{\Omega}_{H,A}$ & Imperfect relative period ring corresponding to $(H_\bullet,H^+_\bullet)$ which is deformed analog of Kedlaya-Liu's imperfect ring.\\
$\widehat{\Pi}^{\mathrm{bd},r}_{H,A}$ & Imperfect relative period ring corresponding to $(H_\bullet,H^+_\bullet)$ which is deformed analog of Kedlaya-Liu's imperfect ring.\\
$\widehat{\Pi}^{\mathrm{bd},\dagger}_{H,A}$ & Imperfect relative period ring corresponding to $(H_\bullet,H^+_\bullet)$ which is deformed analog of Kedlaya-Liu's imperfect ring.\\

\end{longtable}
\end{center}

\newpage

\section{Arithmetic Families of Period Rings}

\subsection{Basic Settings and Basic Definitions of Period Rings}

\indent In this section we first define the corresponding period rings and period sheaves involved in our study which is generalized version of the study in our previous paper, where we consider the corresponding ring of usual $p$-typical Witt vectors. For the convenience of the readers we present the corresponding construction in detail to some extent. First following \cite[Hypothesis 3.1.1]{KL16} we consider the following setting:

\begin{setting}
In our context we consider the corresponding generalized Witt vectors in the style of \cite[3.1]{KL16} where we first fix some complete nonarchimedean (normalized as in \cite[Hypothesis 3.1.1]{KL16}) discrete valued field $E$ where we will use the notation $\pi$ to denote a chosen pseudo-uniformizer of $E$, and we use the notation $\mathfrak{o}_E$ to denote the corresponding ring of integers of $E$. Slightly generalizing the context \cite[Hypothesis 3.1.1]{KL16} of \cite{KL16} we consider the corresponding situation where $k$ is a perfect field of characteristic $p>0$, and we assume the field $k$ to contain $\mathbb{F}_{p^h}$. We now assume that the ultrametric field $E$ has residue field $k$. Again as in our previous paper we will use the notation $(R,R^+)$ to denote a perfect adic Banach uniform algebra over $k$. Following \cite{KL16} we will use the notation $\alpha$ to denote the spectral seminorms on perfect adic Banach uniform algebras. When $E$ is of equal-characteristic we assume that $E$ contains the field $\mathbb{F}_p((\eta))$. 
\end{setting}

\begin{remark}
The field $E$ is assumed to be discrete valued. One can actually consider more general $E$. The most general situation (namely just assume $E$ is ultrametric field) is actually a little bit hard to manage due to the fact that the element in the genralized Witt vector will have no unique expression as a power series, but in some nice situations this is not that hard for instance one can consider the base change to the perfectoid field $E_\infty$.	
\end{remark}

\indent Then we could consider the following definitions of the period rings in our situation which is just the corresponding deformation over $A$ of the corresponding rings defined in \cite[Definition 4.1.1]{KL16} in some generalized fashion:

\begin{setting}
Recall that in our situation we have the corresponding rings of generalized Witt vectors associated to some perfect ring $R$ or $R^+$. Generalizing the situation in \cite[Definition 4.1.1]{KL16} and \cite[Setting 3.1.1]{KL16} we denote them by $W_\pi(R)$ or $W_\pi(R^+)$ where $\pi$ is the chosen pseudo-uniformizer as defined. To be more precise here $W_\pi(R)$ (same to $W_\pi(R^+)$) is defined to be
\begin{displaymath}
W_\pi(R):=W(R)\otimes_{W(k)}\mathcal{O}_E.	
\end{displaymath}
Then we have the following ring of period rings from this notation:
\begin{displaymath}
\widetilde{\Omega}^\mathrm{int}_{R}=W_\pi(R),	\widetilde{\Omega}_{R}=W_\pi(R)[1/\pi].
\end{displaymath}
Each element of $W_\pi(R)$ could be expressed as some power series:
\begin{displaymath}
\sum_{k\geq 0}\pi^k[\overline{x}_k]	
\end{displaymath}
where one has the corresponding Gauss norm coming from the norm $\alpha$ in the style of for each $r>0$ 
\begin{displaymath}
\left\|.\right\|_{\alpha^r}(\sum_{k\geq 0}\pi^k[\overline{x}_k]):=\sup_{k\geq 0}\{p^{-k}\alpha(\overline{x}_k)^r\}.	
\end{displaymath}
Then one can define the corresponding integral Robba ring $\widetilde{\Pi}^{\mathrm{int},r}_{R}$ by taking the completion of the ring $W_\pi(R^+)[[R]]$ by using the norm defined above.
From this one has the corresponding bounded Robba ring $\widetilde{\Pi}_R^{\mathrm{bd},r}$ which is defined just to be $\widetilde{\Pi}_R^{\mathrm{int},r}[1/\pi]$. Then we consider the corresponding Robba $\widetilde{\Pi}^I_R$ with respect to some interval $I\subset (0,\infty)$ which is defined to be the corresponding ind-Fr\'echet completion of $W_\pi(R^+)[[R]][1/\pi]$ (namely $W_\pi(R^+)[[r],r\in R][1/\pi]$) with respect to the family of norms $\left\|.\right\|_{\alpha^r}$, $\forall r\in I$. By taking specific intervals like $(0,r]$ or $(0,\infty)$ we have the corresponding Robba ring $\widetilde{\Pi}_R^\mathrm{r}$ and $\widetilde{\Pi}_R^\mathrm{\infty}$. Then by taking the corresponding union throughout all the radius $r>0$ we have the corresponding Robba rings $\widetilde{\Pi}^\mathrm{int}_R,\widetilde{\Pi}^\mathrm{bd}_R,\widetilde{\Pi}_R$. As in \cite[4.1]{KL16} we have also the corresponding integral versions of some of the rings defined above.
\end{setting}

\indent Our consideration is to extend the scope of the consideration mentioned above by extending the corresponding dimension of the ring involved to some extent. To be more precise one considers the corresponding power or Laurent series over the corresponding functional analytic algebras, and then take suitable quotient.

\begin{setting}
In characteristic zero situation, we are going to use the notation $\mathbb{Q}_p\{T_1,...T_d\}$ to denote a Tate algebra in rigid analytic geometry after Tate, and use in general $A$ to represent the corresponding quotients of the Tate algebras over $\mathbb{Q}_p$ as above. Throughout we assume $A$ to be reduced, carrying the spectral seminorm. While in positive characteristic situation we are going to use the notation $A$ to denote a general affinoid algebra which is assumed to be reduced.
\end{setting}

We first deform the corresponding constructions in \cite[Definition 4.1.1]{KL16}:

\begin{definition}
We first consider the corresponding deformation of the above rings over $\mathbb{Q}_p\{T_1,...T_d\}$. We are going to use the notation $W_\pi(R)_{\mathbb{Q}_p\{T_1,...T_d\}}$ to denote the complete tensor product of $W_\pi(R)$ with the Tate algebra $\mathbb{Q}_p\{T_1,...T_d\}$ consisting of all the element taking the form as:
\begin{displaymath}
\sum_{k\geq 0,i_1\geq 0,...,i_d\geq 0}\pi^k[\overline{x}_{k,i_1,...,i_d}]T_1^{i_1}T_2^{i_2}...T_d^{i_d}	
\end{displaymath}
over which we have the Gauss norm $\left\|.\right\|_{\alpha^r,\mathbb{Q}_p\{T_1,...T_d\}}$ for any $r>0$ which is defined by:
\begin{displaymath}
\left\|.\right\|_{\alpha^r,\mathbb{Q}_p\{T_1,...T_d\}}(\sum_{k\geq 0,i_1\geq 0,...,i_d\geq 0}\pi^k[\overline{x}_{k,i_1,...,i_d}]T_1^{i_1}T_2^{i_2}...T_d^{i_d}):=\sup_{k\geq 0,i_1\geq 0,...,i_d\geq 0}p^{-k}\alpha(\overline{x}_{k,i_1,...,i_d})^r.	
\end{displaymath}
Then we define the corresponding convergent rings:
\begin{displaymath}
\widetilde{\Omega}^\mathrm{int}_{R,\mathbb{Q}_p\{T_1,...T_d\}}=W_\pi(R)_{\mathbb{Q}_p\{T_1,...T_d\}},	\widetilde{\Omega}_{R,\mathbb{Q}_p\{T_1,...T_d\}}=W_\pi(R)_{\mathbb{Q}_p\{T_1,...T_d\}}[1/\pi].	
\end{displaymath}
Then as in the previous setting we define the corresponding $\mathbb{Q}_p\{T_1,...,T_d\}$-relative integral Robba ring $\widetilde{\Pi}_{R,\mathbb{Q}_p\{T_1,...T_d\}}^{\mathrm{int},r}$ as the completion of the ring $W_{\pi}(R^+)_{\mathbb{Q}_p\{T_1,...,T_d\}}[[R]]$ by using the Gauss norm defined above.
Then one can take the corresponding union of all $\widetilde{\Pi}_{R,\mathbb{Q}_p\{T_1,...T_d\}}^{\mathrm{int},r}$ throughout all $r>0$ to define the integral Robba ring $\widetilde{\Pi}_{R,\mathbb{Q}_p\{T_1,...T_d\}}^{\mathrm{int}}$. For the bounded Robba rings we just set $\widetilde{\Pi}_{R,\mathbb{Q}_p\{T_1,...T_d\}}^{\mathrm{bd},r}=\widetilde{\Pi}_{R,\mathbb{Q}_p\{T_1,...T_d\}}^{\mathrm{int},r}[1/\pi]$ and take the union throughout all $r>0$ to define the corresponding ring $\widetilde{\Pi}_{R,\mathbb{Q}_p\{T_1,...T_d\}}^{\mathrm{bd}}$. Then we define the corresponding Robba ring $\widetilde{\Pi}_{R,\mathbb{Q}_p\{T_1,...T_d\}}^{I}$ with respect to some interval $I\subset (0,\infty)$ by taking the Fr\'echet completion of ${W_\pi(R^+)}_{\mathbb{Q}_p\{T_1,...T_d\}}[[R]][1/\pi]$ with respect to all the norms $\left\|.\right\|_{\alpha^r,\mathbb{Q}_p\{T_1,...T_d\}}$ for all $r\in I$ which means that the corresponding equivalence classes in the completion procedure will be simultaneously Cauchy with respect to all the norms $\left\|.\right\|_{\alpha^r,\mathbb{Q}_p\{T_1,...T_d\}}$ for all $r\in I$. Then we take suitable intervals such as $(0,r]$ and $(0,\infty)$ to define the corresponding Robba rings $\widetilde{\Pi}_{R,\mathbb{Q}_p\{T_1,...T_d\}}^{r}$ and $\widetilde{\Pi}_{R,\mathbb{Q}_p\{T_1,...T_d\}}^{\infty}$, respectively. Then taking the union throughout all $r>0$ one can define the corresponding Robba ring $\widetilde{\Pi}_{R,\mathbb{Q}_p\{T_1,...T_d\}}$. Again as in \cite{KL16} one can define the corresponding integral rings of the similar types.
\end{definition}

\indent Then one can define the corresponding period rings in our context deformed over some affinoid $A$ which is isomorphic to some quotient of $\mathbb{Q}_p\{T_1,...T_d\}$, again deforming from the context of \cite[Definition 4.1.1]{KL16}.

\begin{definition}
In the characteristic $0$, we define the following period rings:
\begin{displaymath}
\widetilde{\Omega}^\mathrm{int}_{R,A},\widetilde{\Omega}_{R,A},\widetilde{\Pi}^\mathrm{int,r}_{R,A},\widetilde{\Pi}^\mathrm{bd,r}_{R,A}, \widetilde{\Pi}^I_{R,A},\widetilde{\Pi}^r_{R,A},\widetilde{\Pi}^\infty_{R,A}	
\end{displaymath}
by taking the suitable quotient of the following period rings defined above: 
\begin{align}
\widetilde{\Omega}^\mathrm{int}_{R,\mathbb{Q}_p\{T_1,...T_d\}},&\widetilde{\Omega}_{R,\mathbb{Q}_p\{T_1,...T_d\}},\widetilde{\Pi}^\mathrm{int,r}_{R,\mathbb{Q}_p\{T_1,...T_d\}},\widetilde{\Pi}^\mathrm{bd,r}_{R,\mathbb{Q}_p\{T_1,...T_d\}},\widetilde{\Pi}^I_{R,\mathbb{Q}_p\{T_1,...T_d\}},\\
 &\widetilde{\Pi}^r_{R,\mathbb{Q}_p\{T_1,...T_d\}},\widetilde{\Pi}^\infty_{R,\mathbb{Q}_p\{T_1,...T_d\}}	
\end{align}
with respect to the structure of the affinoid algebra $A$ in the sense of Tate. The corresponding constructions do not depend on the corresponding choices of the presentations. Therefore they carry the corresponding quotient seminorms of the above Gauss norms defined in the previous definition, note that these are not something induced from the corresponding spectral seminorms from $A$ chosen at the very beginning of our study. We use the notation $\overline{\left\|.\right\|}_{\alpha^r,\mathbb{Q}_p\{T_1,...T_d\}}$ to denote the corresponding quotient Gauss norm which induces then the corresponding spectral seminorm $\left\|.\right\|_{\alpha^r,A}$ for each $r>0$. Then we define the corresponding period rings:
\begin{align}
\widetilde{\Pi}^\mathrm{int}_{R,A},\widetilde{\Pi}^\mathrm{bd}_{R,A},\widetilde{\Pi}_{R,A}	
\end{align}
by taking suitable union throughout all $r>0$.
\end{definition}

\begin{definition}
In positive characteristic situation, when we are working over general affinoid algebra $A$, we use the same notations as in the previous definition, but by using $W(R)_{\mathbb{F}_p[[\eta]]\{T_1,...,T_d\}}$ and $W(R)_{A}$ as the starting rings, namely here $A$ is isomorphic to a quotient of $\mathbb{F}_p((\eta))\{T_1,...,T_d\}$.	
\end{definition}

\indent We can then define the corresponding rings over the ring $E_\infty$ which is defined to be the completion of the ring $E(\pi^{p^{{-\infty}}})$. We first consider the corresponding undeformed versions (after \cite[Definition 4.1.1]{KL16}):

\begin{definition}
Consider now the base change $W_{\pi,\infty}(R)$ which is defined now to be the completed tensor product $W_{\pi}(R)\widehat{\otimes}_{\mathcal{O}_E}\mathcal{O}_{E_\infty}$. Then the point is that each element in this ring admits a unique expression taking the form of $\sum_{n\in \mathbb{Z}[1/p]_{\geq 0}}\pi^n[\overline{x}_n]$ which allows us to perform the construction mentioned above. First we can define for some radius $r>0$ the corresponding period ring $\widetilde{\Pi}^{\mathrm{int},r}_{R,\infty}$ by taking the completion of the ring
\begin{displaymath}
W_{\pi,\infty}(R^+)[[R]]	
\end{displaymath}
with respect to the following Gauss type norm:
\begin{displaymath}
\left\|.\right\|_{\alpha^r}(\sum_{n\in \mathbb{Z}[1/p]_{\geq 0}}\pi^n[\overline{x}_n]):=\sup_{n\in \mathbb{Z}[1/p]_{\geq 0}}\{p^{-n}\alpha(\overline{x}_n)^r\}.	
\end{displaymath}
Then we can define the union $\widetilde{\Pi}^{\mathrm{int}}_{R,\infty}$ throughout all the radius $r>0$. Then we just define the bounded Robba ring $\widetilde{\Pi}^{\mathrm{bd},r}_{R,\infty}$ by $\widetilde{\Pi}^{\mathrm{int},r}_{R,\infty}[1/\pi]$ and also we could define the union $\widetilde{\Pi}^{\mathrm{bd}}_{R,\infty}$ throughout all the radius $r>0$. Then for any interval in $(0,\infty)$ which is denoted by $I$ we can define the corresponding Robba rings $\widetilde{\Pi}^{I}_{R,\infty}$ by taking the Fr\'echet completion of 
\begin{displaymath}
W_{\pi,\infty}(R^+)[[R]][1/\pi]	
\end{displaymath}
with respect to all the norms $\left\|.\right\|_{\alpha^t}$ for all $t\in I$. Then by taking suitable specified intervals one can define the rings $\widetilde{\Pi}^{r}_{R,\infty}$ and $\widetilde{\Pi}^{\infty}_{R,\infty}$ as before, and finally one can define the corresponding union $\widetilde{\Pi}_{R,\infty}$ throughout all the radius $r>0$. Again we have the corresponding integral version of the rings defined over $E_\infty$ as \cite{KL16}.

\end{definition}

\indent Then we can define the following affinoid deformations (after \cite[Definition 4.1.1]{KL16} in the flavor as above):

\begin{definition}
Consider now the base change $W_{\pi,\infty,\mathbb{Q}_p\{T_1,...,T_d\}}(R)$ which is defined now to be the completed tensor product $(W_{\pi}(R)\widehat{\otimes}_{\mathcal{O}_E}\mathcal{O}_{E_\infty})\widehat{\otimes}_{\mathbb{Q}_p}\mathbb{Q}_p\{T_1,...,T_d\}$. Then the point is that each element in this ring admits a unique expression taking the form of 
\begin{center}
$\sum_{n\in \mathbb{Z}[1/p]_{\geq 0},i_1\geq 0,...,i_d\geq 0}\pi^n[\overline{x}_{n,i_1,...,i_d}]T_1^{i_1}...T_d^{i_d}$ 	
\end{center}
which allows us to perform the construction mentioned above. First we can define for some radius $r>0$ the corresponding period ring $\widetilde{\Pi}^{\mathrm{int},r}_{R,\infty,\mathbb{Q}_p\{T_1,...,T_d\}}$ by taking the completion of the ring
\begin{displaymath}
W_{\pi,\infty,\mathbb{Q}_p\{T_1,...,T_d\}}(R^+)[[R]]	
\end{displaymath}
with respect to the following Gauss type norm:
\begin{displaymath}
\left\|.\right\|_{\alpha^r}(\sum_{n\in \mathbb{Z}[1/p]_{\geq 0},i_1\geq 0,...,i_d\geq 0}\pi^n[\overline{x}_{n,i_1,...,i_d}]T_1^{i_1}...T_d^{i_d}):=\sup_{n\in \mathbb{Z}[1/p]_{\geq 0},i_1\geq 0,...,i_d\geq 0}\{p^{-n}\alpha(\overline{x}_{n,i_1,...,i_d})^r\}.	
\end{displaymath}
Then we can define the union $\widetilde{\Pi}^{\mathrm{int}}_{R,\infty,\mathbb{Q}_p\{T_1,...,T_d\}}$ throughout all the radius $r>0$. Then we just define the bounded Robba ring $\widetilde{\Pi}^{\mathrm{bd},r}_{R,\infty,\mathbb{Q}_p\{T_1,...,T_d\}}$ by $\widetilde{\Pi}^{\mathrm{int},r}_{R,\infty,\mathbb{Q}_p\{T_1,...,T_d\}}[1/\pi]$ and also we could define the union $\widetilde{\Pi}^{\mathrm{bd}}_{R,\infty,\mathbb{Q}_p\{T_1,...,T_d\}}$ throughout all the radius $r>0$. Then for any interval in $(0,\infty)$ which is denoted by $I$ we can define the corresponding Robba rings $\widetilde{\Pi}^{I}_{R,\infty,\mathbb{Q}_p\{T_1,...,T_d\}}$ by taking the Fr\'echet completion of 
\begin{displaymath}
W_{\pi,\infty,\mathbb{Q}_p\{T_1,...,T_d\}}(R^+)[[R]][1/\pi]	
\end{displaymath}
with respect to all the norms $\left\|.\right\|_{\alpha^t}$ for all $t\in I$. Then by taking suitable specified intervals one can define the rings $\widetilde{\Pi}^{r}_{R,\infty,\mathbb{Q}_p\{T_1,...,T_d\}}$ and $\widetilde{\Pi}^{\infty}_{R,\infty,\mathbb{Q}_p\{T_1,...,T_d\}}$ as before, and finally one can define the corresponding union $\widetilde{\Pi}_{R,\infty,\mathbb{Q}_p\{T_1,...,T_d\}}$ throughout all the radius $r>0$. Again we have the corresponding integral version of the rings defined over $E_\infty$ as \cite{KL16}. Finally over $A$ we can define in the same way as above to deform all the rings over $E_\infty$, and we do not repeat the construction again.

\end{definition}

\subsection{Basic Properties of Period Rings}	
	
\indent Then we do some reality checks over the investigation of the properties of the above period rings in the style taken in \cite[Section 5.2]{KL15} and \cite{KL16}. 

\begin{proposition} \mbox{\bf{(After Kedlaya-Liu \cite[Lemma 5.2.1]{KL15})}}
The function $t\mapsto \left\|x\right\|_{\alpha^t,\mathbb{Q}_p\{T_1,...,T_d\}}$ for $x\in \widetilde{\Pi}^r_{R,\mathbb{Q}_p\{T_1,...,T_d\}}$ is continuous log convex for the corresponding variable $t\in (0,r]$ for any $r>0$.
\end{proposition}

\begin{proof}
Adapt the corresponding argument in the proof of 5.2.1 of \cite[Lemma 5.2.1]{KL15} to our situation we then first look at the situation where the element is just of the form of $\pi^k[\overline{x}_{k,i_1,...,i_d}]T_1^{i_1}...T_d^{i_d}$ where the corresponding norm function in terms of $t>0$ is just affine. Then one focuses on the finite sums of these kind of elements which gives rise to to the log convex directly. Finally by taking the approximation we get the desired result.	
\end{proof}

\begin{proposition}\mbox{\bf{(After Kedlaya-Liu \cite[Lemma 5.2.2]{KL15})}}
For any element $x\in \widetilde{\Pi}_{R,\mathbb{Q}_p\{T_1,...,T_d\}}$ we have that $x\in \widetilde{\Pi}^\mathrm{bd}_{R,\mathbb{Q}_p\{T_1,...,T_d\}}$ if and only if we have the situation where $x$ actually lives in $\widetilde{\Pi}^r_{R,\mathbb{Q}_p\{T_1,...,T_d\}}$ (for some specific $r>0$) such that $x$ itself is bounded under the norm $\left\|.\right\|_{\alpha^t,\mathbb{Q}_p\{T_1,...,T_d\}}$ for each $t\in (0,r]$.	
\end{proposition}

\begin{proof}
One direction of the proof is easy, so we only choose to present the proof of the implication in the other direction as in the original proof of 5.2.2 of \cite[Lemma 5.2.2]{KL15} as in the following. First choose some radius $r>0$ such that the element could be assumed to be living in the ring $\widetilde{\Pi}^r_{R,\mathbb{Q}_p\{T_1,...,T_d\}}$. The idea is to transfer the original question to the question about showing the integrality of $x$ when we add some hypothesis on the norm by taking suitable powers of $p$ (since the norm is bounded for each $t\in (0,r]$ so we are reduced to the situation where the norm is bounded by $1$). Then we argue as in \cite[Lemma 5.2.2]{KL15} to choose some approximating sequence $\{x_i\}$ living in $\widetilde{\Pi}^\mathrm{bd}_{R,\mathbb{Q}_p\{T_1,...,T_d\}}$ of $x$. Therefore we have for any $j\geq 1$ one can find then some integer $N_j\geq 1$ such that for any $i\geq N_j$ we have the estimate:
\begin{displaymath}
\left\|.\right\|_{\alpha^t,\mathbb{Q}_p\{T_1,...,T_d\}}(x_i-x)\leq p^{-j}, \forall t\in [p^{-j}r,r].	
\end{displaymath}
Then the idea is to consider the integral decomposition of the element $x_i$ which has the form of $\sum_{k=n(x_i),i_1\geq 0,i_2\geq 0,...,i_d\geq 0}\pi^k[\overline{x}_{i,k,i_1,...,i_d}]$ into the following two parts:
\begin{align}
x_i &:= y_i+z_i\\
	&:=\sum_{k=0,i_1\geq 0,i_2\geq 0,...,i_d\geq 0}\pi^k[\overline{x}_{i,k,i_1,...,i_d}]T_1^{i_1}...T_d^{i_d}+z_i
\end{align}
from which we actually have the corresponding estimate over the residual part of the decomposition above:
\begin{displaymath}
\left\|.\right\|_{\alpha^{p^{-j}r},\mathbb{Q}_p\{T_1,...,T_d\}}(\pi^k[\overline{x}_{i,k,i_1,...,i_d}]T_1^{i_1}...T_d^{i_d})\leq 1, \forall k<0	
\end{displaymath}
which implies by direct computation:
\begin{align}
\alpha^{p^{-j}r}(\overline{x}_{i,k,i_1,...,i_d})&\leq p^k\\
\alpha(\overline{x}_{i,k,i_1,...,i_d})&\leq p^{kp^{j}/r}	
\end{align}
which implies that we have the following estimate: 
\begin{align}
\left\|.\right\|_{\alpha^{r},\mathbb{Q}_p\{T_1,...,T_d\}}(x_i-y_i)&\leq p^{-k}p^{kp^{j}r/r}\\
&\leq p^{1-p^j}
\end{align}
which implies that $y_i\rightarrow x$ under the norm $\left\|.\right\|_{\alpha^{r},\mathbb{Q}_p\{T_1,...,T_d\}}$ which furthermore under the norm $\left\|.\right\|_{\alpha^{r},\mathbb{Q}_p\{T_1,...,T_d\}}$ due the property of the norm, which finishes the proof of the desired result.
\end{proof}

\begin{proposition}\mbox{\bf{(After Kedlaya-Liu \cite[Corollary 5.2.3]{KL15})}} We have the following identity:
\begin{displaymath}
(\widetilde{\Pi}_{R,\mathbb{Q}_p\{T_1,...,T_d\}})^\times=(\widetilde{\Pi}^\mathrm{bd}_{R,\mathbb{Q}_p\{T_1,...,T_d\}})^\times.
\end{displaymath}
\end{proposition}

\begin{proof}
See 5.2.3 of \cite[Corollary 5.2.3]{KL15}.	
\end{proof}

\begin{proposition}\mbox{\bf{(After Kedlaya-Liu \cite[Lemma 5.2.6]{KL15})}}
For any $0< r_1\leq r_2$ we have the following equality on the corresponding period rings:
\begin{displaymath}
\widetilde{\Pi}^{\mathrm{int},r_1}_{R,\mathbb{Q}_p\{T_1,...,T_d\}}\bigcap	\widetilde{\Pi}^{[r_1,r_2]}_{R,\mathbb{Q}_p\{T_1,...,T_d\}}=\widetilde{\Pi}^{\mathrm{int},r_2}_{R,\mathbb{Q}_p\{T_1,...,T_d\}}.
\end{displaymath}	
\end{proposition}

\begin{proof}
We adapt the argument in \cite[Lemma 5.2.6]{KL15} 5.2 to prove this in the situation where $r_1<r_2$ (otherwise this is trivial), again one direction is easy where we only present the implication in the other direction. We take any element $x\in \widetilde{\Pi}^{\mathrm{int},r_1}_{R,\mathbb{Q}_p\{T_1,...,T_d\}}\bigcap	\widetilde{\Pi}^{[r_1,r_2]}_{R,\mathbb{Q}_p\{T_1,...,T_d\}}$ and take suitable approximating elements $\{x_i\}$ living in the bounded Robba ring such that for any $j\geq 1$ one can find some integer $N_j\geq 1$ we have whenever $i\geq N_j$ we have the following estimate:
\begin{displaymath}
\left\|.\right\|_{\alpha^{t},\mathbb{Q}_p\{T_1,...,T_d\}}(x_i-x) \leq p^{-j}, \forall t\in [r_1,r_2].	
\end{displaymath}
Then we consider the corresponding decomposition of $x_i$ for each $i=1,2,...$ into a form having integral part and the rational part $x_i=y_i+z_i$ by setting
\begin{center}
 $y_i=\sum_{k=0,i_1,...,i_d}\pi^k[\overline{x}_{i,k,i_1,...,i_d}]T_1^{i_1}...T_d^{i_d}$ 
\end{center} 
out of
\begin{center} 
$x_i=\sum_{k=n(x_i),i_1,...,i_d}\pi^k[\overline{x}_{i,k,i_1,...,i_d}]T_1^{i_1}...T_d^{i_d}$.
\end{center}
Note that by our initial hypothesis we have that the element $x$ lives in the ring $\widetilde{\Pi}^{\mathrm{int},r_1}_{R,\mathbb{Q}_p\{T_1,...,T_d\}}$ which further implies that 
\begin{displaymath}
\left\|.\right\|_{\alpha^{r_1},\mathbb{Q}_p\{T_1,...,T_d\}}(\pi^k[\overline{x}_{i,k,i_1,...,i_d}]T_1^{i_1}...T^{i_d}_d)	\leq p^{-j}.
\end{displaymath}
Therefore we have $\alpha(\overline{x}_{i,k,i_1,...,i_d})\leq p^{(k-j)/r_1}$ directly from this through computation, which implies that then:
\begin{align}
\left\|.\right\|_{\alpha^{r_2},\mathbb{Q}_p\{T_1,...,T_d\}}(\pi^k[\overline{x}_{i,k,i_1,...,i_d}]T_1^{i_1}...T^{i_d}_d)	&\leq p^{-k}p^{(k-j)r_2/r_1}\\
	&\leq p^{1+(1-j)r_1/r_1}.
\end{align}
Then one can read off the result directly from this estimate since under this estimate we can have the chance to modify the original approximating sequence $\{x_i\}$ by $\{y_i\}$ which are initially chosen to be in the integral Robba ring, which implies that actually the element $x$ lives in the right-hand side of the identity in the statement of the proposition.
\end{proof}

\begin{proposition} \mbox{\bf{(After Kedlaya-Liu \cite[Lemma 5.2.6]{KL15})}}
For any $0< r_1\leq r_2$ we have the following equality on the corresponding period rings:
\begin{displaymath}
\widetilde{\Pi}^{\mathrm{int},r_1}_{R,A}\bigcap	\widetilde{\Pi}^{[r_1,r_2]}_{R,A}=\widetilde{\Pi}^{\mathrm{int},r_2}_{R,A}.
\end{displaymath}	
Here $A$ is some reduced affinoid algebra over $\mathbb{Q}_p$.	
\end{proposition}

\begin{proof}
See the proof of \cref{proposition5.7}.	
\end{proof}


\indent Then we have the following analog of the corresponding result of \cite[Lemma 5.2.8]{KL15}:

\begin{proposition} \mbox{\bf{(After Kedlaya-Liu \cite[Lemma 5.2.8]{KL15})}}
Consider now in our situation the radii $0< r_1\leq r_2$, and consider any element $x\in \widetilde{\Pi}^{[r_1,r_2]}_{R,\mathbb{Q}_p\{T_1,...,T_d\}}$. Then we have that for each $n\geq 1$ one can decompose $x$ into the form of $x=y+z$ such that $y\in \pi^n\widetilde{\Pi}^{\mathrm{int},r_2}_{R,\mathbb{Q}_p\{T_1,...,T_d\}}$ with $z\in \bigcap_{r\geq r_2}\widetilde{\Pi}^{[r_1,r]}_{R,\mathbb{Q}_p\{T_1,...,T_d\}}$ with the following estimate for each $r\geq r_2$:
\begin{displaymath}
\left\|.\right\|_{\alpha^r,\mathbb{Q}_p\{T_1,...,T_d\}}(z)\leq p^{(1-n)(1-r/r_2)}\left\|.\right\|_{\alpha^{r_2},\mathbb{Q}_p\{T_1,...,T_d\}}(z)^{r/r_2}.	
\end{displaymath}

\end{proposition}

\begin{proof}
As in \cite[Lemma 5.2.8]{KL15} and in the proof of our previous proposition we first consider those elements $x$ living in the bounded Robba rings which could be expressed in general as
\begin{center}
 $\sum_{k=n(x),i_1,...,i_d}\pi^k[\overline{x}_{k,i_1,...,i_d}]T_1^{i_1}...T_d^{i_d}$.
 \end{center}	
In this situation the corresponding decomposition is very easy to come up with, namely we consider the corresponding $y_i$ as the corresponding series:
\begin{displaymath}
\sum_{k\geq n,i_1,...,i_d}\pi^k[\overline{x}_{k,i_1,...,i_d}]T_1^{i_1}...T_d^{i_d}	
\end{displaymath}
which give us the desired result since we have in this situation when focusing on each single term:
\begin{align}
\left\|.\right\|_{\alpha^r,\mathbb{Q}_p\{T_1,...,T_d\}}(\pi^k[\overline{x}_{k,i_1,...,i_d}]T_1^{i_1}...T_d^{i_d})&=p^{-k}\alpha(\overline{x}_{k,i_1,...,i_d})^r\\
&=p^{-k(1-r/r_2)}\left\|.\right\|_{\alpha^{r_2},\mathbb{Q}_p\{T_1,...,T_d\}}(\pi^k[\overline{x}_{k,i_1,...,i_d}]T_1^{i_1}...T_d^{i_d})^{r/r_2}\\
&\leq p^{(1-n)(1-r/r_2)}\left\|.\right\|_{\alpha^{r_2},\mathbb{Q}_p\{T_1,...,T_d\}}(\pi^k[\overline{x}_{k,i_1,...,i_d}]T_1^{i_1}...T_d^{i_d})^{r/r_2}
\end{align}
for all those suitable $k$. Then to tackle the more general situation we consider the approximating sequence consisting of all the elements in the bounded Robba ring as in the usual situation considered in \cite[Lemma 5.2.8]{KL15}, namely we inductively construct the following approximating sequence just as:
\begin{align}
\left\|.\right\|_{\alpha^r,\mathbb{Q}_p\{T_1,...,T_d\}}(x-x_0-...-x_i)\leq p^{-i-1}	\left\|.\right\|_{\alpha^r,\mathbb{Q}_p\{T_1,...,T_d\}}(x), i=0,1,..., r\in [r_1,r_2].
\end{align}
Here all the elements $x_i$ for each $i=0,1,...$ are living in the bounded Robba ring, which immediately gives rise to the suitable decomposition as proved in the previous case namely we have for each $i$ the decomposition $x_i=y_i+z_i$ with the desired conditions as mentioned in the statement of the proposition. We first take the series summing all the elements $y_i$ up for all $i=0,1,...$, which first of all converges under the norm $\left\|.\right\|_{\alpha^r,\mathbb{Q}_p\{T_1,...,T_d\}}$ for all the radius $r\in [r_1,r_2]$, and also note that all the elements $y_i$ within the infinite sum live inside the corresponding integral Robba ring $\widetilde{\Pi}^{\mathrm{int},r_2}_{R,\mathbb{Q}_p\{T_1,...,T_d\}}$, which further implies the corresponding convergence ends up in $\widetilde{\Pi}^{\mathrm{int},r_2}_{R,\mathbb{Q}_p\{T_1,...,T_d\}}$. For the elements $z_i$ where $i=0,1,...$ also sum up to a converging series in the desired ring since combining all the estimates above we have:
\begin{displaymath}
\left\|.\right\|_{\alpha^r,\mathbb{Q}_p\{T_1,...,T_d\}}(z_i)\leq p^{(1-n)(1-r/r_2)}\left\|.\right\|_{\alpha^{r_2},\mathbb{Q}_p\{T_1,...,T_d\}}(x)^{r/r_2}.	
\end{displaymath}
\end{proof}

%

\begin{proposition} \mbox{\bf{(After Kedlaya-Liu \cite[Lemma 5.2.10]{KL15})}}
We have the following identity:
\begin{displaymath}
\widetilde{\Pi}^{[s_1,r_1]}_{R,\mathbb{Q}_p\{T_1,...,T_d\}}\bigcap\widetilde{\Pi}^{[s_2,r_2]}_{R,\mathbb{Q}_p\{T_1,...,T_d\}}=\widetilde{\Pi}^{[s_1,r_2]}_{R,\mathbb{Q}_p\{T_1,...,T_d\}},
\end{displaymath}
here the radii satisfy $<s_1\leq s_2 \leq r_1 \leq r_2$.
\end{proposition}

\begin{proof}
In our situation one direction is obvious while on the other hand we consider any element $x$ in the intersection on the left, then by the previous proposition we	have the decomposition $x=y+z$ where $y\in \widetilde{\Pi}^{\mathrm{int},r_1}_{R,\mathbb{Q}_p\{T_1,...,T_d\}}$ and $z\in \widetilde{\Pi}^{[s_1,r_2]}_{R,\mathbb{Q}_p\{T_1,...,T_d\}}$. Then as in \cite[Lemma 5.2.10]{KL15} section 5.2 we look at $y=x-z$ which lives in the intersection:
\begin{displaymath}
\widetilde{\Pi}^{\mathrm{int},r_1}_{R,\mathbb{Q}_p\{T_1,...,T_d\}}\bigcap	\widetilde{\Pi}^{[s_2,r_2]}_{R,\mathbb{Q}_p\{T_1,...,T_d\}}=\widetilde{\Pi}^{\mathrm{int},r_2}_{R,\mathbb{Q}_p\{T_1,...,T_d\}}
\end{displaymath}
which finishes the proof.
\end{proof}

\begin{proposition} \mbox{\bf{(After Kedlaya-Liu \cite[Lemma 5.2.10]{KL15})}}
We have the following identity:
\begin{displaymath}
\widetilde{\Pi}^{[s_1,r_1]}_{R,A}\bigcap\widetilde{\Pi}^{[s_2,r_2]}_{R,A}=\widetilde{\Pi}^{[s_1,r_2]}_{R,A},
\end{displaymath}
here the radii satisfy $<s_1\leq s_2 \leq r_1 \leq r_2$.
	
\end{proposition}

\begin{proof}
See the proof of \cref{proposition5.7}.	
\end{proof}

\begin{remark}
This is subsection is finished so far only for the situation where $E$ is of mixed characteristic. But everything uniformly carries over for our original assumption on the field $E$ and $A$. We will not repeat the proof again.	
\end{remark}

\newpage

\section{Period Modules and Period Sheaves}

\indent We now consider the corresponding Frobenius modules over the corresponding period rings defined in the previous section. Also one can consider the corresponding period sheaves in the style of \cite{KL16}. We would like to point out that actually the corresponding sheaves in our context could mean the following two different objects. First the corresponding period rings defined in the previous section themselves are sheafy, which means that one can study the corresponding analytic geometry over the relative affinoid spaces over for instance the ring $\widetilde{\Pi}_{R}^{\mathrm{int},r}$ or the ring $\widetilde{\Pi}_{R}^{\mathrm{int},r}$ which has its own interests. On the other hand we have the situation where one can glue along the algebra $R$ over corresponding \'etale, corresponding pro-\'etale and corresponding v-sites but leaving the algebra $A$ unglued. We point out that both could have the chance to be compared in a coherent way, to the corresponding representation theoretic consideration in the pseudocoherent setting.

\begin{setting}
We will work in the categories of the pseudocoherent, fpd and finite projective modules over the period rings defined above. First we specify the Frobenius in our setting. The rings involved are:
\begin{align}
\widetilde{\Omega}^\mathrm{int}_{R,A},\widetilde{\Omega}_{R,A}, \widetilde{\Pi}^\mathrm{int}_{R,A}, \widetilde{\Pi}^{\mathrm{int},r}_{R,A},\widetilde{\Pi}^\mathrm{bd}_{R,A},\widetilde{\Pi}^{\mathrm{bd},r}_{R,A}, \widetilde{\Pi}_{R,A}, \widetilde{\Pi}^r_{R,A},\widetilde{\Pi}^+_{R,A}, \widetilde{\Pi}^\infty_{R,A},\widetilde{\Pi}^I_{R,A}.
\end{align}
We are going to endow these rings with the Frobenius induced by continuation from the Witt vector part only, which is to say the corresponding Frobenius induced by the $p^h$-power absolute Frobenius over $R$. Note all the rings above are defined by taking the product of $\triangle$ where each $\triangle$ representing one of the following rings (over $E$):
\begin{align}
\widetilde{\Omega}^\mathrm{int}_{R},\widetilde{\Omega}_{R}, \widetilde{\Pi}^\mathrm{int}_{R}, \widetilde{\Pi}^{\mathrm{int},r}_{R},\widetilde{\Pi}^\mathrm{bd}_{R},\widetilde{\Pi}^{\mathrm{bd},r}_{R}, \widetilde{\Pi}_{R}, \widetilde{\Pi}^r_{R},\widetilde{\Pi}^+_{R}, \widetilde{\Pi}^\infty_{R},\widetilde{\Pi}^I_{R}
\end{align}
with the affinoid ring $A$. The Frobenius acts on $A$ trivially and we assume that the action is $A$-linear.  
\end{setting}

\indent First we consider the following sheafification as in \cite[Definition 4.4.2]{KL16}:

\begin{setting} 
Consider the space $X=\mathrm{Spa}(R,R^+)$, over this perfectoid space there were sheaves:
\begin{align}
\widetilde{\Omega}^\mathrm{int}_{},\widetilde{\Omega}_{}, \widetilde{\Pi}^\mathrm{int}_{}, \widetilde{\Pi}^{\mathrm{int},r}_{},\widetilde{\Pi}^\mathrm{bd}_{},\widetilde{\Pi}^{\mathrm{bd},r}_{}, \widetilde{\Pi}_{}, \widetilde{\Pi}^r_{},\widetilde{\Pi}^+_{}, \widetilde{\Pi}^\infty_{},\widetilde{\Pi}^I_{}.
\end{align}
defined over this space through the corresponding adic, \'etale, pro-\'etale or $v$-topology, we consider the corresponding sheaves defined over the same Grothendieck sites but with further deformed consideration:
\begin{align}
\widetilde{\Omega}^\mathrm{int}_{*,A},\widetilde{\Omega}_{*,A}, \widetilde{\Pi}^\mathrm{int}_{*,A}, \widetilde{\Pi}^{\mathrm{int},r}_{*,A},\widetilde{\Pi}^\mathrm{bd}_{*,A},\widetilde{\Pi}^{\mathrm{bd},r}_{*,A}, \widetilde{\Pi}_{*,A}, \widetilde{\Pi}^r_{*,A},\widetilde{\Pi}^+_{*,A}, \widetilde{\Pi}^\infty_{*,A},\widetilde{\Pi}^I_{*,A}.
\end{align}
\end{setting}

\begin{remark}
The consideration in the previous setting in some sense reflects some kind of rigidity. Since the corresponding construction is just that the ring $A$ acts through the completed tensor product directly on the sheaves after \cite{KL16}, in most cases when one would like to compare those sheaves of modules over the sheaves of rings above and the corresponding modules over the period rings defined above, one needs to work through the corresponding vanishing results and so on in this deformed context. We work around this by considering the corresponding globalized version of this. 	
\end{remark}

\begin{setting}
In our context we can make the following parallel discussion to that in \cite{KL16}. First we can use the corresponding notation $E_\infty$ to denote the corresponding completion of $E(\pi^{1/{p^\infty}})$ and we denote the corresponding integral ring by $\mathcal{O}_{E_\infty}$. Then we have could have that the corresponding splitting of the ring $\mathcal{O}_E$ in this bigger perfectoid ring. Then we can consider the corresponding extended period rings $\widetilde{\Omega}^\mathrm{int}_{R}\widehat{\otimes}_{\mathcal{O}_E}\mathcal{O}_{E_\infty}$ and $\widetilde{\Pi}^{\mathrm{int},r}_{R}\widehat{\otimes}_{\mathcal{O}_E}\mathcal{O}_{E_\infty}$	in the same fashion as in \cite[Definition 4.1.11]{KL16}.
\end{setting}

\begin{proposition}\mbox{\bf{(After Kedlaya-Liu \cite[Lemma 4.1.12, Corollary 4.1.14]{KL16})}}
The ring $\widetilde{\Pi}^{\mathrm{int},r}_{R}\widehat{\otimes}_{\mathcal{O}_E}\mathcal{O}_{E_\infty}$ is perfectoid, so stably uniform and sheafy. By using the same notation in \cite[Lemma 4.1.12, Corollary 4.1.14]{KL16} we have that the corresponding tilting perfectoid ring is $R\{(\overline{\pi}/p^{-1/r})^{1/p^\infty}\}$ which is the completion of $R[(\overline{\pi}/p^{-1/r})^{1/p^\infty}]$.	
\end{proposition}

\begin{proof}
This is the same as the proof for \cite[Lemma 4.1.12, Corollary 4.1.14]{KL16} after we choose a suitable topologically nilpotent element.
\end{proof}

\begin{proposition}\mbox{\bf{(After Kedlaya-Liu \cite[Proposition 4.1.13, Corollary 4.1.14]{KL16})}}
The ring $\widetilde{\Pi}_R^I\widehat{\otimes}_{\mathcal{O}_E}\mathcal{O}_{E_\infty}$ for some closed interval $I=[s,r]$ is perfectoid so stably uniform and sheafy. Then in our situation this perfectoid admits tilting ring $R\{(\overline{\pi}/p^{-1/r})^{1/p^\infty},(p^{-1/s}/\overline{\pi})^{1/p^\infty}\}$ which is the completion of $R[(\overline{\pi}/p^{-1/r})^{1/p^\infty},(p^{-1/s}/\overline{\pi})^{1/p^\infty}]$.	
\end{proposition}

\begin{proof}
This is the same as the proof of \cite[Proposition 4.1.13, Corollary 4.1.14]{KL16} by considering the previous proposition.	
\end{proof}

\begin{setting}
We are going to use the notation $\mathfrak{X}$ to denote a rigid analytic space in rigid geometry. Then we can consider the corresponding relative period rings or sheaves over $\mathcal{O}_{\mathfrak{X}}$. The period rings and the corresponding period sheaves over the sheaf $\mathcal{O}_{\mathfrak{X}}$ are defined to be the following sheaves (one takes the complete tensor product of the undeformed period rings with the corresponding exact sequence for the sheaf $?$, which again gives the corresponding exact sequence due to the fact that the undeformed period rings admit Schauder bases as in \cite[Definition 6.1]{KP}):
\begin{align}
\widetilde{\Pi}_{R,?}, \widetilde{\Pi}^r_{R,?},\widetilde{\Pi}^\infty_{R,?},\widetilde{\Pi}^I_{R,?},
\end{align}
with
\begin{align}
\widetilde{\Pi}_{*,?}, \widetilde{\Pi}^r_{*,?},\widetilde{\Pi}^\infty_{*,?},\widetilde{\Pi}^I_{*,?},
\end{align}
where $*$ is some \'etale site, pro\'et site or $v$-site and etc, and $?$ represents $\mathcal{O}_{\mathfrak{X}}$ which implies that we treat these sheaves of rings as sheaves over $*$ in some relative sense. One should understand this as the corresponding sheaves essentially with respect $*$.
\end{setting}


\begin{remark}
It is hard to consider the corresponding comparison between the representation theories over these sheaves (deformed) with the rings (without any sheafified consideration), since the corresponding global section functor is a little bit hard to study. But on the other hand comparing this among the sheaves themselves and with the Fargues-Fontaine curves are very natural and interesting.
\end{remark}

\begin{assumption}
From now on, we are going to assume that the affinoid ring $A$ is splitting in some perfectoid covering $A^\mathrm{perf}$, but only when the corresponding sheafiness of the deformed period rings is essentially relevant. For instance if $A$ is just some analytic field as those in \cite{KL16} then this is satisfied. This will guarantee the corresponding period rings over $A$ (we now assume them to be sousperfectoid) will then be sheafy by Kedlaya-Hansen's sheafiness criterion \cite{KH}.	
\end{assumption}

\indent We first consider the following result which is slight generalization of the corresponding result from \cite[4.3.1]{KL16}. Note that our $E$ is general than the basic setting of \cite[Setting 3.1.1]{KL16}.

\begin{remark}
In what follows, there will be some situation we will consider the perfectoid deformed version of the period rings and sheaves. All the definitions are parallel to the base level definitions (including the corresponding Frobenius modules and sheaves).	
\end{remark}

\begin{lemma} \mbox{\bf{(After Kedlaya-Liu \cite[4.3.1]{KL16})}} \label{lemma3.11}
Let $M$ be any \'etale-stably pseudocoherent or fpd module defined over $\widetilde{\Pi}^{\mathrm{int},r}_{R}$ and $\widetilde{\Pi}^{\mathrm{int},r}_{R,\infty}$ for some $r>0$ or $\widetilde{\Pi}^I_{R}$ and $\widetilde{\Pi}^I_{R,\infty}$ for some $I$ which is assumed to be closed, and let $\widetilde{M}$ be the corresponding sheaf attached to the module $M$ over the sheaf $\widetilde{\Pi}^{\mathrm{int},r}_{*}$ for some $r>0$ or $\widetilde{\Pi}^I_{*}$ for some $I$ which is assumed to be closed, over the \'etale and the pro-\'etale sites of $\mathrm{Spa}(R,R^+)$. Then we have the following statement:
\begin{align}
H^i(X,\widetilde{M})=H^i(X_\text{\'et},\widetilde{M})=H^i(X_\text{pro\'et},\widetilde{M})	
\end{align}
vanish beyond the $0$-th degree and give the module $M$ at the degree $0$. Also furthermore we have that in the projective situation we have that the same holds under the $v$-topology. 
\end{lemma}

\begin{proof}
The corresponding results could be transformed to perfectoid spaces as in \cite[4.3.1]{KL16} by considering the splitting after passing to perfectoid field $E_\infty$ and perfectoid ring $\mathcal{O}_{\mathfrak{X}^\mathrm{perf}}$, then by applying the results in \cite[2.5.1,2.5.11,3.4.6,3.5.6]{KL16} we can get the desired results.	
\end{proof}

\begin{definition}
Now working over some perfect uniform and adic space over $k$, as in \cite[Definition 4.3.2]{KL16} we define the corresponding pseudocoherent sheaves or fpd sheaves over $\widetilde{\Pi}^{\mathrm{int},r}_*$ or $\widetilde{\Pi}^{I}_*$ (where $I$ is closed) where $*$ represents one of $X$, $X_\text{\'et}$ and $X_{\text{pro\'et}}$ to be the sheaves associated to \'etale-stably pseudocoherent or fpd modules in some local manner. 	
\end{definition}

\begin{proposition} \mbox{\bf{(After Kedlaya-Liu \cite[Theorem 4.3.3]{KL16})}} \label{proposition3.13}
The global section functor establishes an equivalence between the  categories of the pseudocoherent sheaves over $\widetilde{\Pi}^{\mathrm{int},r}_*$ or $\widetilde{\Pi}^{I}_*$ (and $\widetilde{\Pi}^{\mathrm{int},r}_{*,\infty}$ or $\widetilde{\Pi}^{I}_{*,\infty}$) over the pro-\'etale site of $X=\mathrm{Spa}(R,R^+)$ and the \'etale-stably pseudocoherent modules over $\widetilde{\Pi}^{\mathrm{int},r}_R$ or $\widetilde{\Pi}^{I}_R$ (and $\widetilde{\Pi}^{\mathrm{int},r}_{R,\infty}$ or $\widetilde{\Pi}^{I}_{R,\infty}$). The global section functor establishes an equivalence between the  categories of the fpd sheaves over $\widetilde{\Pi}^{\mathrm{int},r}_*$ or $\widetilde{\Pi}^{I}_*$ over the pro-\'etale site of $X=\mathrm{Spa}(R,R^+)$ and the \'etale-stably fpd modules over $\widetilde{\Pi}^{\mathrm{int},r}_R$ or $\widetilde{\Pi}^{I}_R$. The following morphisms are effective descent morphisms for pseudocoherent Banach modules: I. $\widetilde{\Pi}_{R,A}^r\rightarrow \widetilde{\Pi}_{R,A}^s\oplus \widetilde{\Pi}_{R,A}^{[s,r]}$; II. $\widetilde{\Pi}_{R,A}^I\rightarrow \oplus_{i=1}^k \widetilde{\Pi}_{R,A}^{I_i}$. here the corresponding set of interval $\{I_i\}_{i=1,...,k}$ consists of finitely many closed intervals covering the interval $I$.
 
\end{proposition}

\begin{proof}
As in \cite[Theorem 4.3.3]{KL16}, by taking the splitting base change to $\mathcal{O}_{E_\infty}$	and considering the corresponding properties of being perfectoid we can transform the proofs around the global section functor to the corresponding statement for juts perfectoid spaces, which could be finished by \cite[2.5.5,2.5.14,3.4.9]{KL16}. The rest statements are further consequences of the fact that the period rings $\widetilde{\Pi}_{R,A}^r$ and $\widetilde{\Pi}_{R,A}^I$ are sheafy, see \cite[2.5.5,2.5.14]{KL16}.
\end{proof}

\indent Then we generalize the corresponding definitions of Frobenius action and Frobenius modules in \cite[Definition 4.4.2-4.4.4]{KL16} to our situation.

\begin{definition}
In our situation we consider the corresponding Frobenius action on the following period rings and sheaves:
\begin{align}
\widetilde{\Omega}^\mathrm{int}_{R,A},\widetilde{\Omega}^\mathrm{int}_{R,A},\widetilde{\Pi}^\mathrm{int}_{R,A},\widetilde{\Pi}^\mathrm{bd}_{R,A},\widetilde{\Pi}_{R,A},\widetilde{\Pi}^\infty_{R,A}, \widetilde{\Omega}^\mathrm{int}_{*,A}, \widetilde{\Pi}^\mathrm{int}_{*,A}, \widetilde{\Pi}^+_{*,A}, \widetilde{\Omega}_{*,A}, \widetilde{\Pi}^\mathrm{bd}_{*,A}, \widetilde{\Pi}^\infty_{*,A}, \widetilde{\Pi}^+_{*,A}, \widetilde{\Pi}_{*,A},   
\end{align}
which is defined by considering the corresponding lift of the absolute Frobenius of $p^h$-power in characteristic $p>0$ induced from $R$, which will be denoted by $\varphi$. We then introduce more general consideration by taking $E_a$ to be some unramified extension of $E$ of degree $a$ divisible by $h$, the corresponding Frobenius will be denoted by $\varphi^a$.
\end{definition}

	
\indent Then we generalize the corresponding Frobenius modules in \cite[Definition 4.4.4]{KL16} to our situation as in the following.

\begin{definition}
Over the period rings and sheaves (each is denoted by $\triangle$ in this definition) defined in the previous definition we define as in \cite[Definition 4.4.4]{KL16} the corresponding $\varphi^a$-modules over $\triangle$ which are respectively projective, pseudocoherent or fpd to be the corresponding finite projetive, pseudocoherent or fpd modules over $\triangle$ with further assigned semilinear action of the operator $\varphi^a$. Here we define in our situation the corresponding $\varphi^a$-cohomology to be the (hyper)-cohomology of the following complex:
\[
\xymatrix@R+0pc@C+0pc{
0\ar[r]\ar[r]\ar[r] &M
\ar[r]^{\varphi-1}\ar[r]\ar[r] &M
\ar[r]\ar[r]\ar[r] &0.
}
\]
We also require that the modules are complete for the natural topology involved in our situation and for any module over $\widetilde{\Pi}_{*,A}$ to be some base change of some module over $\widetilde{\Pi}^r_{*,A}$ (which will be defined in the following) over each perfectoid subdomain $Y$ (in this situation we are considering the pro-\'etale site). 
\end{definition}

\indent Now we define the corresponding modules over the rings which are the domains in the following morphisms induced from the Frobenius map $\varphi^a$:
\begin{align} 
\widetilde{\Pi}^{\mathrm{int},r}_{R,A}\rightarrow \widetilde{\Pi}^{\mathrm{int},rp^{-ha}}_{R,A},\widetilde{\Pi}^{\mathrm{bd},r}_{R,A}\rightarrow \widetilde{\Pi}^{\mathrm{bd},rp^{-ha}}_{R,A},\widetilde{\Pi}^{r}_{R,A}\rightarrow \widetilde{\Pi}^{rp^{-ha}}_{R,A}\\
\widetilde{\Pi}^{\mathrm{int},r}_{*,A}\rightarrow \widetilde{\Pi}^{\mathrm{int},rp^{-ha}}_{*,A},\widetilde{\Pi}^{\mathrm{bd},r}_{*,A}\rightarrow \widetilde{\Pi}^{\mathrm{bd},rp^{-ha}}_{*,A},\widetilde{\Pi}^{r}_{*,A}\rightarrow \widetilde{\Pi}^{rp^{-ha}}_{*,A}.	
\end{align}

\begin{definition}
Over each rings $\triangle$ which are the domains in the morphisms as mentioned just before this definition, we define the corresponding projective, pseudocoherent or fpd $\varphi^a$-module over any $\triangle$ listed above to be the corresponding finite projective, pseudocoherent or fpd module $M$ over $\triangle$ with additionally endowed semilinear Frobenius action from $\varphi^a$ such that we have the isomorphism $\varphi^{a*}M\overset{\sim}{\rightarrow}M\otimes \square$ where the ring $\square$ is one of the targets listed above. Also as in \cite[Definition 4.4.4]{KL16} we assume that the module over $\widetilde{\Pi}^I_{*,A}$ is then complete for the natural topology and the corresponding base change to $\widetilde{\Pi}^I_{*,A}$ for any interval which is assumed to be closed $I\subset [0,r)$ gives rise to a module over $\widetilde{\Pi}^I_{*,A}$ with specified conditions which will be specified below. Also the cohomology of any module under this definition will be defined to be the (hyper)cohomology of the complex in the following form:
\[
\xymatrix@R+0pc@C+0pc{
0\ar[r]\ar[r]\ar[r] &M
\ar[r]^{\varphi-1}\ar[r]\ar[r] &M\otimes_\triangle \square
\ar[r]\ar[r]\ar[r] &0.
}
\]
\end{definition}

\indent Then we consider the following morphisms of specific period rings induced by the Frobenius. 
\begin{align}
\widetilde{\Pi}_{R,A}^{[s,r]}	\rightarrow \widetilde{\Pi}_{R,A}^{[sp^{-ah},rp^{-ah}]}\\
\widetilde{\Pi}_{*,A}^{[s,r]}	\rightarrow \widetilde{\Pi}_{*,A}^{[sp^{-ah},rp^{-ah}]}
\end{align}

with the corresponding morphisms in the following:

\begin{align}
\widetilde{\Pi}_{R,A}^{[s,r]}	\rightarrow \widetilde{\Pi}_{R,A}^{[s,rp^{-ah}]}\\
\widetilde{\Pi}_{*,A}^{[s,r]}	\rightarrow \widetilde{\Pi}_{*,A}^{[s,rp^{-ah}]}
\end{align}

\begin{definition}
Again as in \cite[Definition 4.4.4]{KL16}, we define the corresponding projective, pseudocoherent and fpd $\varphi^a$-modules over the domain rings or sheaves of rings in the morphisms just before this definition to be the finite projective, pseudocoherent and fpd modules (which will be denoted by $M$) over the domain rings in the morphism just before this definition additionally endowed with semilinear Frobenius action from $\varphi^a$ with the following isomorphisms:
\begin{align}
\varphi^{a*}M\otimes_{\widetilde{\Pi}_{R,A}^{[sp^{-ah},rp^{-ah}]}}\widetilde{\Pi}_{R,A}^{[s,rp^{-ah}]}\overset{\sim}{\rightarrow}M\otimes_{\widetilde{\Pi}_{R,A}^{[s,r]}}\widetilde{\Pi}_{R,A}^{[s,rp^{-ah}]},\\
\varphi^{a*}M\otimes_{\widetilde{\Pi}_{*,A}^{[sp^{-ah},rp^{-ah}]}}\widetilde{\Pi}_{*,A}^{[s,rp^{-ah}]}\overset{\sim}{\rightarrow}M\otimes_{\widetilde{\Pi}_{*,A}^{[s,r]}}\widetilde{\Pi}_{*,A}^{[s,rp^{-ah}]}.
\end{align}
We now assume that the modules are complete with respect to the natural topology. And we assume that for any perfectoid subdomain $Y$ in the corresponding topology (defining the sheaves in this situation) the corresponding global sections over $Y$ give rise to \'etale-stably pseudocoherent modules.
\end{definition}

\noindent Also one can further define the corresponding bundles carrying semilinear Frobenius in our context as in the situation of \cite[Definition 4.4.4]{KL16}:

\begin{definition}
Over the ring $\widetilde{\Pi}_{R,A}$ we define a corresponding projective, pseudocoherent and fpd Frobenius bundle to be a family $(M_I)_I$ of finite projective, \'etale stably pseudocoherent and \'etale stably fpd modules over each $\widetilde{\Pi}^I_{R,A}$ carrying the natural Frobenius action coming from the operator $\varphi^a$ such that for any two involved intervals having the relation $I\subset J$ we have:
\begin{displaymath}
M_J\otimes_{\widetilde{\Pi}^J_{R,A}}\widetilde{\Pi}^I_{R,A}\overset{\sim}{\rightarrow}	M_I
\end{displaymath}
with the obvious cocycle condition. Here we have to propose condition on the intervals that for each $I=[s,r]$ involved we have $s\leq rp^{ah}$.
\end{definition}

\indent We then have the following analog of \cite[Lemma 4.4.8]{KL16}:

\begin{proposition} \mbox{\bf{(After Kedlaya \cite[Lemma 4.4.8]{KL15})}} \label{prop3.18}
Consider the corresponding finite generated Frobenius modules over the following period rings or sheaves defined above:
\begin{displaymath}
\widetilde{\Omega}^\mathrm{int}_{R,A},\widetilde{\Omega}^\mathrm{int}_{R,A},\widetilde{\Pi}^\mathrm{int}_{R,A},\widetilde{\Pi}^\mathrm{bd}_{R,A},\widetilde{\Pi}_{R,A},\widetilde{\Pi}^\infty_{R,A}, \widetilde{\Omega}^\mathrm{int}_{*,A}, \widetilde{\Pi}^\mathrm{int}_{*,A}, \widetilde{\Pi}^+_{*,A}, \widetilde{\Omega}_{*,A}, \widetilde{\Pi}^\mathrm{bd}_{*,A}, \widetilde{\Pi}^\infty_{*,A}, \widetilde{\Pi}^+_{*,A}, \widetilde{\Pi}_{*,A},   
\end{displaymath}
and
\begin{align} 
\widetilde{\Pi}^{\mathrm{int},r}_{R,A}\rightarrow \widetilde{\Pi}^{\mathrm{int},rp^{-ha}}_{R,A},\widetilde{\Pi}^{\mathrm{bd},r}_{R,A}\rightarrow \widetilde{\Pi}^{\mathrm{bd},rp^{-ha}}_{R,A},\widetilde{\Pi}^{r}_{R,A}\rightarrow \widetilde{\Pi}^{rp^{-ha}}_{R,A}\\
\widetilde{\Pi}^{\mathrm{int},r}_{*,A}\rightarrow \widetilde{\Pi}^{\mathrm{int},rp^{-ha}}_{*,A},\widetilde{\Pi}^{\mathrm{bd},r}_{*,A}\rightarrow \widetilde{\Pi}^{\mathrm{bd},rp^{-ha}}_{*,A},\widetilde{\Pi}^{r}_{*,A}\rightarrow \widetilde{\Pi}^{rp^{-ha}}_{*,A}.	
\end{align} 
Then we have these are the quotients of finite projective ones again endowed with the corresponding Frobenius actions.
\end{proposition}

\begin{proof}
See 1.5.2 of \cite[Lemma 1.5.2]{KL15}.	
\end{proof}

\newpage

\section{Comparison Theorems}

\subsection{The Comparison between the Local Systems and the Period Modules}

\indent We now in our context study the corresponding local systems in the generalized setting. The objects to compare will be definitely the Frobenius modules defined in the previous section. However this is not that far from the situation studied in \cite{KL16} since the undeformed period rings are actually finite projective over the rings defined in \cite{KL16} in the most simplified situation. We will also use the notation $\underline{L}$ to denote the local system associated to some topological ring $L$. We will consider the pro-\'etale site $X_\text{pro\'et}$. As in \cite[Definition 4.5.1]{KL16} we have the notion of projective, pseudocoherent and fpd $L$-local systems in our context over the site $X_\text{pro\'et}$.

\begin{proposition} \mbox{\bf{(After Kedlaya-Liu \cite[Lemma 4.5.4]{KL16})}} \label{proposition4.1}
The Frobenius invariances give rise to the following exact sequences:
\[
\xymatrix@R+0pc@C+0pc{
0\ar[r]\ar[r]\ar[r] &\underline{\mathcal{O}_{E_a^{\varphi^a}}}
\ar[r]\ar[r]\ar[r] &\widetilde{\Omega}_{*}^\mathrm{int}
\ar[r]\ar[r]\ar[r] &\widetilde{\Omega}_{*}^\mathrm{int} \ar[r]\ar[r]\ar[r] &0,
}
\]
\[
\xymatrix@R+0pc@C+0pc{
0\ar[r]\ar[r]\ar[r] &\underline{\mathcal{O}_{E_a^{\varphi^a}}}
\ar[r]\ar[r]\ar[r] &\widetilde{\Pi}_{*}^\mathrm{int}
\ar[r]\ar[r]\ar[r] &\widetilde{\Pi}_{*}^\mathrm{int} \ar[r]\ar[r]\ar[r] &0,
}
\]
\[
\xymatrix@R+0pc@C+0pc{
0\ar[r]\ar[r]\ar[r] &\underline{E_a}^{\varphi^a}
\ar[r]\ar[r]\ar[r] &\widetilde{\Pi}_{*}^\mathrm{bd}
\ar[r]\ar[r]\ar[r] &\widetilde{\Pi}_{*}^\mathrm{bd} \ar[r]\ar[r]\ar[r] &0,
}
\]
\[
\xymatrix@R+0pc@C+0pc{
0\ar[r]\ar[r]\ar[r] &\underline{E_a}^{\varphi^a}
\ar[r]\ar[r]\ar[r] &\widetilde{\Pi}_{*}
\ar[r]\ar[r]\ar[r] &\widetilde{\Pi}_{*} \ar[r]\ar[r]\ar[r] &0,
}
\]
where in each exact sequence the third arrow represents the morphism $\varphi^a-1$.
	
\end{proposition}

\begin{proof}
It is actually \cite[Lemma 4.5.4]{KL16}, although we consider some bigger base $k$, we reproduce the argument for the convenience of the reader. So it is obviously we have the exactness at the first positions which are not zero. The exactness in the middle could actually be derived from \cite[Lemma 5.2.4]{KL15}. Then for the exactness in between the third and the last arrows one may follow the proof in \cite[Lemma 4.5.3]{KL16} to prove this. For the first sequence see \cite[Lemma 4.5.3]{KL16}. For the second sequence, relying on the first sequence one only needs to prove that suppose image of some $y\in \widetilde{\Omega}^\mathrm{int}_R$ under the map $\varphi^a-1$ is $x$ living in $\widetilde{\Pi}^\mathrm{int}_R\subset \widetilde{\Omega}^\mathrm{int}_R$ then we have that the preimage $y$ is not only in $\widetilde{\Omega}^\mathrm{int}_R$ but also actually in $\widetilde{\Pi}^\mathrm{int}_R$. This is because we have for each $n$ and the corresponding expressions $x=\sum_{n\geq 0}\pi^n[\overline{x}_n]$ and $y=\sum_{n\geq 0}\pi^n[\overline{y}_n]$, the coefficients for $y$ could be controlled by those of $x$. The corresponding statement for the last sequence is proved by reducing to the second one, namely put:
\begin{displaymath}
x=y+(\varphi^a-1)(-\sum_{k=0}\sum_{\ell=0}\pi^{i_k}[\overline{x}_{i_k}^{p_\ell}]),	
\end{displaymath}
where one may beforehand put $x$ into a summation of some element in the bounded Robba ring and some element of the form of $\sum_{n_k}\pi^{n_k}[\overline{x}_{i_k}]$ with $\overline{x}_{i_k}$ living in $R^{\circ\circ}$.
\end{proof}

We then make the following discussion as well, as in \cite[Lemma 4.5.4]{KL16} in the following:

\begin{proposition} \mbox{\bf{(After Kedlaya-Liu \cite[Lemma 4.5.4]{KL16})}}
We consider the finitely presented Frobenius modules defined over the ring $\widetilde{\Omega}^\mathrm{int}_{R}$ or $\widetilde{\Pi}^\mathrm{int}_R$. Then in our context, we have: (1) First the Fitting ideal of $M$ is generated by the elements which are in general form of $\pi^n[\overline{x}_n]$ where $\overline{x}_n$ is some idempotent. (2) And we have that the $\pi$-torsion submodule $T$ of $M$ is finitely generated, therefore could be annihilated by some power $\pi^n$ with some integer $n\geq 0$. (3) And the quotient $M/T$ is then projective which implies that the module $T$ is finitely presented. (4) And the submodule $T$ is then finite union of Frobenius modules taking the form of $?/\pi^k?$ for integer $k\geq 0$, where $?$ represents the ring $\widetilde{\Omega}^\mathrm{int}_{R}$ or $\widetilde{\Pi}^\mathrm{int}_R$. (5) Finally we have that the corresponding module $M$ is strictly pseudocoherent. 	
\end{proposition}

\begin{proof}
Again this is basically \cite[Lemma 4.5.4]{KL16}, although we consider a larger base $k$. We reproduce the argument for the convenience of the reader. Then corresponding property of the Fitting ideal follows from the applying \cite[Proposition 3.2.13]{KL15}. Then we apply \cref{prop3.18} to choose a projective covering of $M$ taking the form of $P\rightarrow M$ which further induces a map from $P\rightarrow M/T$, then we use the notation $N$ to denote the kernel. Taking quotient by $\pi$ we can form the exact sequence in our context due to the $\pi$-torsion freeness:
\[
\xymatrix@R+0pc@C+0pc{
0\ar[r]\ar[r]\ar[r] &N/\pi N
\ar[r]\ar[r]\ar[r] &P/\pi P \ar[r]\ar[r]\ar[r] &(M/T)/\pi (M/T) \ar[r]\ar[r]\ar[r] &0.
}
\]
Then as in \cite[Lemma 4.5.4]{KL16} one can then in the situation where the base is $\widetilde{\Omega}_R^\mathrm{int}$ under the $\pi$-adic topology lift the generators of $N/\pi N$ by using the converging series of finite sums of the generators from the quotient. Then in the situation where the base is $\widetilde{\Pi}_R^\mathrm{int}$ we use the Fr\'echet topology to run the same argument. $T$ is a union of each $T_n$ for each $n$, which are the corresponding $\pi^n$-torsion components of $T$. Then by using \cite[Lemma 1.1.5]{KL16} one derives that the module $M/T$ is then finitely presented, which further implies that the corresponding module $T$ is finite, which is just equal to some $T_n$. Then we apply \cite[Lemma 1.1.5]{KL16} to finish the proof.
\end{proof}

\begin{corollary}
Let $M$ be a pseudocoherent Frobenius module over the period ring $\widetilde{\Omega}_R^\mathrm{int}$ or the period ring $\widetilde{\Pi}_R^\mathrm{int}$. Then we have that $M$ admits a projective resolution of length at most $1$.  
\end{corollary}

\indent Then we generalize the comparison of local systems and Frobenius modules in \cite[Theorem 4.5.7]{KL16} to our context as in the following:

\begin{theorem} \mbox{\bf{(After Kedlaya-Liu \cite[Theorem 4.5.7]{KL16})}} \label{theorem4.4}
The following categories are equivalent:\\
I. The category of projective, pseudocoherent or fpd $\mathcal{O}_{E_a^{\varphi^a}}$-local systems over $X_\text{pro\'et}$;\\
II. The category of projective, pseudocoherent or fpd Frobenius modules over the period ring $\widetilde{\Omega}_R^\mathrm{int}$;\\
III. The category of projective, pseudocoherent or fpd Frobenius modules over the period ring $\widetilde{\Pi}_R^\mathrm{int}$;\\
IV.  The category of projective, pseudocoherent or fpd Frobenius modules over the period sheaf $\widetilde{\Omega}_*^\mathrm{int}$ over $X$, $X_\text{\'et}$, $X_\text{pro\'et}$;\\
V.  The category of projective, pseudocoherent or fpd Frobenius modules over the period sheaf $\widetilde{\Pi}_*^\mathrm{int}$ over $X$, $X_\text{\'et}$, $X_\text{pro\'et}$.\\ 
Under further deformation we do not have the corresponding equivalence as above but we have the following statement:\\
1. There is a fully faithful embedding functor from the category of the projective, pseudocoherent or fpd $\mathcal{O}_{E_a^{\varphi^a}}\widehat{\otimes}\mathcal{O}_A$-local systems over $X_\text{pro\'et}$ to the following two categories:\\
1(a).  The category of projective, pseudocoherent or fpd Frobenius modules over the period sheaf $\widetilde{\Omega}_{*,A}^\mathrm{int}$ over $X$, $X_\text{\'et}$, $X_\text{pro\'et}$;\\
1(b).  The category of projective, pseudocoherent or fpd Frobenius modules over the period sheaf $\widetilde{\Pi}_{*,A}^\mathrm{int}$ over $X$, $X_\text{\'et}$, $X_\text{pro\'et}$.\\ 
2. We have that the following two categories are equivalent:\\
2(a). The category of projective, pseudocoherent or fpd Frobenius modules over the period ring $\widetilde{\Omega}_{R,\infty}^\mathrm{int}$;\\ 
2(b). The category of projective, pseudocoherent or fpd Frobenius modules over the period sheaf $\widetilde{\Omega}_{*,\infty}^\mathrm{int}$ over $X$, $X_\text{\'et}$, $X_\text{pro\'et}$.\\
3. Furthermore we have the following two categories are equivalent:\\
3(a). The category of projective, pseudocoherent or fpd Frobenius modules over the period ring $\widetilde{\Pi}_{R,\infty}^\mathrm{int}$;\\
3(b). The category of projective, pseudocoherent or fpd Frobenius modules over the period sheaf $\widetilde{\Pi}_{*,\infty}^\mathrm{int}$ over $X$, $X_\text{\'et}$, $X_\text{pro\'et}$.

\end{theorem}

\begin{proof}
The first main statement, see \cite[Theorem 4.5.7]{KL16}. Then we proceed to consider the rest three main statements along the line of the proof of \cite[Theorem 4.5.7]{KL16}. For the first statement the corresponding functor is just the base change functor:
\begin{displaymath}
\mathbb{L}\mapsto \mathbb{L}\otimes_{\underline{\mathcal{O}_{E_a^{\varphi^a}}\widehat{\otimes}\mathcal{O}_A}}\widetilde{\Omega}_{*,A}^\mathrm{int}	
\end{displaymath}
and 
\begin{displaymath}
\mathbb{L}\mapsto \mathbb{L}\otimes_{\underline{\mathcal{O}_{E_a^{\varphi^a}}\widehat{\otimes}\mathcal{O}_A}}\widetilde{\Pi}_{*,A}^\mathrm{int},	
\end{displaymath}
these are basically fully faithful by applying the corresponding deformed versions (which is just directly taking the complete tensor product with $A$) of the corresponding exact sequences in \cref{proposition4.1}. For the second and the last statements we follow the strategy of the proof of \cite[Theorem 4.5.7]{KL16} to proceed as in the following. First for the second one, this is direct consequence of \cref{lemma3.11} and \cref{proposition3.13}. Then the third statement could be reduced to the statement by applying the second statement to the corresponding completion of $R[T^{\pm p^{-\infty}}]$ or $R[T^{p^{-\infty}}]$. 
\end{proof}

\indent Then as in \cite[Corollary 4.5.8]{KL16} we consider the setting of more general adic spaces:

\begin{proposition} \mbox{\bf{(After Kedlaya-Liu \cite[Corollary 4.5.8]{KL16})}}
1. Let $X$	be a preadic space over $E_a$. Then we have that there is fully faithful embedding functor from the category of all the $\mathcal{O}_{E_{a}^{\varphi^a}}\widehat{\otimes}\mathcal{O}_A$-local systems (projective, pseudocoherent or fpd) over $X$, $X_{\text{\'et}}$ and $X_\text{pro\'et}$ to the category of all the projective, pseudocoherent or fpd Frobenis modules over the sheaves over $\widetilde{\Pi}_{R,A}^\mathrm{int}$ over the corresponding sites;\\
2. Let $X$	be a preadic space over $E_{\infty,a}$. Then we have that there is fully faithful embedding functor from the category of all the $\mathcal{O}_{E_{\infty,a}^{\varphi^a}}\widehat{\otimes}\mathcal{O}_A$-local systems (projective, pseudocoherent or fpd) over $X$, $X_{\text{\'et}}$ and $X_\text{pro\'et}$ to the category of all the projective, pseudocoherent or fpd Frobenius modules over the sheaves over $\widetilde{\Pi}_{R,\infty,A}^\mathrm{int}$ over the corresponding sites.\\
\end{proposition}

\begin{proof}
For one, apply the previous theorem. For two, repeat the argument in the proof of the previous theorem.	
\end{proof}

\indent The following definition is kind of generalization of the corresponding one in \cite[Definition 4.5.9]{KL16}:

\begin{definition}
Now over the ring $\widetilde{\Pi}_{R,A}$ or $\widetilde{\Pi}^\mathrm{bd}_{R,A}$	we call the corresponding Frobenius modules globally \'etale if they arise from the Frobenius modules over the ring $\widetilde{\Pi}^\mathrm{int}_{R,A}$. Now over the ring $\widetilde{\Pi}_{R,\infty,A}$ or $\widetilde{\Pi}^\mathrm{bd}_{R,\infty,A}$	we call the corresponding Frobenius modules globally \'etale if they arise from the Frobenius modules over the ring $\widetilde{\Pi}^\mathrm{int}_{R,\infty,A}$. \\
Now over the sheaf $\widetilde{\Pi}_{*,A}$ or $\widetilde{\Pi}^\mathrm{bd}_{*,A}$	we call the corresponding Frobenius modules globally \'etale if they arise from the Frobenius modules over the sheaf $\widetilde{\Pi}^\mathrm{int}_{*,A}$. Now over the ring $\widetilde{\Pi}_{*,\infty,A}$ or $\widetilde{\Pi}^\mathrm{bd}_{*,\infty,A}$	we call the corresponding Frobenius modules globally \'etale if they arise from the Frobenius modules over the sheaf $\widetilde{\Pi}^\mathrm{int}_{*,\infty,A}$.

\end{definition}

\begin{proposition} \mbox{\bf{(After Kedlaya-Liu \cite[Theorem 4.5.11]{KL16})}}
1. Let $X$ be a preadic space over $E_{,a}$. Then we have that there is a fully faithful embedding of the category of the corresponding $E^{\varphi^a}_{,a}\widehat{\otimes}A$-local systems (in the projective setting) into the category of the corresponding projective Frobenius $\varphi^a$-modules over $\widetilde{\Pi}_{*,A}$;\\
2. Let $X$ be a preadic space over $E_{\infty,a}$. Then we have that there is a fully faithful embedding of the category of the corresponding $E^{\varphi^a}_{\infty,a}\widehat{\otimes}A$-local systems (in the projective setting) into the category of the corresponding projective Frobenius $\varphi^a$-modules over $\widetilde{\Pi}_{*,A}$;\\
\end{proposition}

\begin{proof}
As in \cite[Theorem 4.5.11]{KL16}, we consider the corresponding base change of the corresponding exact sequence in \cref{proposition4.1} which reflects an exact sequence on the sheaves.
\end{proof}

\begin{remark}
We actually did not get the corresponding $v$-topology involved here, which is due to the fact that the pseudocoherent modules might possibly behavior annoyingly in such topology. But on the other hand, at least one will have the corresponding parallel statements as above over $v$-topology for finite projective objects.\\	
\end{remark}

\subsection{The Comparison beyond the \'Etale Setting}

\indent In this subsection we are going to study the corresponding relationship between the corresponding sheaves over the certain Fargues-Fontaine curves and the corresponding Frobenius modules, which is beyond the corresponding consideration of just \'etale objects as what we considered in the previous subsection. This is to some extend in our situation important since we would like to use different point of views to study essentially equivalent objects. The following result is actually directly reduced to \cite[Theorem 4.6.1]{KL16}:

\begin{theorem} \mbox{\bf (\cite[Theorem 4.6.1]{KL16})}
In the setting where the space $X$ is $\mathrm{Spa}(R,R^+)$ we have the following categories are equivalent:\\
1.  The category of all the pseudocohrent Frobenius $\varphi^a$-modules over the sheaf $\widetilde{\Pi}^\infty_{*}$;\\
2.  The category of all the pseudocohrent Frobenius $\varphi^a$-modules over the sheaf $\widetilde{\Pi}_{*}$;\\
3.  The category of all the pseudocohrent Frobenius $\varphi^a$-modules over the sheaf $\widetilde{\Pi}^I_{*}$, for some closed interval $I$;\\
4.  The category of all the strictly-pseudocoherent Frobenius $\varphi^a$-modules over $\widetilde{\Pi}^\infty_{R}$ such that for any closed interval $I$ the corresponding base change of any such module $M$ to the module over $\widetilde{\Pi}^I_{R}$ is \'etale-stably pseudocoherent;\\
5.  The category of all the pseudocoherent Frobenius $\varphi^a$-modules over $\widetilde{\Pi}_{R}$ such that each $M$ of such modules comes from a base change from some strictly-pseudocoherent Frobenius $\varphi^a$-module $M'$ over $\widetilde{\Pi}^r_{R}$ for some radius $r>0$ where the corresponding base change to (for any closed interval $I\subset (0,r]$) the module over $\widetilde{\Pi}^I_{R}$ is assumed to be \'etale-stably pseudocoherent;\\
6.  The category of all the pseudocoherent $\varphi^a$-bundles over $\widetilde{\Pi}_{R}$;\\
7.  The category of all the Frobenius $\varphi^a$-modules over $\widetilde{\Pi}^{[s,r]}_{R}$ where we have $0<s\leq r/p^{ha}$;\\
8.  The category of all the pseudocoherent sheaves over the adic version Fargues-Fontaine curve $\mathrm{FF}_{X,\text{\'et}}$ in the \'etale setting;\\
9.  The category of all the pseudocoherent sheaves over the adic version Fargues-Fontaine curve $\mathrm{FF}_{X,\text{pro\'et}}$ in the pro-\'etale setting.\\
\end{theorem}

\begin{proof}
This is \cite[Theorem 4.6.1]{KL16}. We remind the readers of the corresponding functors involved. First the functors from 4 to 5 and to 7, from 4 to 1 and to 3 are base changes, and the functor from 9 to 8 is the pullback along the corresponding morphism of the sites. The functors from 9 to 7 and 6 are restrictions of the corresponding objects involved. 	
\end{proof}

\indent We do the first generalization to the general ultrametric field:

\begin{theorem} \mbox{\bf{(After Kedlaya-Liu \cite[Theorem 4.6.1]{KL16})}}
In the setting where the space $X$ is $\mathrm{Spa}(R,R^+)$ we have the following categories are equivalent:\\
1.  The category of all the pseudocohrent Frobenius $\varphi^a$-modules over the sheaf $\widetilde{\Pi}^\infty_{*,\infty}$;\\
2.  The category of all the pseudocohrent Frobenius $\varphi^a$-modules over the sheaf $\widetilde{\Pi}_{*,\infty}$;\\
3.  The category of all the pseudocohrent Frobenius $\varphi^a$-modules over the sheaf $\widetilde{\Pi}^I_{*,\infty}$, for some closed interval $I$;\\
4.  The category of all the strictly-pseudocoherent Frobenius $\varphi^a$-modules over $\widetilde{\Pi}^\infty_{R,\infty}$ such that for any closed interval $I$ the corresponding base change of any such module $M$ to the module over $\widetilde{\Pi}^I_{R,\infty}$ is \'etale-stably pseudocoherent;\\
5.  The category of all the pseudocoherent Frobenius $\varphi^a$-modules over $\widetilde{\Pi}_{R,\infty}$ such that each $M$ of such modules comes from a base change from some strictly-pseudocoherent Frobenius $\varphi^a$-module $M'$ over $\widetilde{\Pi}^r_{R,\infty}$ for some radius $r>0$ where the corresponding base change to (for any closed interval $I\subset (0,r]$) the module over $\widetilde{\Pi}^I_{R,\infty}$ is assumed to be \'etale-stably pseudocoherent;\\
6.  The category of all the pseudocoherent $\varphi^a$-bundles over $\widetilde{\Pi}_{R,\infty}$;\\
7.  The category of all the Frobenius $\varphi^a$-modules over $\widetilde{\Pi}^{[s,r]}_{R,\infty}$ where we have $0<s\leq r/p^{ha}$;\\
8.  The category of all the pseudocoherent sheaves over the adic version Fargues-Fontaine curve $\mathrm{FF}_{X,\infty,\text{\'et}}$ in the \'etale setting;\\
9.  The category of all the pseudocoherent sheaves over the adic version Fargues-Fontaine curve $\mathrm{FF}_{X,\infty,\text{pro\'et}}$ in the pro-\'etale setting.
\end{theorem}

\begin{proof}
The functors are the same as in the previous theorem. The proof is actually encoded into the corresponding proof of the following theorem.	
\end{proof}

\indent We are then going to focus on the deformed version of the corresponding correspondences established above:

\begin{theorem}  \mbox{\bf{(After Kedlaya-Liu \cite[Theorem 4.6.1]{KL16})}}
In the setting where the space $X$ is $\mathrm{Spa}(R,R^+)$ we have the following first group of categories are equivalent:\\
1.  The category of all the pseudocohrent Frobenius $\varphi^a$-modules over the sheaf $\widetilde{\Pi}^\infty_{*,A}$;\\
2.  The category of all the pseudocohrent Frobenius $\varphi^a$-modules over the sheaf $\widetilde{\Pi}_{*,A}$;\\
3.  The category of all the pseudocohrent Frobenius $\varphi^a$-modules over the sheaf $\widetilde{\Pi}^I_{*,A}$, for some closed interval $I$.\\
\indent Then we have the second group of categories which are equvalent:\\
4.  The category of all the strictly-pseudocoherent Frobenius $\varphi^a$-modules over $\widetilde{\Pi}^\infty_{R,A}$ such that for any closed interval $I$ the corresponding base change of any such module $M$ to the module over $\widetilde{\Pi}^I_{R,A}$ is \'etale-stably pseudocoherent;\\
5.  The category of all the pseudocoherent Frobenius $\varphi^a$-modules over $\widetilde{\Pi}_{R,A}$ such that each $M$ of such modules comes from a base change from some strictly-pseudocoherent Frobenius $\varphi^a$-module $M'$ over $\widetilde{\Pi}^r_{R,A}$ for some radius $r>0$ where the corresponding base change to (for any closed interval $I\subset (0,r]$) the module over $\widetilde{\Pi}^I_{R,A}$ is assumed to be \'etale-stably pseudocoherent;\\
6.  The category of all the pseudocoherent $\varphi^a$-bundles over $\widetilde{\Pi}_{R,A}$;\\
7.  The category of all the Frobenius $\varphi^a$-modules over $\widetilde{\Pi}^{[s,r]}_{R,A}$ where we have $0<s\leq r/p^{ha}$;\\
8.  The category of all the pseudocoherent sheaves over the adic version Fargues-Fontaine curve $\mathrm{FF}_{X,A,\text{\'et}}$ in the \'etale setting;\\
9.  The category of all the pseudocoherent sheaves over the adic version Fargues-Fontaine curve $\mathrm{FF}_{X,A,\text{pro\'et}}$ in the pro-\'etale setting.\\
\end{theorem}

\begin{proof}
First the functors are the same as in the proof of the previous theorem. For the proof, we can follow the idea of the proof of \cite[Theorem 4.6.1]{KL16} to do so. First the equivalences between 9,8 and 6 could be proved by considering \cref{proposition3.13}. By considering the corresponding Frobenius action, one can establish the corresponding equivalences among them and further 7. Now as in the proof of \cite[Theorem 4.6.1]{KL16}, we can now prove the equivalence between 4 and 6 as in the following. First as in the proof of \cite[Theorem 4.6.1]{KL16} one considers the exact base change from the key period ring in 4 to the corresponding key period ring in 6, which implies that the corresponding base change functor is then fully faithful. To show it is also essentially surjective, we pick now an arbitrary Frobenius bundle in the category 6, and consider the corresponding reified quasi-Stein covering of the total space by spaces taking the general form as $\mathrm{Spa}(\Pi^{[sp^{nah},rp^{nah}]}_{R,A},\Pi^{[sp^{nah},rp^{nah}],\mathrm{Gr}}_{R,A})$ with well-located radii $r$ and $s$. This will imply that the corresponding global section of the Frobenius bundle is finitely generated after applying the \cite[Proposition 2.7.16]{KL16} since the application of Frobenius action will imply the uniform bound of the numbers of generators over each such covering subspace as in \cite[Theorem 4.6.1]{KL16}. Then having shown this we can as in \cite[Thereom 4.6.1]{KL16} extract the corresponding desired property of being pseudocoherent for the corresponding desired object in the category 4 by considering the kernel of some finite free covering of the Frobenius bundle. Then as in \cite[Theorem 4.6.1]{KL16} one can show that we have now the categories 4,5 and 6 are equivalent, and furthermore in the same way one may prove that the categories 1,2 and 3 are equivalent.
\end{proof}

\begin{remark}
We did not manage to compare among then the two groups of categories within this theorem, but we believe this could be achieved by further careful analysis.	Please note that the corresponding sheaves in the categories 1,2,3 are actually required to be \'etale-stably pseudocoherent over the ring $A$ as well. 
\end{remark}

\indent Now over smooth Fr\'echet-Stein algebra $A_\infty(G)$ (as in our previous paper) in the commutative setting where each $A_n(G)$ in the family for each $n\geq 0$ is smooth reduced affinoid algebra in rigid analytic space, we have the following comparison:

\begin{theorem} \mbox{\bf{(After Kedlaya-Liu \cite[Theorem 4.6.1]{KL16})}} 
In the setting where the space $X$ is $\mathrm{Spa}(R,R^+)$ and the field $E$ is of mixed-characteristic we have the following categories are equivalent:\\
1.  The category of all the families of all the pseudocohrent Frobenius $\varphi^a$-modules over the sheaf $\widetilde{\Pi}^\infty_{*,A_\infty(G)}$;\\
2.  The category of all the families of all the pseudocohrent Frobenius $\varphi^a$-modules over the sheaf $\widetilde{\Pi}_{*,A_\infty(G)}$;\\
3.  The category of all the families of all the pseudocohrent Frobenius $\varphi^a$-modules over the sheaf $\widetilde{\Pi}^I_{*,A_\infty(G)}$, for some closed interval $I$.\\
\indent Then we have the second group of categories which are equvalent:\\
4.  The category of all the families of all the strictly-pseudocoherent Frobenius $\varphi^a$-modules over $\widetilde{\Pi}^\infty_{R,A_\infty(G)}$ such that for any closed interval $I$ the corresponding base change of any such module $M$ to the module over $\widetilde{\Pi}^I_{R,A_\infty(G)}$ is \'etale-stably pseudocoherent;\\
5.  The category of all the families of all the pseudocoherent Frobenius $\varphi^a$-modules over $\widetilde{\Pi}_{R,A_\infty(G)}$ such that each $M$ of such modules comes from a base change from some strictly-pseudocoherent Frobenius $\varphi^a$-module $M'$ over $\widetilde{\Pi}^r_{R,A_\infty(G)}$ for some radius $r>0$ where the corresponding base change to (for any closed interval $I\subset (0,r]$) the module over $\widetilde{\Pi}^I_{R,A_\infty(G)}$ is assumed to be \'etale-stably pseudocoherent;\\
6.  The category of all the families of all the pseudocoherent $\varphi^a$-bundles over $\widetilde{\Pi}_{R,A_\infty(G)}$;\\
7.  The category of all the families of all the Frobenius $\varphi^a$-modules over $\widetilde{\Pi}^{[s,r]}_{R,A_\infty(G)}$ where we have $0<s\leq r/p^{ha}$;\\
8.  The category of all the families of all the pseudocoherent sheaves over the adic version Fargues-Fontaine curve $\mathrm{FF}_{X,A_\infty(G),\text{\'et}}$ in the \'etale setting;\\
9.  The category of all the families of all the pseudocoherent sheaves over the adic version Fargues-Fontaine curve $\mathrm{FF}_{X,A_\infty(G),\text{pro\'et}}$ in the pro-\'etale setting.
\end{theorem}

\begin{proof}
This is the direct consequence of the previous theorem if one considers the corresponding systems of all the objects defined over each $A_n(G)$.	
\end{proof}

\indent We now work over a perfectoid field $R$ and drop the assumption on the affinoid algebras we imposed before:

\begin{theorem} \mbox{\bf{(After Kedlaya-Liu \cite[Theorem 4.6.1]{KL16})}} 
In the setting where the space $X$ is $\mathrm{Spa}(R,R^+)$ where $R$ is further assumed to be an analytic field, and we let $A$ be a general reduced affinoid algebra, then we have the following categories are equivalent:\\
1.  The category of all the pseudocohrent Frobenius $\varphi^a$-modules over the sheaf $\widetilde{\Pi}^\infty_{*,A}$;\\
2.  The category of all the pseudocohrent Frobenius $\varphi^a$-modules over the sheaf $\widetilde{\Pi}_{*,A}$;\\
3.  The category of all the pseudocohrent Frobenius $\varphi^a$-modules over the sheaf $\widetilde{\Pi}^I_{*,A}$, for some closed interval $I$.\\
\indent Then we have the second group of categories which are equvalent:\\
4.  The category of all the strictly-pseudocoherent Frobenius $\varphi^a$-modules over $\widetilde{\Pi}^\infty_{R,A}$ such that for any closed interval $I$ the corresponding base change of any such module $M$ to the module over $\widetilde{\Pi}^I_{R,A}$ is \'etale-stably pseudocoherent;\\
5.  The category of all the pseudocoherent Frobenius $\varphi^a$-modules over $\widetilde{\Pi}_{R,A}$ such that each $M$ of such modules comes from a base change from some strictly-pseudocoherent Frobenius $\varphi^a$-module $M'$ over $\widetilde{\Pi}^r_{R,A}$ for some radius $r>0$ where the corresponding base change to (for any closed interval $I\subset (0,r]$) the module over $\widetilde{\Pi}^I_{R,A}$ is assumed to be \'etale-stably pseudocoherent;\\
6.  The category of all the pseudocoherent $\varphi^a$-bundles over $\widetilde{\Pi}_{R,A}$;\\
7.  The category of all the Frobenius $\varphi^a$-modules over $\widetilde{\Pi}^{[s,r]}_{R,A}$ where we have $0<s\leq r/p^{ha}$;\\
8.  The category of all the pseudocoherent sheaves over the adic version Fargues-Fontaine curve $\mathrm{FF}_{X,A,\text{\'et}}$ in the \'etale setting;\\
9.  The category of all the pseudocoherent sheaves over the adic version Fargues-Fontaine curve $\mathrm{FF}_{X,A,\text{pro\'et}}$ in the pro-\'etale setting.\\
\end{theorem}

\begin{proof}
See the proof of the previous theorem.	\\
\end{proof}


\subsection{The Comparison with the Schematic Fargues-Fontaine Curve}

\indent In this section we are going to study the corresponding relationship between the Frobenius modules over period rings and the certain sheaves over the deformed version of schematic Fargues-Fontaine curves. We actually discussed in our previous paper the corresponding results in the setting where $E$ is just $\mathbb{Q}_p$. Now we consider the more general setting and the corresponding story in the equal characteristic. First we consider the following key lemmas which are analogs in our context of \cite[6.2.2-6.2.4]{KL15}, \cite[Lemma 4.6.9]{KL16} and \cite[Proposition 2.11]{T1}:

\begin{lemma} \mbox{\bf{(After Kedlaya-Liu \cite[6.2.2-6.2.4]{KL15}, \cite[Lemma 4.6.9]{KL16}})}\\ \mbox{\bf{(And also see \cite[Proposition 2.11]{T1})}}
For any Frobenius $\varphi^a$-bundle $M$ over $\widetilde{\Pi}_{R,A}$, then we have that for any interval $I=[s,r]$ where $0<s\leq r$ the map $\varphi^a-1: M_{[s,rq]}\rightarrow M_{[s,r]}$ is surjective after taking some Frobenius twist as in \cite[6.2.2]{KL15} and our previous work namely the new morphism $\varphi^a-1: M(n)_{[s,rq]}\rightarrow M(n)_{[s,r]}$ for sufficiently large $n\geq 1$. As in the previous established version one may have the chance to take the number to be $1$ if the bundle initially comes from the corresponding base change from the integral Robba ring.	
\end{lemma}

\begin{proof}
See the proof of \cite[Proposition 2.11]{T1}.	
\end{proof}

\begin{lemma} \mbox{\bf{(After Kedlaya-Liu \cite[6.2.2-6.2.4]{KL15}, \cite[Lemma 4.6.9]{KL16}})}\\ \mbox{\bf{(And also see \cite[Proposition 2.11]{T1})}}
For any finitely generated bundle $M$ carrying the Frobenius action over $\widetilde{\Pi}_{R,A}$, then we have that for any interval $I=[s,r]$ where $0<s\leq r$ the map $\varphi^a-1: M_{[s,rq]}\rightarrow M_{[s,r]}$ is surjective after taking some Frobenius twist as in \cite[6.2.2]{KL15} and our previous work namely the new morphism $\varphi^a-1: M(n)_{[s,rq]}\rightarrow M(n)_{[s,r]}$ for sufficiently large $n\geq 1$. As in the previous established version one may have the chance to take the number to be $1$ if the bundle initially comes from the corresponding base change from the integral Robba ring.	
\end{lemma}

\begin{proof}
See the proof of \cite[Proposition 2.11]{T1}.	
\end{proof}

\begin{lemma} \mbox{\bf{(After Kedlaya-Liu \cite[6.2.2-6.2.4]{KL15}, \cite[Lemma 4.6.9]{KL16}})}\\ \mbox{\bf{(And also see \cite[Proposition 2.12]{T1})}}
For any Frobenius $\varphi^a$-bundle $M$ over $\widetilde{\Pi}_{R,A}$, then we have that for sufficiently large number $n\geq 0$ the module $M$ could be generated by finitely many Frobenius $\varphi^a$-invariant elements of the global sections of $M(n)$.	
\end{lemma}

\begin{proof}
See the proof of \cite[Proposition 2.12]{T1}.	
\end{proof}

\begin{lemma} \mbox{\bf{(After Kedlaya-Liu \cite[6.2.2-6.2.4]{KL15}, \cite[Lemma 4.6.9]{KL16}})}\\ \mbox{\bf{(And also see \cite[Proposition 2.12, Proposition 2.11]{T1})}}
For any finitely generated bundle $M$ carrying Frobenius action over $\widetilde{\Pi}_{R,A}$, then we have that for sufficiently large number $n\geq 0$ the module $M$ could be generated by finitely many Frobenius $\varphi^a$-invariant elements of the global sections of $M(n)$.	
\end{lemma}

\begin{proof}
See the proof of \cite[Proposition 2.12, Proposition 2.11]{T1}.	
\end{proof}

\begin{lemma} \mbox{\bf{(After Kedlaya-Liu \cite[6.2.2-6.2.4]{KL15}, \cite[Lemma 4.6.9]{KL16}})}\\ \mbox{\bf{(And also see \cite[Proposition 2.12, Proposition 2.11, Proposition 2.14]{T1})}}
For any Frobenius $\varphi^a$-module $M$ over $\widetilde{\Pi}_{R,A}$, then we have that for sufficiently large number $n\geq 1$ the space $H^0_{\varphi^a}(M(n))$ generates $M$ and the space $H^1_{\varphi^a}(M(n))$ vanishes.
\end{lemma}

\begin{proof}
See the proof of \cite[Proposition 2.12, Proposition 2.11, Proposition 2.14]{T1}.	
\end{proof}

\begin{lemma} \mbox{\bf{(After Kedlaya-Liu \cite[6.2.2-6.2.4]{KL15}, \cite[Lemma 4.6.9]{KL16}})}\\ \mbox{\bf{(And also see \cite[Proposition 2.12, Proposition 2.11, Proposition 2.14]{T1})}}
For any finitely generated module $M$ carrying Frobenius action over $\widetilde{\Pi}_{R,A}$, then we have that for sufficiently large number $n\geq 1$ the space $H^0_{\varphi^a}(M(n))$ generates $M$ and the space $H^1_{\varphi^a}(M(n))$ vanishes.
\end{lemma}

\begin{proof}
See the proof of \cite[Proposition 2.12, Proposition 2.11, Proposition 2.14]{T1}.	
\end{proof}

\indent Then we have the following key corollary which is analog of \cite[Corollary 6.2.3, Lemma 6.3.3]{KL15}, \cite[Corollary 4.6.10]{KL16} and \cite[Corollary 2.13, Corollary 3.4, Proposition 3.9]{T1}:

\begin{corollary} \label{coro4.16}
I. When $M_\alpha,M,M_\beta$ are three Frobenius $\varphi^a$-bundles over the ring $\widetilde{\Pi}_{R,A}$ then we have that for sufficiently large $n\geq 0$ we have the following exact sequence:
\[
\xymatrix@R+0pc@C+0pc{
0\ar[r]\ar[r]\ar[r] &M_{\alpha,I}(n)^{\varphi^a}
\ar[r]\ar[r]\ar[r] &M_I(n)^{\varphi^a}
\ar[r]\ar[r]\ar[r] &M_{\beta,I}(n)^{\varphi^a} \ar[r]\ar[r]\ar[r] &0,
}
\]	
for each interval $I$ (same holds true for finite objects). \\
II. When $M_\alpha,M,M_\beta$ are three Frobenius $\varphi^a$-modules over the ring $\widetilde{\Pi}_{R,A}$ then we have that for sufficiently large $n\geq 0$ we have the following exact sequence:
\[
\xymatrix@R+0pc@C+0pc{
0\ar[r]\ar[r]\ar[r] &M_{\alpha,I}(n)^{\varphi^a}
\ar[r]\ar[r]\ar[r] &M_I(n)^{\varphi^a}
\ar[r]\ar[r]\ar[r] &M_{\beta,I}(n)^{\varphi^a} \ar[r]\ar[r]\ar[r] &0,
}
\]	
for each interval $I$ (same holds true for finite objects).\\
III. For a Frobenius $\varphi^a$-bundle $M$ over $\widetilde{\Pi}_{R,A}$, and for each module $M_I$ over some $\widetilde{\Pi}_{R,A}^{I}$, we put for any element $f$ such that $\varphi^af=p^df$:
\begin{displaymath}
M_{I,f}:=\bigcup_{n\in \mathbb{Z}}f^{-n}M_I(dn)^{\varphi^a}.	
\end{displaymath}
Then with this convention suppose we have three Frobenius $\varphi^a$-bundles taking the form of $M_\alpha,M,M_\beta$ over
$\widetilde{\Pi}_{R,A}$, then for each closed interval $I$ we have the following is an exact sequence:
\[
\xymatrix@R+0pc@C+0pc{
0\ar[r]\ar[r]\ar[r] &M_{\alpha,I,f}\ar[r]\ar[r]\ar[r] &M_{I,f}
\ar[r]\ar[r]\ar[r] &M_{\beta,I,f} \ar[r]\ar[r]\ar[r] &0,
}
\]
where each module in the exact sequence is now a module over $\widetilde{\Pi}_{R,A}[1/f]^{\varphi^a}$ (same holds true for finite objects). \\
IV. For a Frobenius $\varphi^a$-module $M$ over $\widetilde{\Pi}_{R,A}$, we put for any element $f$ such that $\varphi^af=p^df$:
\begin{displaymath}
M_{f}:=\bigcup_{n\in \mathbb{Z}}f^{-n}M(dn)^{\varphi^a}.	
\end{displaymath}
Then with this convention suppose we have three Frobenius $\varphi^a$-modules taking the form of $M_\alpha,M,M_\beta$ over
$\widetilde{\Pi}_{R,A}$, then we have the following is an exact sequence:
\[
\xymatrix@R+0pc@C+0pc{
0\ar[r]\ar[r]\ar[r] &M_{\alpha,f}\ar[r]\ar[r]\ar[r] &M_{f}
\ar[r]\ar[r]\ar[r] &M_{\beta,f} \ar[r]\ar[r]\ar[r] &0,
}
\]
where each module in the exact sequence is now a module over $\widetilde{\Pi}_{R,A}[1/f]^{\varphi^a}$ (same holds true for finite objects). \\
V. Suppose $M$ is now a pseudocoherent Frobenius $\varphi^a$-module over $\widetilde{\Pi}_{R,A}$, then we have that the corresponding module $M_f$ then is also pseudocoherent over the ring $\widetilde{\Pi}_{R,A}[1/f]^{\varphi^a}$. For a Frobenius $\varphi^a$-bundle $M$ over $\widetilde{\Pi}_{R,A}$, and for each module $M_I$ over some $\widetilde{\Pi}_{R,A}^{I}$, we put for any element $f$ such that $\varphi^af=p^df$:
\begin{displaymath}
M_{I,f}:=\bigcup_{n\in \mathbb{Z}}f^{-n}M_I(dn)^{\varphi^a}.	
\end{displaymath}
Then with this convention suppose we have three Frobenius $\varphi^a$-bundles taking the form of $M_\alpha,M,M_\beta$ over
$\widetilde{\Pi}_{R,A}$, then for each closed interval $I$ we have the following is an exact sequence:
\[
\xymatrix@R+0pc@C+0pc{
0\ar[r]\ar[r]\ar[r] &M_{\alpha,I,f}\ar[r]\ar[r]\ar[r] &M_{I,f}
\ar[r]\ar[r]\ar[r] &M_{\beta,I,f} \ar[r]\ar[r]\ar[r] &0,
}
\]
where each module in the exact sequence is now in our situation a pseudocoherent module over $\widetilde{\Pi}_{R,A}[1/f]^{\varphi^a}$.\\

 \end{corollary}

\begin{proof}
See the proof of \cite[Corollary 6.2.3, Lemma 6.3.3]{KL15}, \cite[Corollary 4.6.10]{KL16} and \cite[Corollary 2.13, Corollary 3.4, Proposition 3.9]{T1}.	
\end{proof}

\indent Then we recall the following construction both from our previous consideration and the corresponding consideration in \cite[Definition 6.3.1]{KL15} and \cite[Definition 4.6.11]{KL16}:

\begin{setting}
Now we consider the graded commutative ring $P_{R,A}$ which is constructed from the ring $\widetilde{\Pi}_{R,A}^+$, or $\widetilde{\Pi}_{R,A}$, or $\widetilde{\Pi}_{R,A}^\infty$ by taking each $\pi^d$-eigenfunction of the operator $\varphi^a$, where $d$ corresponds to the degree of each element in $P_{d,R,A}$. Then for each element $f$ in this graded commutative ring, one considers the affine charts taking the form of $\mathrm{Spec}P_{R,A}[1/f]_0$ which further glues to a projective scheme which is denoted by $\mathrm{Proj}P_{R,A}$ and called the schematic deformed version of the schematic Fargues-Fontaine curve.
\end{setting}

And we have the following analog of \cite{KL15} and the corresponding construction in \cite[Setting 3.3]{T1}:

\begin{setting}
Starting from any $\varphi^a$-bundle $M$ over the ring $\widetilde{\Pi}_{R,A}$, one has by the construction in the previous corollary the corresponding module $M_{I,f}$ over the ring $P_{R,A}[1/f]_0$, which then by considering the natural base change construction gives rise to the map:
\begin{displaymath}
M_{I,f}\otimes_{P_{R,A}[1/f]_0}	\widetilde{\Pi}^I_{R,A}[1/f]\rightarrow 
M_{I,f}\otimes_{\widetilde{\Pi}^I_{R,A}}\widetilde{\Pi}^I_{R,A}[1/f]
\end{displaymath}
for each closed interval $I$.
\end{setting}

\begin{proposition} \mbox{\bf{(After Kedlaya-Liu \cite[Theorem 6.3.9]{KL15})}}
The map defined in the previous setting is a bijection and the corresponding module $M_{I,f}$ for each specific interval $I$ is then in our situation projective of finite type.
\end{proposition}

\begin{proof}
See \cite[Proposition 3.6]{T1}.	
\end{proof}

\begin{proposition} \mbox{\bf{(After Kedlaya-Liu \cite[Theorem 6.3.12]{KL15})}}
We have the following categories are equivalent:\\
I. The category of all the $\varphi^a$-bundles over $\widetilde{\Pi}_{R,A}$;\\
II. The category of all the $\varphi^a$-modules over $\widetilde{\Pi}_{R,A}$;\\
III. The category of all the $\varphi^a$-modules over $\widetilde{\Pi}^\infty_{R,A}$;\\
IV.  The category of all the quasicoherent finite locally free sheaves over the scheme $\mathrm{Proj}P_{R,A}$.	
\end{proposition}

\begin{proof}
Based on the previous proposition, we only need to recall the corresponding functor involved. Starting from an object in the category $IV$ one locally considers the corresponding pullbacks along the natural map from the localized scheme attached to the key ring in the category $III$, which globally glues to a well-defined object in the corresponding module with Frobenius structure in the category $III$, by applying the corresponding analog of \cite[Lemma 6.3.7]{KL15}. Then the natural base change functor sends the resulting object to some well-defined object in the category $II$. Finally we associate the corresponding $\varphi^a$-bundle to derive the corresponding object in the first category $I$. 
\end{proof}

\begin{proposition} \mbox{\bf{(After Kedlaya-Liu \cite[Theorem 4.6.12]{KL16})}}
The natural pullback functor from the Fargues-Fontaine curve (which is schematic) to the scheme associated to the Robba ring $\widetilde{\Pi}^\infty_{R,A}$ establishes an equivalence between the category of all the pseudocoherent sheaves over the deformed version of the schematic Fargues-Fontaine curve (which is schematic) and the category of all the pseudocoherent modules over $\widetilde{\Pi}_{R,A}$ with isomorphisms established by the Frobenius pullbacks. 	
\end{proposition}

\begin{proof}
As in \cite[Theorem 4.6.12]{KL16} and our previous paper \cite[Proposition 3.10]{T1}, we start from an object $V$ in the category of the pseudocoherent sheaves over the schematic Fargues-Fontaine curve, which gives rise to the following exact sequence:
\[
\xymatrix@R+0pc@C+0pc{
0\ar[r]\ar[r]\ar[r] &V_2\ar[r]\ar[r]\ar[r] &V_1
\ar[r]\ar[r]\ar[r] &V \ar[r]\ar[r]\ar[r] &0,
}
\]	
where the sheaf $V_1$ is projective of finite type. Then the corresponding functor mentioned in the statement of the proposition establishes then the following exact sequence:
\[
\xymatrix@R+0pc@C+0pc{
 W_2\ar[r]\ar[r]\ar[r] &W_1
\ar[r]\ar[r]\ar[r] &W \ar[r]\ar[r]\ar[r] &0,
}
\]	
in the category of all the pseudocoherent modules over the Robba ring with isomorphisms established by the Frobenius pullbacks, where then $W_1$ is finite projective and $W_2$ is finite generated. Then as in \cite[Theorem 4.6.12]{KL16} and \cite[Proposition 3.10]{T1} we consider the kernel $K$ of the map $W_2\rightarrow \mathrm{Ker}(W_1\rightarrow W)$, which gives rise to the following exact sequence:
\[
\xymatrix@R+0pc@C+0pc{
0\ar[r]\ar[r]\ar[r] &K\ar[r]\ar[r]\ar[r] &W_2\ar[r]\ar[r]\ar[r] &W_1
\ar[r]\ar[r]\ar[r] &W \ar[r]\ar[r]\ar[r] &0,
}
\]
which then gives rise to the following exact sequence by \cref{coro4.16}:
\[
\xymatrix@R+0pc@C+0pc{
0\ar[r]\ar[r]\ar[r] &K_f\ar[r]\ar[r]\ar[r] &W_{2,f}\ar[r]\ar[r]\ar[r] &W_{1,f}
\ar[r]\ar[r]\ar[r] &W_f \ar[r]\ar[r]\ar[r] &0,
}
\]
over the ring $\widetilde{\Pi}_{R,A}[1/f]^{\varphi^a}$. Then by taking the corresponding section over the ring $P_{R,A}[1/f]_0$ we have the following commutative diagram:
\[
\xymatrix@R+2pc@C+1.7pc{
&0\ar[r]\ar[r]\ar[r] &V_{2}\Big|_{P_{R,A}[1/f]_0}\ar[d]\ar[d]\ar[d]\ar[r]\ar[r]\ar[r] &V_{1}\Big|_{P_{R,A}[1/f]_0}
\ar[d]\ar[d]\ar[d]\ar[r]\ar[r]\ar[r] &V \Big|_{P_{R,A}[1/f]_0} \ar[d]\ar[d]\ar[d]\ar[r]\ar[r]\ar[r] &0\\
0\ar[r]\ar[r]\ar[r] &K_f\ar[r]\ar[r]\ar[r] &W_{2,f}\ar[r]\ar[r]\ar[r] &W_{1,f}
\ar[r]\ar[r]\ar[r] &W_f \ar[r]\ar[r]\ar[r] &0.\\
}
\]
Now the middle vertical arrow is isomorphism which shows that the right vertical arrow is surjective. Then apply the corresponding construction to the situation where we replace $V$ by $V_2$ we will have the parallel result which imply that the first vertical arrow is also surjective. Now by five lemma we have that the rightmost vertical arrow is injective. Then by applying the corresponding consideration to the situation where we replace $V$ by $V_2$ we could derive the fact that actually the first vertical arrow is also injective. Up to here, we have that the corresponding module $K_f$ is zero which implies by considering \cref{coro4.16} $K$ itself is zero. Then by the previous results on the corresponding finite projective objects we have that functor in our situation sends the pseudocoherent objects to pseudocoherent objects. And from the argument above we have that the functor is exact. Then to finish the composition of the functors from category of Frobenius modules to the sheaves and back from sheaves to the category of Frobenius modules is an equivalence. And the other direction could be directly deduced from the results on the vector bundles.
\end{proof}

\newpage

\section{Deformation of Imperfect Period Rings}

\subsection{Key Rings}

\noindent We now consider the corresponding imperfectization of the corresponding constructions we considered above, after \cite{KL16}. The construction in \cite[Chapter 5]{KL16} is actually quite general, which is sort of direct imperfection of the corresponding perfect period rings and modules off the tower. We will consider the corresponding towers in \cite[Chapter 5]{KL16}, so we recall the setting as in the following:

\begin{assumption}
In this section, we are going to assume that $k$ is just $\mathbb{F}_{p^h}$.
\end{assumption}

\begin{setting}
Recall from \cite[Chapter 5]{KL16}, we have the following setting up. The base will be a Banach adic ring $(H,H^+)$ over $\mathfrak{o}_E$ which is assumed to be uniform and carrying the corresponding spectral norm $\alpha$. Then we consider the tower:
\[
\xymatrix@R+0pc@C+0pc{
...\ar[r]\ar[r]\ar[r] &H_{n-1}\ar[r]\ar[r]\ar[r] &H_n
\ar[r]\ar[r]\ar[r] &H_{n+1} \ar[r]\ar[r]\ar[r] &...,
}
\]	
where each map linking two adjacent rings is the corresponding morphism of the corresponding Banach uniform adic rings whose corresponding induced map on the adic spaces is surjective as in the corresponding construction in \cite[Definition 5.1.1]{KL16} with the corresponding spectral norm $\alpha_n$. The infinite level of the tower will be denoted by $H_\infty$. This is not actually necessarily complete under the multiplicative extension of all the finite level spectral norms $\alpha_n$ namely $\alpha_\infty$, so we need to consider the completed ring $\overline{H}_\infty$. In the situation where the tower is Fontaine perfectoid, we have that following \cite[Definition 5.1.1]{KL16} that this is called the corresponding perfectoid tower, which gives rise the equal characteristic counterpart $\overline{H'}_\infty$ wit the spectral norm $\overline{\alpha}_\infty$ correspondingly under the perfectoid correspondence in the sense of \cite[Theorem 3.3.8]{KL16}. Recall that from \cite[Definition 5.1.1]{KL16} the tower is called finite \'etale if each transition map is finite \'etale.
\end{setting}

\indent Now we describe the corresponding imperfect period rings which we will deform. These rings are those introduced in \cite[Definition 5.2.1]{KL16} by using series of imperfection processes. Recall in more detail we have:

\begin{setting} \label{setting6.3}
Fix a perfectoid tower $(H_\bullet,H^+_\bullet)$ Recall from \cite[Definition 5.2.1]{KL16} we have the following different imperfect constructions:\\
A. First we have the ring $\overline{H'}_\infty$, which could give us the corresponding ring $\widetilde{\Omega}^{\mathrm{int}}_{\overline{H'}_\infty}$; \\
B. We then have the ring $\widetilde{\Pi}^{\mathrm{int},r}_{\overline{H'}_\infty}$ coming from $\overline{H'}_\infty$;\\
C. We then have the ring $\widetilde{\Pi}^{\mathrm{int}}_{\overline{H'}_\infty}$ coming from $\overline{H'}_\infty$;\\
D. We then have the ring $\widetilde{\Pi}^{\mathrm{bd},r}_{\overline{H'}_\infty}$ coming from $\overline{H'}_\infty$;\\
E. We then have the ring $\widetilde{\Pi}^{\mathrm{bd}}_{\overline{H'}_\infty}$ coming from $\overline{H'}_\infty$;\\
F. We then have the ring $\widetilde{\Pi}^{r}_{\overline{H'}_\infty}$ coming from $\overline{H'}_\infty$;\\
G. We then have the ring $\widetilde{\Pi}^{[s,r]}_{\overline{H'}_\infty}$ coming from $\overline{H'}_\infty$;\\
H. We then have the ring $\widetilde{\Pi}_{\overline{H'}_\infty}$ coming from $\overline{H'}_\infty$;\\
I. $\Pi^{\mathrm{int},r}_{H}$ comes from the ring $\widetilde{\Pi}^{\mathrm{int},r}_{\overline{H'}_\infty}$ consisting of those elements of $\widetilde{\Pi}^{\mathrm{int},r}_{\overline{H'}_\infty}$ with the requirement that whenever we have $n$ an integer such that $nh>-\log_pr$ we have then $\theta(\varphi^{-n}(x))\in H_n$;\\
J. $\Pi^{\mathrm{int},\dagger}_{H}$ is defined to be the corresponding union of the rings in $I$;\\
K. $\Omega^\mathrm{int}_H$ is defined to be the corresponding period ring coming from the corresponding $\pi$-adic completion of the ring $\Pi^{\mathrm{int},\dagger}_{H}$ in $J$;\\
L. $\breve{\Omega}^{\mathrm{int}}_{H}$ is the ring which is defined to be the union of all the $\varphi^{-n}\Omega^\mathrm{int}_H$;\\
M. $\breve{\Pi}^{\mathrm{int},r}_H$ is then the ring which is defined to be the union of all the $\varphi^{-n}\Pi^{\mathrm{int},p^{hn}r}_H$;\\
N. $\breve{\Pi}^{\mathrm{int},\dagger}_H$ is define to be union of all the $\varphi^{-n}\breve{\Pi}^{\mathrm{int},\dagger}_H$;\\
O. $\widehat{\Omega}^{\mathrm{int}}_{H}$ is defined to be the $\pi$-completion of $\breve{\Omega}^{\mathrm{int}}_{H}$;\\
P. $\widehat{\Pi}^{\mathrm{int},r}_H$ is defined to be the $\mathrm{max}\{\|.\|_{\overline{\alpha}_\infty^r},\|.\|_{\pi-\text{adic}}\}$-completion of $\breve{\Pi}^{\mathrm{int},r}_H$;\\
Q. We then have $\widehat{\Pi}^{\mathrm{int},\dagger}_H$ by taking the union over $r>0$;\\
R. Correspondingly we have $\breve{\Omega}_{H}$, $\breve{\Pi}^{\mathrm{bd},r}_H$, $\breve{\Pi}^{\mathrm{bd},\dagger}_H$ by inverting the element $\pi$;\\
S. Correspondingly we also have $\widehat{\Omega}_{H}$, $\widehat{\Pi}^{\mathrm{bd},r}_H$, $\widehat{\Pi}^{\mathrm{bd},\dagger}_H$ by inverting the corresponding element $\pi$;\\
T. We also have $\Omega_{H}$, $\Pi^{\mathrm{bd},r}_H$, $\Pi^{\mathrm{bd},\dagger}_H$ again by inverting the element $\pi$;\\
U. Taking the $\max\{\|.\|_{\overline{\alpha}_\infty^s},\|.\|_{\overline{\alpha}_\infty^r}\}$ (for $0<s\leq r$) completion of the ring $\Pi^{\mathrm{bd},r}_H$ we have the ring $\Pi^{[s,r]}_{H}$, while taking the Fr\'echet completion with respect to the norm $\|.\|_{\overline{\alpha}^t_\infty}$ for each $0<t\leq r$ we have the ring $\Pi^r_{H}$;\\
V. Taking the union we have the ring $\Pi_H$;\\
W. We use the notation $\breve{\Pi}_H$ to denote the corresponding union of all the $\varphi^{-n}\Pi_H$;\\
X. We use the notation $\breve{\Pi}^{[s,r]}_{H}$ to be the corresponding union of all the $\varphi^{-n}\Pi^{[p^{-hn}s,p^{-hn}r]}_H$;\\
Y. We use the notation $\breve{\Pi}_H^r$ to be the union of all the $\varphi^{-n}\breve{\Pi}_H^{p^{-hn}r}$.
\end{setting}

\indent Then we have the following direct analog of the relative version of the ring defined above (here as before the ring $A$ denotes a reduced affinoid algebra):

\begin{setting} \label{setting6.4}
Now we consider the deformation of the rings above:\\
I. We have the first group of the period rings in the deformed setting:
\begin{align}
\Pi^{\mathrm{int},r}_{H,A},\Pi^{\mathrm{int},\dagger}_{H,A},\Omega^\mathrm{int}_{H,A}, \Omega_{H,A}, \Pi^{\mathrm{bd},r}_{H,A}, \Pi^{\mathrm{bd},\dagger}_{H,A},\Pi^{[s,r]}_{H,A}, \Pi^r_{H,A}, \Pi_{H,A}.
\end{align}
II. We also have the second group of desired rings in the desired setting:
\begin{align}
\breve{\Pi}^{\mathrm{int},r}_{H,A},\breve{\Pi}^{\mathrm{int},\dagger}_{H,A},\breve{\Omega}^\mathrm{int}_{H,A}, \breve{\Omega}_{H,A}, \breve{\Pi}^{\mathrm{bd},r}_{H,A}, \breve{\Pi}^{\mathrm{bd},\dagger}_{H,A},\breve{\Pi}^{[s,r]}_{H,A}, \breve{\Pi}^r_{H,A}, \breve{\Pi}_{H,A}.	
\end{align}
III. We also have the third group:
\begin{align}
\widehat{\Pi}^{\mathrm{int},r}_{H,A},\widehat{\Pi}^{\mathrm{int},\dagger}_{H,A},\widehat{\Omega}^\mathrm{int}_{H,A}, \widehat{\Omega}_{H,A}, \widehat{\Pi}^{\mathrm{bd},r}_{H,A}, \widehat{\Pi}^{\mathrm{bd},\dagger}_{H,A}.	
\end{align}
\end{setting}

\indent Then we have the following globalization as what we did before in the perfect setting:

\begin{definition}
Now consider the sheaf $\mathcal{O}_{\mathfrak{X}}$ of a rigid analytic space $\mathfrak{X}$ in rigid analytic geometry. We have the following three groups of sheaves of rings over $\mathfrak{X}$:
\begin{align}
&\Pi^{[s,r]}_{H,\mathfrak{X}}, \Pi^r_{H,\mathfrak{X}}, \Pi_{H,\mathfrak{X}},\\
&\breve{\Pi}^{[s,r]}_{H,\mathfrak{X}}, \breve{\Pi}^r_{H,\mathfrak{X}}, \breve{\Pi}_{H,\mathfrak{X}}
\end{align}
to be:
\begin{align}
&\Pi^{[s,r]}_{H,\mathcal{O}_\mathfrak{X}}, \Pi^r_{H,\mathcal{O}_\mathfrak{X}}, \Pi_{H,\mathcal{O}_\mathfrak{X}},\\
&\breve{\Pi}^{[s,r]}_{H,\mathcal{O}_\mathfrak{X}}, \breve{\Pi}^r_{H,\mathcal{O}_\mathfrak{X}}, \breve{\Pi}_{H,\mathcal{O}_\mathfrak{X}}.
\end{align}	
\end{definition}

\indent Now we discuss some properties of the corresponding deformed version of the imperfect rings in our context, which is parallel to the corresponding discussion we made in the perfect setting and generalizing the corresponding discussion in \cite{KL16}:

\begin{proposition}\mbox{\bf{(After Kedlaya-Liu \cite[Lemma 5.2.10]{KL16})}}
For any $0< r_1\leq r_2$ we have the following equality on the corresponding period rings:
\begin{displaymath}
\Pi^{\mathrm{int},r_1}_{H,\mathbb{Q}_p\{T_1,...,T_d\}}\bigcap	\Pi^{[r_1,r_2]}_{H,\mathbb{Q}_p\{T_1,...,T_d\}}=\Pi^{\mathrm{int},r_2}_{H,\mathbb{Q}_p\{T_1,...,T_d\}}.
\end{displaymath}	
\end{proposition}

\begin{proof}
We adapt the argument in \cite[Lemma 5.2.10]{KL16} to prove this in the situation where $r_1<r_2$ (otherwise this is trivial), again one direction is easy where we only present the implication in the other direction. We take any element $x\in \Pi^{\mathrm{int},r_1}_{H,\mathbb{Q}_p\{T_1,...,T_d\}}\bigcap	\Pi^{[r_1,r_2]}_{H,\mathbb{Q}_p\{T_1,...,T_d\}}$ and take suitable approximating elements $\{x_i\}$ living in the bounded Robba ring such that for any $j\geq 1$ one can find some integer $N_j\geq 1$ we have whenever $i\geq N_j$ we have the following estimate:
\begin{displaymath}
\left\|.\right\|_{\overline{\alpha}_\infty^{t},\mathbb{Q}_p\{T_1,...,T_d\}}(x_i-x) \leq p^{-j}, \forall t\in [r_1,r_2].	
\end{displaymath}
Then we consider the corresponding decomposition of $x_i$ for each $i=1,2,...$ into a form having integral part and the rational part $x_i=y_i+z_i$ by setting
\begin{center}
 $y_i=\sum_{k=0,i_1,...,i_d}\pi^kx_{i,k,i_1,...,i_d}T_1^{i_1}...T_d^{i_d}$ 
\end{center} 
out of
\begin{center} 
$x_i=\sum_{k=n(x_i),i_1,...,i_d}\pi^kx_{i,k,i_1,...,i_d}T_1^{i_1}...T_d^{i_d}$.
\end{center}
Note that by our initial hypothesis we have that the element $x$ lives in the ring $\Pi^{\mathrm{int},r_1}_{H,\mathbb{Q}_p\{T_1,...,T_d\}}$ which further implies that 
\begin{displaymath}
\left\|.\right\|_{\overline{\alpha}_\infty^{r_1},\mathbb{Q}_p\{T_1,...,T_d\}}(\pi^kx_{i,k,i_1,...,i_d}T_1^{i_1}...T^{i_d}_d)	\leq p^{-j}.
\end{displaymath}
Therefore we have ${\overline{\alpha}_\infty}(\overline{x}_{i,k,i_1,...,i_d})\leq p^{(k-j)/r_1},\forall k< 0$ directly from this through computation, which implies that then:
\begin{align}
\left\|.\right\|_{\overline{\alpha}_\infty^{r_2},\mathbb{Q}_p\{T_1,...,T_d\}}(\pi^kx_{i,k,i_1,...,i_d}T_1^{i_1}...T^{i_d}_d)	&\leq p^{-k}p^{(k-j)r_2/r_1}\\
	&\leq p^{1+(1-j)r_1/r_1}.
\end{align}
Then one can read off the result directly from this estimate since under this estimate we can have the chance to modify the original approximating sequence $\{x_i\}$ by $\{y_i\}$ which are initially chosen to be in the integral Robba ring, which implies that actually the element $x$ lives in the right-hand side of the identity in the statement of the proposition.
\end{proof}

\begin{proposition} \mbox{\bf{(After Kedlaya-Liu \cite[Lemma 5.2.10]{KL16})}}\label{proposition5.7}
For any $0< r_1\leq r_2$ we have the following equality on the corresponding period rings:
\begin{displaymath}
\Pi^{\mathrm{int},r_1}_{H,A}\bigcap	\Pi^{[r_1,r_2]}_{H,A}=\Pi^{\mathrm{int},r_2}_{H,A}.
\end{displaymath}	
Here $A$ is some reduced affinoid algebra over $\mathbb{Q}_p$.	
\end{proposition}

\begin{proof}
This is actually not a direct corollary from the corresponding result as in the previous proposition. But since the map from the corresponding Tate algebra to $A$ is strict, which will remain to be so when tensor (completely) with the corresponding Robba rings in the previous proposition (see \cite[2.1.8, Proposition 6]{BGR}). Then for any element $\overline{x}$ in the intersection on the left hand side, one can consider the corresponding construction to any lift $x$ of $\overline{x}$ with the corresponding approximating estimate by some sequence:
\begin{displaymath}
\left\|.\right\|_{\overline{\alpha}_\infty^{t},\mathbb{Q}_p\{T_1,...,T_d\}}(x_i-x) \leq p^{-j}, \forall t\in [r_1,r_2],	
\end{displaymath}
for any $i\geq N_j$ with some $N_j$ when $j$ is arbitrarily chosen. Then we to each $x_i$ we associate $y_i$ and $z_i$ as in the proof of the previous proposition. Also for $x$ we have that it is living in the integral Robba ring which implies that we have:
\begin{displaymath}
\left\|.\right\|_{\overline{\alpha}_\infty^{r_1},\mathbb{Q}_p\{T_1,...,T_d\}}(\pi^kx_{i,k,i_1,...,i_d}T_1^{i_1}...T^{i_d}_d)	\leq p^{-j}.
\end{displaymath}
Therefore we have ${\overline{\alpha}_\infty}(\overline{x}_{i,k,i_1,...,i_d})\leq p^{(k-j)/r_1},\forall k< 0$ directly from this through computation, which implies that then:
\begin{align}
\left\|.\right\|_{\overline{\alpha}_\infty^{r_2},\mathbb{Q}_p\{T_1,...,T_d\}}(\pi^kx_{i,k,i_1,...,i_d}T_1^{i_1}...T^{i_d}_d)	&\leq p^{-k}p^{(k-j)r_2/r_1}\\
	&\leq p^{1+(1-j)r_1/r_1}.
\end{align}
This shows that actually one can rearrange the corresponding lifting to be some lifting with respect to the ring on the right hand side, which will finish the proof then by the strictness again after the corresponding projection.	
\end{proof}


\begin{proposition} \mbox{\bf{(After Kedlaya-Liu \cite[Lemma 5.2.8]{KL15} and \cite{KL16})}}
Consider now in our situation the radii $0< r_1\leq r_2$, and consider any element $x\in \Pi^{[r_1,r_2]}_{H,\mathbb{Q}_p\{T_1,...,T_d\}}$. Then we have that for each $n\geq 1$ one can decompose $x$ into the form of $x=y+z$ such that $y\in \pi^n\Pi^{\mathrm{int},r_2}_{H,\mathbb{Q}_p\{T_1,...,T_d\}}$ with $z\in \bigcap_{r\geq r_2}\Pi^{[r_1,r]}_{H,\mathbb{Q}_p\{T_1,...,T_d\}}$ with the following estimate for each $r\geq r_2$:
\begin{displaymath}
\left\|.\right\|_{\overline{\alpha}_\infty^r,\mathbb{Q}_p\{T_1,...,T_d\}}(z)\leq p^{(1-n)(1-r/r_2)}\left\|.\right\|_{\overline{\alpha}_\infty^{r_2},\mathbb{Q}_p\{T_1,...,T_d\}}(z)^{r/r_2}.	
\end{displaymath}

\end{proposition}

\begin{proof}
As in \cite[Lemma 5.2.8]{KL15} and \cite{KL16} and in the proof of our previous proposition we first consider those elements $x$ lives in the bounded Robba rings which could be expressed in general as
\begin{center}
 $\sum_{k=n(x),i_1,...,i_d}\pi^kx_{k,i_1,...,i_d}T_1^{i_1}...T_d^{i_d}$.
 \end{center}	
In this situation the corresponding decomposition is very easy to come up with, namely we consider the corresponding $y_i$ as the corresponding series:
\begin{displaymath}
\sum_{k\geq n,i_1,...,i_d}\pi^kx_{k,i_1,...,i_d}T_1^{i_1}...T_d^{i_d}	
\end{displaymath}
which give us the desired result since we have in this situation when focusing on each single term:
\begin{align}
\left\|.\right\|_{\overline{\alpha}_\infty^r,\mathbb{Q}_p\{T_1,...,T_d\}}(\pi^kx_{k,i_1,...,i_d}T_1^{i_1}...T_d^{i_d})&=p^{-k}\overline{\alpha}_\infty(\overline{x}_{k,i_1,...,i_d})^r\\
&=p^{-k(1-r/r_2)}\left\|.\right\|_{\overline{\alpha}_\infty^{r_2},\mathbb{Q}_p\{T_1,...,T_d\}}(\pi^kx_{k,i_1,...,i_d}T_1^{i_1}...T_d^{i_d})^{r/r_2}\\
&\leq p^{(1-n)(1-r/r_2)}\left\|.\right\|_{\overline{\alpha}_\infty^{r_2},\mathbb{Q}_p\{T_1,...,T_d\}}(\pi^kx_{k,i_1,...,i_d}T_1^{i_1}...T_d^{i_d})^{r/r_2}
\end{align}
for all those suitable $k$. Then to tackle the more general situation we consider the approximating sequence consisting of all the elements in the bounded Robba ring as in the usual situation considered in \cite[Lemma 5.2.8]{KL15} and \cite{KL16}, namely we inductively construct the following approximating sequence just as:
\begin{align}
\left\|.\right\|_{\overline{\alpha}_\infty^r,\mathbb{Q}_p\{T_1,...,T_d\}}(x-x_0-...-x_i)\leq p^{-i-1}	\left\|.\right\|_{\overline{\alpha}_\infty^r,\mathbb{Q}_p\{T_1,...,T_d\}}(x), i=0,1,..., r\in [r_1,r_2].
\end{align}
Here all the elements $x_i$ for each $i=0,1,...$ are living in the bounded Robba ring, which immediately gives rise to the suitable decomposition as proved in the previous case namely we have for each $i$ the decomposition $x_i=y_i+z_i$ with the desired conditions as mentioned in the statement of the proposition. We first take the series summing all the elements $y_i$ up for all $i=0,1,...$, which first of all converges under the norm $\left\|.\right\|_{\overline{\alpha}_\infty^r,\mathbb{Q}_p\{T_1,...,T_d\}}$ for all the radius $r\in [r_1,r_2]$, and also note that all the elements $y_i$ within the infinite sum live inside the corresponding integral Robba ring $\Pi^{\mathrm{int},r_2}_{H,\mathbb{Q}_p\{T_1,...,T_d\}}$, which further implies the corresponding convergence ends up in $\Pi^{\mathrm{int},r_2}_{H,\mathbb{Q}_p\{T_1,...,T_d\}}$. For the elements $z_i$ where $i=0,1,...$ also sum up to a converging series in the desired ring since combining all the estimates above we have:
\begin{displaymath}
\left\|.\right\|_{\overline{\alpha}_\infty^r,\mathbb{Q}_p\{T_1,...,T_d\}}(z_i)\leq p^{(1-n)(1-r/r_2)}\left\|.\right\|_{\overline{\alpha}_\infty^{r_2},\mathbb{Q}_p\{T_1,...,T_d\}}(x)^{r/r_2}.	
\end{displaymath}
\end{proof}

%

\begin{proposition} \mbox{\bf{(After Kedlaya-Liu \cite[Lemma 5.2.10]{KL15})}}
We have the following identity:
\begin{displaymath}
\Pi^{[s_1,r_1]}_{H,\mathbb{Q}_p\{T_1,...,T_d\}}\bigcap\Pi^{[s_2,r_2]}_{H,\mathbb{Q}_p\{T_1,...,T_d\}}=\Pi^{[s_1,r_2]}_{H,\mathbb{Q}_p\{T_1,...,T_d\}},
\end{displaymath}
here the radii satisfy $<s_1\leq s_2 \leq r_1 \leq r_2$.
\end{proposition}

\begin{proof}
In our situation one direction is obvious while on the other hand we consider any element $x$ in the intersection on the left, then by the previous proposition we	have the decomposition $x=y+z$ where $y\in \Pi^{\mathrm{int},r_1}_{H,\mathbb{Q}_p\{T_1,...,T_d\}}$ and $z\in \Pi^{[s_1,r_2]}_{H,\mathbb{Q}_p\{T_1,...,T_d\}}$. Then as in \cite[Lemma 5.2.10]{KL15} section 5.2 we look at $y=x-z$ which lives in the intersection:
\begin{displaymath}
\Pi^{\mathrm{int},r_1}_{H,\mathbb{Q}_p\{T_1,...,T_d\}}\bigcap	\Pi^{[s_2,r_2]}_{H,\mathbb{Q}_p\{T_1,...,T_d\}}=\Pi^{\mathrm{int},r_2}_{H,\mathbb{Q}_p\{T_1,...,T_d\}}
\end{displaymath}
which finishes the proof.
\end{proof}

\begin{proposition} \mbox{\bf{(After Kedlaya-Liu \cite[Lemma 5.2.10]{KL15})}}
We have the following identity:
\begin{displaymath}
\Pi^{[s_1,r_1]}_{H,A}\bigcap \Pi^{[s_2,r_2]}_{H,A}=\Pi^{[s_1,r_2]}_{H,A},
\end{displaymath}
here the radii satisfy $<s_1\leq s_2 \leq r_1 \leq r_2$.
	
\end{proposition}

\begin{proof}
See the proof of \cref{proposition5.7}.	
\end{proof}

\begin{remark}
Again one can follow the same strategy to deal with the corresponding equal-characteristic situation.	
\end{remark}

\subsection{Modules and Bundles}

Now we consider the modules and bundles over the rings introduced in the previous subsection. First we make the following assumption:

\begin{setting}
Recall that from \cite[Definition 5.2.3]{KL16} any tower $(H_\bullet,H_\bullet^+)$	is called weakly decompleting if we have that first the density of the perfection of $H_{\infty}$ in $\overline{H}_\infty$. Here the ring $H_\infty$ is the ring coming from the mod-$\pi$ construction of the ring $\Omega^\mathrm{int}_{H}$, also at the same time one can find some $r>0$ such that the corresponding modulo $\pi$ operation from the ring $\Omega^\mathrm{int}_{H}$ to the ring $H_\infty$ is actually surjective strictly. 
\end{setting}

\begin{assumption}
We now assume that we are basically in the situation where $(H_\bullet,H_\bullet^+)$ is actually weakly decompleting. Also as in \cite[Lemma 5.2.7]{KL16} we assume we fix some radius $r_0>0$, for instance this will correspond to the corresponding index in the situation we have the corresponding noetherian tower. Recall that a tower is called noetherian if we have some specific radius as above such that we have the strongly noetherian property on the ring $\Pi^{[s,r]}_{H}$ with $[s,r]\subset (0,r_0]$. Under this condition we have that actually we have the corresponding strongly noetherian property on the ring $\Pi^{[s,r]}_{H,A}$ with $[s,r]\subset (0,r_0]$, therefore consequently we have that the ring $(\Pi^{[s,r]}_{H,A},\Pi^{[s,r],+}_{H,A})$ is sheafy. Here the ring $\Pi^{[s,r],+}_{H,A}$ is defined by taking the product construction between $A$ and the ring $\Pi^{[s,r],+}_{H}$ which is defined in \cite[Definition 5.3.2]{KL16}. We now assume that the tower is then noetherian. For more examples, see \cite[5.3.3]{KL16}.
\end{assumption}

\indent Then we can start to discuss the corresponding modules over the rings in our deformed setting, first as in \cite[Lemma 5.3.3]{KL16} the following result should be derived from our construction:

\begin{lemma}\mbox{\bf{(After Kedlaya-Liu \cite[Lemma 5.3.3]{KL16})}}
We have the following strict isomorphism in our setting, where the corresponding notations are as in \cite[Lemma 5.3.3]{KL16}:
\begin{align}
\Pi^{[s,r]}_{H,A}\{X/p^{-t}\}/(\pi X-1)\rightarrow \Pi^{[t,r]}_{H,A}\\
\Pi^{[s,r]}_{H,A}\{X/p^{-t}\}/(X-\pi)\rightarrow \Pi^{[s,t]}_{H,A}\\
\Pi^{r}_{H,A}\{X/p^{-s}\}/(X-\pi)\rightarrow \Pi^{s}_{H,A}\\
\Pi^{r}_{H,A}\{X/p^{-s}\}/(\pi X-1)\rightarrow \Pi^{[s,r]}_{H,A},
\end{align}
where the corresponding radii satisfy $0<s\leq t\leq r\leq r_{0}$.
	
\end{lemma}

\begin{proof}
See \cite[Lemma 5.3.3]{KL16}.	
\end{proof}

We then as in \cite[Lemma 5.3.4]{KL16} have the following:

\begin{lemma} \mbox{\bf{(After Kedlaya-Liu \cite[Lemma 5.3.4]{KL16})}}
We have the 2-pseudoflatness of the following maps:
\begin{align}
\Pi^{[r_1,r_2]}_{H,A}\rightarrow\Pi^{[r_1,t]}_{H,A}, \Pi^{[r_1,r_2]}_{H,A}\rightarrow\Pi^{[t,r_2]}_{H,A} 	
\end{align}
where we have $0<r_1\leq t\leq r_2$.
	
\end{lemma}

\begin{proof}
See \cite[Lemma 5.3.4]{KL16}.	
\end{proof}

\begin{remark}
In effect, one can make more strong statement here, due to the fact that we are working in the noetherian setting. To be more precise see the corresponding result (and the proof) of \cite[Corollary 2.6.9]{KL16}, one can actually have even the flatness of the corresponding maps in the previous lemma as long as we are working over noetherian rings.	
\end{remark}

\indent Then we have the following implication:

\begin{proposition}\label{proposition6.15} \mbox{\bf{(After Kedlaya-Liu \cite[Proposition 5.3.5]{KL16})}} 
Suppose that we have a module $M$ over the ring $\Pi_{H,A}^{I}$ (which is assumed to be stably pseudocoherent) where the closed interval $I$ admits a covering by finite many closed intervals $I=\cup_{i}I_i$. Then we have in our situation the following exact augmented \v{C}ech complex:
\begin{displaymath}
0\rightarrow M \rightarrow \bigoplus_{i} M_i \rightarrow...,	
\end{displaymath}
where the module $M_i$ is defined to be the base change of the module $M$ to the ring $\Pi_{H,A}^{I_i}$ with respect to each $i$. Also we have that the morphism $\Pi_{H,A}^{I}\rightarrow \bigoplus_i \Pi_{H,A}^{I_i}$ is in our situation effective descent with respect to the corresponding categories of \'etale-stably pseudocoherent modules in the Banach setting and with respect to the corresponding categories of finite projective modules in the Banach setting. 
\end{proposition}

\begin{proof}
As in \cite[Proposition 5.3.5]{KL16} we only have to to look at the situation of two intervals. Then as in \cite[Proposition 5.3.5]{KL16} by using the previous lemma one can argue as in \cite[Proposition 5.3.5]{KL16}.	
\end{proof}

\indent Then we deform the basic notation of bundles in \cite[Definition 5.3.6]{KL16}:

\begin{definition} \mbox{\bf{(After Kedlaya-Liu \cite[Definition 5.3.6]{KL16})}}
We define the bundle over the ring $\Pi^{r_0}_{H,A}$ to be a collection $(M_I)_I$ of finite projective modules over each $\Pi_{H,A}^{I}$ with $I\subset (0,r_0]$ closed subintervals of $(0,r_0]$ such that we have the following requirement in the glueing fashion. First for any $I_1\subset I_2$ two closed intervals we have $M_{I_2}\otimes_{\Pi_{H,A}^{I_2}}\Pi_{H,A}^{I_1}\overset{\sim}{\rightarrow} M_{I_1}$ with the obvious cocycle condition with respect to three closed subintervals of $(0,r_0]$ namely taking the form of $I_1\subset I_2\subset I_3$.\\
\indent We define the pseudocoherent sheaf over the ring $\Pi^{r_0}_{H,A}$ to be a collection $(M_I)_I$ of \'etale-stably pseudocoherent modules over each $\Pi_{H,A}^{I}$ with $I\subset (0,r_0]$ closed subintervals of $(0,r_0]$ such that we have the following requirement in the glueing fashion. First for any $I_1\subset I_2$ two closed intervals we have $M_{I_2}\otimes_{\Pi_{H,A}^{I_2}}\Pi_{H,A}^{I_1}\overset{\sim}{\rightarrow} M_{I_1}$ with the obvious cocycle condition with respect to three closed subintervals of $(0,r_0]$ namely taking the form of $I_1\subset I_2\subset I_3$. 	
\end{definition}

\indent We make the following discussion around the corresponding module and sheaf structures defined above.

\begin{lemma} \mbox{\bf{(After Kedlaya-Liu \cite[Lemma 5.3.8]{KL16})}}
We have the isomorphism between the ring $\Pi^r_{H,A}$ and the inverse limit of the ring $\Pi^{[s,r]}_{H,A}$ with respect to the radius $s$ by the map $\Pi^r_{H,A}\rightarrow \Pi^{[s,r]}_{H,A}$.	
\end{lemma}
 
\begin{proof}
As in \cite[Lemma 5.3.8]{KL16} it is injective by the isometry, and then use the corresponding elements $x_n,n=0,1,...$ in the dense ring $\Pi^{r}_{H,A}$ to approximate any element $x$ in the ring $\Pi^{[s,r]}_{H,A}$ in the same way as in \cite[Lemma 5.3.8]{KL16}:
\begin{displaymath}
\|.\|_{\overline{\alpha}_\infty^t,A}(x-x_n)\leq p^{n}	
\end{displaymath}
for any radius $t$ now living in the corresponding interval $[r2^{-n},r]$. This will establish Cauchy sequence which finishes the proof as in \cite[Lemma 5.3.8]{KL16}.
\end{proof}

\begin{proposition} \mbox{\bf{(After Kedlaya-Liu \cite[Lemma 5.3.9]{KL16})}}
For some radius $r\in (0,r_0]$. Suppose we have that $M$ is a vector bundle in the general setting or $M$ is a pseudocoherent sheaf in the setting where the tower is noetherian. Then we have that the corresponding global section is actually dense in each section with respect to some closed interval. And then we have the corresponding vanishing result of the first derived inverse limit functor.	
\end{proposition}

\begin{proof}
See \cite[Lemma 5.3.9]{KL16}.	
\end{proof}

\indent The interesting issue here as in \cite{KL16} is the corresponding finitely generatedness of the global section of a pseudocoherent sheaf which is actually not guaranteed in general. Therefore as in \cite{KL16} we have to distinguish the corresponding well-behaved sheaves out from the corresponding category of all the corresponding pseudocoherent sheaves.

\begin{proposition} \mbox{\bf{(After Kedlaya-Liu \cite[Lemma 5.3.10]{KL16})}}
As in the previous proposition we choose some $r\in (0,r_0]$. Now assume that the corresponding tower $(H_\bullet,H_\bullet^+)$ is noetherian. Now for any pseudocoherent sheaf $M$ defined above we have the following three statements are equivalent. A. The first statement is that one can find a sequence of positive integers $x_1,x_2,...$ such that for any closed interval living inside $(0,r]$ the section of the sheaf with respect this closed interval admits a projective resolution of modules with corresponding ranks bounded by the sequence of integer $x_1,x_2,...$. B. The second statement is that for any locally finite covering of the corresponding interval $(0,r]$ which takes the corresponding form of $\{I_i\}$ one can find a sequence of positive integers $x_1,x_2,...$ such that for any closed interval living inside $\{I_i\}$ the section of the sheaf with respect this closed interval admits a projective resolution of modules with corresponding ranks bounded by the sequence of integer $x_1,x_2,...$. C. Lastly the third statement is that the corresponding global section is a pseudocoherent module over the ring $\Pi^r_{H,A}$. 	
\end{proposition}

\begin{proof}
See \cite[Lemma 5.3.10]{KL16}.	
\end{proof}

\indent As in \cite[Definition 5.3.11]{KL16} we call the sheaf satisfying the corresponding equivalent conditions in the proposition above uniform pseudocoherent sheaf. Then we have the following analog of \cite[Lemma 5.3.12]{KL16}:

\begin{proposition} \mbox{\bf{(After Kedlaya-Liu \cite[Lemma 5.3.12]{KL16})}}
The global section functor defines the corresponding equivalence between the categories of the following two sorts of objects. The first ones are the corresponding uniform pseudocoherent sheaves over $\Pi^r_{H,A}$. The second ones are those pseudocoherent modules over the ring $\Pi^r_{H,A}$. 	
\end{proposition}

\begin{proof}
See \cite[Lemma 5.3.12]{KL16}.	
\end{proof}

\subsection{Frobenius Structure and $\Gamma$-Structure on Hodge-Iwasawa Modules}

\indent Now we consider the corresponding Frobenius actions over the corresponding imperfect rings we defined before, note that the corresponding Frobenius actions are induced from the corresponding imperfect rings in the undeformed situation from \cite{KL16} which is to say that the Frobenius action on the ring $A$ is actually trivial.

\indent First we consider the corresponding Frobenius modules:

\begin{definition} \mbox{\bf{(After Kedlaya-Liu \cite[Definition 5.4.2]{KL16})}}
Over the period rings $\Pi_{H,A}$ or $\breve{\Pi}_{H,A}$  (which is denoted by $\triangle$ in this definition) we define the corresponding $\varphi^a$-modules over $\triangle$ which are respectively projective, pseudocoherent or fpd to be the corresponding finite projective, pseudocoherent or fpd modules over $\triangle$ with further assigned semilinear action of the operator $\varphi^a$. We also require that the modules are complete for the natural topology involved in our situation and for any module over $\Pi_{H,A}$ to be some base change of some module over $\Pi^r_{H,A}$ (which will be defined in the following). We also require that the modules are complete for the natural topology involved in our situation and any module over $\widetilde{\Pi}_{H,A}$ to be some base change of some module over $\widetilde{\Pi}^r_{H,A}$ (which will be defined in the following).
\end{definition}

\begin{definition} \mbox{\bf{(After Kedlaya-Liu \cite[Definition 5.4.2]{KL16})}}
Over each rings $\triangle=\Pi^r_{H,A},\breve{\Pi}^r_{H,A}$ we define the corresponding projective, pseudocoherent or fpd $\varphi^a$-module over any $\triangle$ to be the corresponding finite projective, pseudocoherent or fpd module $M$ over $\triangle$ with additionally endowed semilinear Frobenius action from $\varphi^a$ such that we have the isomorphism $\varphi^{a*}M\overset{\sim}{\rightarrow}M\otimes \square$ where the ring $\square$ is one $\triangle=\Pi^{r/p}_{H,A},\breve{\Pi}^{r/p}_{H,A}$. Also as in \cite[Definition 5.4.2]{KL16} we assume that the module over $\Pi^r_{H,A}$ is then complete for the natural topology and the corresponding base change to $\Pi^I_{H,A}$ for any interval which is assumed to be closed $I\subset [0,r)$ gives rise to a module over $\Pi^I_{H,A}$ with specified conditions which will be specified below. Also as in \cite[Definition 5.4.2]{KL16} we assume that the module over $\widetilde{\Pi}^r_{H,A}$ is then complete for the natural topology and the corresponding base change to $\widetilde{\Pi}^I_{H,A}$ for any interval which is assumed to be closed $I\subset [0,r)$ gives rise to a module over $\widetilde{\Pi}^I_{H,A}$ with specified conditions which will be specified below.

\end{definition}

\begin{definition} \mbox{\bf{(After Kedlaya-Liu \cite[Definition 5.4.2]{KL16})}}
Again as in \cite[Definition 5.4.2]{KL16}, we define the corresponding projective, pseudocoherent and fpd $\varphi^a$-modules over ring $\Pi^{[s,r]}_{H,A}$ or $\breve{\Pi}^{[s,r]}_{H,A}$  to be the finite projective, pseudocoherent and fpd modules (which will be denoted by $M$) over $\Pi^{[s,r]}_{H,A}$ or $\breve{\Pi}^{[s,r]}_{H,A}$ respectively additionally endowed with semilinear Frobenius action from $\varphi^a$ with the following isomorphisms:
\begin{align}
\varphi^{a*}M\otimes_{\Pi_{H,A}^{[sp^{-ah},rp^{-ah}]}}\Pi_{H,A}^{[s,rp^{-ah}]}\overset{\sim}{\rightarrow}M\otimes_{\Pi_{H,A}^{[s,r]}}\Pi_{H,A}^{[s,rp^{-ah}]}
\end{align}
and
\begin{align}
\varphi^{a*}M\otimes_{\breve{\Pi}_{R,A}^{[sp^{-ah},rp^{-ah}]}}\breve{\Pi}_{R,A}^{[s,rp^{-ah}]}\overset{\sim}{\rightarrow}M\otimes_{\breve{\Pi}_{R,A}^{[s,r]}}\breve{\Pi}_{R,A}^{[s,rp^{-ah}]}.\\
\end{align}
We now assume that the modules are complete with respect to the natural topology and \'etale stably pseudocoherent.
\end{definition}

\noindent Also one can further define the corresponding bundles carrying semilinear Frobenius in our context as in the situation of \cite[Definition 5.4.10]{KL16}:

\begin{definition} \mbox{\bf{(After Kedlaya-Liu \cite[Definition 5.4.10]{KL16})}}
Over the ring $\Pi^r_{H,A}$ or $\breve{\Pi}^r_{H,A}$ we define a corresponding projective, pseudocoherent and fpd Frobenius bundle to be a family $(M_I)_I$ of finite projective, \'etale stably pseudocoherent and \'etale stably fpd modules over each $\Pi^I_{H,A}$ or $\breve{\Pi}^I_{H,A}$ respectively carrying the natural Frobenius action coming from the operator $\varphi^a$ such that for any two involved intervals having the relation $I\subset J$ we have:
\begin{displaymath}
M_J\otimes_{\Pi^J_{H,A}}\Pi^I_{H,A}\overset{\sim}{\rightarrow}	M_I
\end{displaymath}
and 
\begin{displaymath}
M_J\otimes_{\breve{\Pi}^J_{H,A}}\breve{\Pi}^I_{H,A}\overset{\sim}{\rightarrow}	M_I
\end{displaymath}
with the obvious cocycle condition. Here we have to propose condition on the intervals that for each $I=[s,u]$ involved we have $s\leq u/p^{ah}$. Then as in \cite[Definition 5.4.10]{KL16}, we can consider the corresponding direct 2-limit to achieve the corresponding objects in the categories over full Robba rings.
\end{definition}

\indent We can then compare the corresponding objects defined above:

\begin{proposition} \mbox{\bf{(After Kedlaya-Liu \cite[Lemma 5.4.11]{KL16})}}\\
I. Consider the following objects for some radius $r_0$ in our situation. The first group of objects are those finite projective $\varphi^a$-modules over the Robba ring $\Pi^{r_0}_{H,A}$. The second group of objects are those finite projective $\varphi^a$-bundles over the Robba ring $\Pi^{r_0}_{H,A}$. The third group of objects are those finite projective $\varphi^a$-modules over the Robba ring $\Pi^{[s,r]}_{H,A}$ for some $[s,r]\in (0,r_0)$. Then we have that the corresponding categories of the three groups of objects are equivalent. \\
II. Consider the following objects for some radius $r_0$ in our situation. The first group of objects are those pseudocoherent $\varphi^a$-modules over the Robba ring $\Pi^{r_0}_{H,A}$. The second group of objects are those pseudocoherent $\varphi^a$-bundles over the Robba ring $\Pi^{r_0}_{H,A}$. The third group of objects are those pseudocoherent $\varphi^a$-modules over the Robba ring $\Pi^{[s,r]}_{H,A}$ for some $[s,r]\in (0,r_0)$. Then we have that the corresponding categories of the three groups of objects are equivalent. \\
III. Consider the following objects for some radius $r_0$ in our situation. The first group of objects are those finite projective dimensional $\varphi^a$-modules over the Robba ring $\Pi^{r_0}_{H,A}$. The second group of objects are those finite projective dimensional $\varphi^a$-bundles over the Robba ring $\Pi^{r_0}_{H,A}$. The third group of objects are those finite projective dimensional $\varphi^a$-modules over the Robba ring $\Pi^{[s,r]}_{H,A}$ for some $[s,r]\in (0,r_0)$. Then we have that the corresponding categories of the three groups of objects are equivalent.  	
\end{proposition}

\begin{proof}
See the proof of \cite[Lemma 5.4.11]{KL16}. In our situation we need to use the corresponding \cref{proposition6.15} to prove the corresponding equivalences between the corresponding categories of bundles in the second groups and modules in the third groups. The corresponding equivalences between the corresponding categories of bundles in the second groups and modules in the first groups are following \cite[Lemma 5.4.11]{KL16} after applying the Frobenius pullbacks to compare those sections over different intervals, and note that we have the corresponding uniform pseudocoherent objects in our development above which actually restricts further to the finite projective objects as in \cite[Proposition 2.7.16]{KL16}.
\end{proof}

\indent Now we define the corresponding $\Gamma$-modules over the period rings attached to the tower $(H_\bullet,H_\bullet^+)$. The corresponding structures are actually abstractly defined in the same way as in \cite{KL16}. First we consider the deformation of the corresponding complex $*_{H^\bullet}$ for any ring $*$ in \cref{setting6.4}.

\begin{assumption}
Recall that the corresponding tower is called decompleting if it is weakly decompleting, finite \'etale on each finite level and having the exact sequence $\overline{\varphi}^{-1}H'_{H^\bullet}/H'_{H^\bullet}$ is exact. We now assume that the tower $(H_\bullet,H^+_\bullet)$ is then decompleting.	
\end{assumption}


\begin{setting} 
Assume now $\Gamma$ is a topological group as in \cite[Definition 5.5.5]{KL16} acting on the corresponding period rings in the \cref{setting6.3} in the original context of \cref{setting6.3}. Then we consider the corresponding induced continous action over the corresponding deformed version in our context namely in \cref{setting6.4}. Assume now that the tower is Galois with the corresponding Galois group $\Gamma$.	
\end{setting}

\begin{definition}
We now consider the corresponding inhomogeneous continuous cocycles of the group $\Gamma$, as in \cite[Definition 5.5.5]{KL16} we use the following notation to denote the corresponding complex extracted from a single tower for a given period ring $*_{H,A}$ in \cref{setting6.4} for each $k>0$:
\begin{displaymath}
*_{H^k,A}:=C_\mathrm{con}(\Gamma^k,*_{H})\widehat{\otimes}_{\sharp} ?
\end{displaymath}
where $\sharp=\mathbb{Q}_p,\mathbb{Z}_p$ and $?=A,\mathfrak{o}_A$ respectively, which forms the corresponding complex $(*_{H^\bullet,A},d^\bullet)$ with the corresponding differential as in \cite[Definition 5.5.5]{KL16} in the sense of continuous group cohomology.	
\end{definition}

\begin{definition}
Having established the corresponding meaning of the $\Gamma$-structure we now consider the corresponding definition of $\Gamma$-modules. Such modules called the corresponding $\Gamma$-modules are defined over the corresponding rings in \cref{setting6.4}. Again we allow the corresponding modules to be finite projective, or pseudocoherent or fpd over the rings in \cref{setting6.4}. And the modules are defined to be carrying the corresponding continuous semilinear action from the group $\Gamma$.	
\end{definition}

\begin{proposition} \mbox{\bf{(After Kedlaya-Liu \cite[Corollary 5.6.5]{KL16})}}
The complex $\varphi^{-1}\Pi^{[sp^{-h},rp^{-h}]}_{H^\bullet_{\geq n},A}/\Pi^{[s,r]}_{H^\bullet_{\geq n},A}$ and the complex $\widetilde{\Pi}^{[s,r]}_{H^\bullet_{\geq n},A}/\Pi^{[s,r]}_{H^\bullet_{\geq n},A}$ are strict exact for any truncation index $n$. The corresponding radii satisfy the corresponding relation $0<s\leq r\leq r_0$.
\end{proposition}

\begin{proof}
See \cite[Corollary 5.6.5]{KL16}, and consider the corresponding Schauder basis of $A$.	
\end{proof}

\begin{proposition} \mbox{\bf{(After Kedlaya-Liu \cite[Lemma 5.6.6]{KL16})}}
The complex 
\begin{displaymath}
M\otimes_{\Pi^{[s,r]}_{H,A}} \varphi^{-(\ell+1)}\Pi^{[sp^{-h(\ell+1)},rp^{-h(\ell+1)}]}_{H^\bullet,A}/ \varphi^{-\ell}\Pi^{[sp^{-h\ell},rp^{-h\ell}]}_{H^\bullet,A}	
\end{displaymath}
and the complex 
\begin{displaymath}
M\otimes_{\Pi^{[s,r]}_{H,A}} \widetilde{\Pi}^{[s,r]}_{H^\bullet,A}/\varphi^{-\ell}\Pi^{[sp^{-h(\ell)},rp^{-h(\ell)}]}_{H^\bullet,A}
\end{displaymath}
are strict exact for any truncation index $n$. The corresponding radii satisfy the corresponding relation $0<s\leq r\leq r_0$, and $\ell$ is bigger than some existing truncated integer $\ell_0\geq  0$. Here $M$ is any $\Gamma$-module in our context.
\end{proposition}

\begin{proof}
See \cite[Lemma 5.6.6]{KL16}.	
\end{proof}

\begin{proposition} \mbox{\bf{(After Kedlaya-Liu \cite[Lemma 5.6.9]{KL16})}}
With the corresponding notations as above we have that the corresponding base change from the ring taking the form of $\breve{\Pi}^{[s,r]}_{H,A}$ to the ring taking the form of $\widetilde{\Pi}^{[s,r]}_{H,A}$ establishes the corresponding equivalence on the corresponding categories of $\Gamma$-modules.	
\end{proposition}

\begin{proof}
This is a relative version of the corresponding result in \cite[Lemma 5.6.9]{KL16} we adapt the corresponding argument here. Indeed the corresponding fully faithfulness comes from the previous proposition the idea to prove the corresponding essential surjectivity comes from writing the module $M$ (the corresponding differential) over the ring taking form of $\widetilde{\Pi}^{[s,r]}_{H,A}$ as the base change from $\varphi^{-k}(\Pi^{[s,r]}_{H,A})$ of a module $M_0$ (the corresponding differential) after the corresponding \cite[Lemma 5.6.8]{KL16}.	Then as in \cite{KL16} we consider the corresponding norms on the corresponding $M$ and the base change of $M_0$ which could be controlled up to some constant which could be further modified to be zero by reducing each time positive amount of constant from the constant represented by the difference of the norms.
\end{proof}

\begin{definition} \mbox{\bf{(After Kedlaya-Liu \cite[Definition 5.7.2]{KL16})}}
Over the period rings $\Pi_{H,A}$ or $\breve{\Pi}_{H,A}$  (which is denoted by $\triangle$ in this definition) we define the corresponding $(\varphi^a,\Gamma)$-modules over $\triangle$ which are respectively projective, pseudocoherent or fpd to be the corresponding finite projective, pseudocoherent or fpd $\Gamma$-modules over $\triangle$ with further assigned semilinear action of the operator $\varphi^a$ with the isomorphism defined by using the Frobenius. We also require that the modules are complete for the natural topology involved in our situation and for any module over $\Pi_{H,A}$ to be some base change of some module over $\Pi^r_{H,A}$ (which will be defined in the following). 
\end{definition}

\begin{definition} \mbox{\bf{(After Kedlaya-Liu \cite[Definition 5.7.2]{KL16})}}
Over each rings $\triangle=\Pi^r_{H,A},\breve{\Pi}^r_{H,A}$ we define the corresponding projective, pseudocoherent or fpd $(\varphi^a,\Gamma)$-module over any $\triangle$ to be the corresponding finite projective, pseudocoherent or fpd $\Gamma$-module $M$ over $\triangle$ with additionally endowed semilinear Frobenius action from $\varphi^a$ such that we have the isomorphism $\varphi^{a*}M\overset{\sim}{\rightarrow}M\otimes \square$ where the ring $\square$ is one $\triangle=\Pi^{r/p}_{H,A},\breve{\Pi}^{r/p}_{H,A}$. Also as in \cite[Definition 5.7.2]{KL16} we assume that the module over $\Pi^r_{H,A}$ is then complete for the natural topology and the corresponding base change to $\Pi^I_{H,A}$ for any interval which is assumed to be closed $I\subset [0,r)$ gives rise to a module over $\Pi^I_{H,A}$ with specified conditions which will be specified below. 

\end{definition}

\begin{definition} \mbox{\bf{(After Kedlaya-Liu \cite[Definition 5.7.2]{KL16})}}
Again as in \cite[Definition 5.7.2]{KL16}, we define the corresponding projective, pseudocoherent and fpd $(\varphi^a,\Gamma)$-modules over ring $\Pi^{[s,r]}_{H,A}$ or $\breve{\Pi}^{[s,r]}_{H,A}$ to be the finite projective, pseudocoherent and fpd $\Gamma$-modules (which will be denoted by $M$) over $\Pi^{[s,r]}_{H,A}$ or $\breve{\Pi}^{[s,r]}_{H,A}$ respectively additionally endowed with semilinear Frobenius action from $\varphi^a$ with the following isomorphisms:
\begin{align}
\varphi^{a*}M\otimes_{\Pi_{H,A}^{[sp^{-ah},rp^{-ah}]}}\Pi_{H,A}^{[s,rp^{-ah}]}\overset{\sim}{\rightarrow}M\otimes_{\Pi_{H,A}^{[s,r]}}\Pi_{H,A}^{[s,rp^{-ah}]}
\end{align}
and
\begin{align}
\varphi^{a*}M\otimes_{\breve{\Pi}_{R,A}^{[sp^{-ah},rp^{-ah}]}}\breve{\Pi}_{R,A}^{[s,rp^{-ah}]}\overset{\sim}{\rightarrow}M\otimes_{\breve{\Pi}_{R,A}^{[s,r]}}\breve{\Pi}_{R,A}^{[s,rp^{-ah}]}.\\
\end{align}
We also require the corresponding topological conditions as we considered in the Frobenius module situation. 
\end{definition}

\begin{definition} \mbox{\bf{(After Kedlaya-Liu \cite[Definition 5.7.2]{KL16})}}
Over the ring $\Pi^r_{H,A}$ or $\breve{\Pi}^r_{H,A}$ we define a corresponding projective, pseudocoherent and fpd $(\varphi^a,\Gamma)$ bundle to be a family $(M_I)_I$ of finite projective, pseudocoherent and fpd $\Gamma$-modules over each $\Pi^I_{H,A}$ or $\breve{\Pi}^I_{H,A}$ respectively carrying the natural Frobenius action coming from the operator $\varphi^a$ such that for any two involved intervals having the relation $I\subset J$ we have:
\begin{displaymath}
M_J\otimes_{\Pi^J_{H,A}}\Pi^I_{H,A}\overset{\sim}{\rightarrow}	M_I
\end{displaymath}
and 
\begin{displaymath}
M_J\otimes_{\breve{\Pi}^J_{H,A}}\breve{\Pi}^I_{H,A}\overset{\sim}{\rightarrow}	M_I
\end{displaymath}
with the obvious cocycle condition. Here we have to propose condition on the intervals that for each $I=[v,u]$ involved we have $v\leq u/p^{ah}$. We put the corresponding topological conditions as before when we consider the corresponding Frobenius bundles. And one can take the corresponding 2-limit in the direct sense to define the corresponding objects over the full Robba rings.
\end{definition}

\begin{definition} \mbox{\bf{(After Kedlaya-Liu \cite[Definition 5.7.2]{KL16})}}
Over the period rings $\widetilde{\Pi}_{H,A}$ (which is denoted by $\triangle$ in this definition) we define the corresponding $(\varphi^a,\Gamma)$-modules over $\triangle$ which are respectively projective, pseudocoherent or fpd to be the corresponding finite projective, pseudocoherent or fpd $\Gamma$-modules over $\triangle$ with further assigned semilinear action of the operator $\varphi^a$ with the isomorphism defined by using the Frobenius. We also require that the modules are complete for the natural topology involved in our situation and for any module over $\widetilde{\Pi}_{H,A}$ to be some base change of some module over $\widetilde{\Pi}^r_{H,A}$ (which will be defined in the following).
\end{definition}

\begin{definition} \mbox{\bf{(After Kedlaya-Liu \cite[Definition 5.7.2]{KL16})}}
Over each ring $\triangle=\widetilde{\Pi}^r_{H,A}$ we define the corresponding projective, pseudocoherent or fpd $(\varphi^a,\Gamma)$-module over any $\triangle$ to be the corresponding finite projective, pseudocoherent or fpd $\Gamma$-module $M$ over $\triangle$ with additionally endowed semilinear Frobenius action from $\varphi^a$ such that we have the isomorphism $\varphi^{a*}M\overset{\sim}{\rightarrow}M\otimes \square$ where the ring $\square$ is one $\triangle=\widetilde{\Pi}^{r/p}_{H,A}$. Also as in \cite[Definition 5.7.2]{KL16} we assume that the module over $\widetilde{\Pi}^r_{H,A}$ is then complete for the natural topology and the corresponding base change to $\widetilde{\Pi}^I_{H,A}$ for any interval which is assumed to be closed $I\subset [0,r)$ gives rise to a module over $\widetilde{\Pi}^I_{H,A}$ with specified conditions which will be specified below.

\end{definition}

\begin{definition} \mbox{\bf{(After Kedlaya-Liu \cite[Definition 5.7.2]{KL16})}}
Again as in \cite[Definition 5.7.2]{KL16}, we define the corresponding projective, pseudocoherent and fpd $(\varphi^a,\Gamma)$-modules over ring $\widetilde{\Pi}^{[s,r]}_{H,A}$ to be the finite projective, pseudocoherent and fpd $\Gamma$-modules (which will be denoted by $M$) over $\widetilde{\Pi}^{[s,r]}_{H,A}$ additionally endowed with semilinear Frobenius action from $\varphi^a$ with the following isomorphisms:
\begin{align}
\varphi^{a*}M\otimes_{\widetilde{\Pi}_{H,A}^{[sp^{-ah},rp^{-ah}]}}\widetilde{\Pi}_{H,A}^{[s,rp^{-ah}]}\overset{\sim}{\rightarrow}M\otimes_{\widetilde{\Pi}_{H,A}^{[s,r]}}\widetilde{\Pi}_{H,A}^{[s,rp^{-ah}]}.
\end{align}
We also require the corresponding topological conditions as we considered in the Frobenius module situation. 
\end{definition}

\begin{definition} \mbox{\bf{(After Kedlaya-Liu \cite[Definition 5.7.2]{KL16})}}
Over the ring $\widetilde{\Pi}^r_{H,A}$ we define a corresponding projective, pseudocoherent and fpd $(\varphi^a,\Gamma)$ bundle to be a family $(M_I)_I$ of finite projective, pseudocoherent and fpd $\Gamma$-modules over each $\widetilde{\Pi}^I_{H,A}$ carrying the natural Frobenius action coming from the operator $\varphi^a$ such that for any two involved intervals having the relation $I\subset J$ we have:
\begin{displaymath}
M_J\otimes_{\widetilde{\Pi}^J_{H,A}}\widetilde{\Pi}^I_{H,A}\overset{\sim}{\rightarrow}	M_I
\end{displaymath}
with the obvious cocycle condition. Here we have to propose condition on the intervals that for each $I=[v,u]$ involved we have $v\leq u/p^{ah}$. We put the corresponding topological conditions as before when we consider the corresponding Frobenius bundles. And one can take the corresponding 2-limit in the direct sense to define the corresponding objects over the full Robba rings.
\end{definition}

\begin{remark}
In the following proposition, we assume that the corresponding ring $\widetilde{\Pi}^{[s,r]}_{\overline{H}'_\infty,A}$ is sheafy.	
\end{remark}

\begin{proposition} \mbox{\bf{(After Kedlaya-Liu \cite[Theorem 5.7.5]{KL16})}}
We have now the following categories are equivalence for the corresponding radii $0< s\leq r\leq r_0$ (with the further requirement as in \cite[Theorem 5.7.5]{KL16} that $s\in (0,r/q]$):\\
1. The category of all the finite projective sheaves over the ring $\widetilde{\Pi}_{\mathrm{Spa}(H_0,H_0^+),A}$, carrying the $\varphi^a$-action;\\
2. The category of all the finite projective sheaves over the ring $\widetilde{\Pi}^r_{\mathrm{Spa}(H_0,H_0^+),A}$, carrying the $\varphi^a$-action;\\
3. The category of all the finite projective sheaves over the ring $\widetilde{\Pi}^{[s,r]}_{\mathrm{Spa}(H_0,H_0^+),A}$, carrying the $\varphi^a$-action.\\
\indent Then we have the second group of categories which are equvalent:\\
4. The category of all the finite projective quasi-coherent sheaves over corresponding adic Fargues-Fontaine curve in the deformed setting $\mathrm{FF}_{\overline{H}'_\infty,A}$, carrying the corresponding action from the group $\Gamma$ which is assumed to be semilinear and continuous over each section over any affinoid subspace of the whole space which is assumed to be $\Gamma$-invariant;\\
5. The category of all the finite projective modules over the ring $\Pi_{H,A}$, carrying the $(\varphi^a,\Gamma)$-action;\\
6. The category of all the finite projective bundles over the ring $\Pi_{H,A}$, carrying the $(\varphi^a,\Gamma)$-action;\\
7. The category of all the finite projective modules over the ring $\breve{\Pi}_{H,A}$, carrying the $(\varphi^a,\Gamma)$-action;\\
8. The category of all the finite projective bundles over the ring $\breve{\Pi}_{H,A}$, carrying the $(\varphi^a,\Gamma)$-action;\\
9. The category of all the finite projective modules over the ring $\breve{\Pi}^{[s,r]}_{H,A}$, carrying the $(\varphi^a,\Gamma)$-action;\\
10. The category of all the finite projective modules over the ring $\widetilde{\Pi}_{H,A}$, carrying the $(\varphi^a,\Gamma)$-action;\\
11. The category of all the finite projective bundles over the ring $\widetilde{\Pi}_{H,A}$, carrying the $(\varphi^a,\Gamma)$-action;\\
12. The category of all the finite projective modules over the ring $\widetilde{\Pi}^{[s,r]}_{H,A}$, carrying the $(\varphi^a,\Gamma)$-action.\\	
\end{proposition}

\begin{proof}
The corresponding comparisons on the sheaves and bundles in 1-4 and 10-12 are derived in the corresponding context in the perfect setting as what we did in the previous sections. The rest ones are proved exactly the same as \cite[Theorem 5.7.5]{KL16} by using our development.	
\end{proof}

\begin{remark}
This proposition generalizes the corresponding results in \cite{KP} including the situation therein considered by Chojecki-Gaisin, while note that we have also included the situation in the equal characteristic situation.	
\end{remark}

\indent Furthermore if one considers the corresponding context where the ring $H'_{H}$ is further assumed to be $F$-(finite projective) then we can discuss the level of pseudocoherent objects.

\begin{lemma}\mbox{\bf{(After Kedlaya-Liu \cite[Lemma 5.8.7]{KL16})}}
For any radii in our situation namely $0<s\leq r\leq r_0$ we have the following isomorphism:
\begin{displaymath}
\widetilde{\Pi}^{[s,r]}_{H,A}\overset{\sim}{\rightarrow} \Pi_{H,A}^{[s,r]}\oplus (\oplus_{\ell=0}^\infty \varphi^{-(\ell+1)
}\Pi_{H,A}^{[sp^{-h(\ell+1)},rp^{-h(\ell+1)}]}/\varphi^{-\ell
}\Pi_{H,A}^{[sp^{-h\ell},rp^{-h\ell}]})^\wedge.	
\end{displaymath}	
\end{lemma}

\begin{proof}
This is by considering the Schauder basis of $A$. See \cite[Lemma 5.8.7]{KL16}.	
\end{proof}

\begin{corollary}\mbox{\bf{(After Kedlaya-Liu \cite[Corollary 5.8.8, Corollary 5.8.11]{KL16})}}
The map for the radii as in the previous lemma is 2-pseudoflat:
\begin{displaymath}
\Pi_{H,A}^{[s,r]}\rightarrow \widetilde{\Pi}^{[s,r]}_{H,A}.	
\end{displaymath}
Also we have that the corresponding base change along this map will preserve the corresponding \'etale stably pseudocoherence.
	
\end{corollary}

\begin{lemma} \mbox{\bf{(After Kedlaya-Liu \cite[Lemma 5.8.14]{KL16})}}
The complex 
\begin{displaymath}
M\otimes_{\Pi^{[s,r]}_{H,A}} \varphi^{-(\ell+1)}\Pi^{[sp^{-h(\ell+1)},rp^{-h(\ell+1)}]}_{H^\bullet,A}/ \varphi^{-\ell}\Pi^{[sp^{-h\ell},rp^{-h\ell}]}_{H^\bullet,A}	
\end{displaymath}
and the complex 
\begin{displaymath}
M\otimes_{\Pi^{[s,r]}_{H,A}} \widetilde{\Pi}^{[s,r]}_{H^\bullet,A}/\varphi^{-\ell}\Pi^{[sp^{-h(\ell)},rp^{-h(\ell)}]}_{H^\bullet,A}
\end{displaymath}
are strict exact for any truncation index $n$. The corresponding radii satisfy the corresponding relation $0<s\leq r\leq r_0$, and $\ell$ is bigger than some existing truncated integer $\ell_0\geq  0$. Here $M$ is any pseudocoherent $\Gamma$-module in our context.	
\end{lemma}

\begin{proposition}\mbox{\bf{(After Kedlaya-Liu \cite[Lemma 5.8.17]{KL16})}}
The corresponding base change along the following map is fully faithful for the corresponding pseudocoherent modules carrying the corresponding structure of $\Gamma$-action:
\begin{displaymath}
\breve{\Pi}_{H,A}^{[s,r]}\rightarrow \widetilde{\Pi}^{[s,r]}_{H,A}.	
\end{displaymath}
The image in the essential sense consists of those modules descending to the corresponding ring in the domain of this map when forgetting the corresponding $\Gamma$-action.	
\end{proposition}

\begin{proof}
See \cite[Lemma 5.8.17]{KL16}.	
\end{proof}

\begin{proposition} \mbox{\bf{(After Kedlaya-Liu \cite[Lemma 5.9.2]{KL16})}}
Keep the notations as above. For any finitely generated in general module over the ring $\widetilde{\Pi}^{[s,r]}_{H,A}$ carrying the action of $\Gamma$, one can find another module over $\breve{\Pi}^{[s,r]}_{H,A}$ which is now pseudocoherent carrying the action of $\Gamma$ which covers the previous module through a surjective map after taking the corresponding base change to the ring $\widetilde{\Pi}^{[s,r]}_{H,A}$.	
\end{proposition}

\begin{proof}
As in the proof of \cite[Lemma 5.9.2]{KL16}, one can prove this in the similar fashion by considering the corresponding isomorphism respecting the cocycle requirement coming from the free module $\widetilde{\Lambda}$ in the presentation of a given module $\widetilde{\Delta}$ over the perfect Robba ring:
\begin{displaymath}
\widetilde{\Lambda}\otimes_{i_{0,0}} \widetilde{\Pi}^{[s,r]}_{H^0,A}\rightarrow \widetilde{\Lambda} \otimes_{i_{0,1}} \widetilde{\Pi}^{[s,r]}_{H^1,A}	
\end{displaymath}
coming from our prescribed $\Gamma$-action on the module $\widetilde{\Delta}$ over the perfect Robba ring in the way that one considers a converging sequence of different desired lifts (as in \cite[Lemma 5.6.9]{KL16}) of 
\begin{displaymath}
\Lambda\otimes_{i_{0,0}} \breve{\Pi}^{[s,r]}_{H^0,A}\rightarrow \Lambda \otimes_{i_{0,1}} \breve{\Pi}^{[s,r]}_{H^1,A}.	
\end{displaymath}	
To get the corresponding desired covering of $\widetilde{\Delta}$ one considers the corresponding finitely generated submodule of the corresponding kernel of the map from $\Lambda$ to $\widetilde{\Delta}$ then take the quotient of $\Lambda$ through this corresponding submodule. Note that in our situation the ring $\breve{\Pi}^{[s,r]}_{H,A}$ is also coherent which finishes the proof as in \cite[Lemma 5.9.2]{KL16}. 
\end{proof}

\begin{proposition} \mbox{\bf{(After Kedlaya-Liu \cite[Theorem 5.9.4]{KL16})}}
We have now the following categories are equivalence for the corresponding radii $0< s\leq r\leq r_0$ (with the further requirement as in \cite[Theorem 5.7.5]{KL16} that $s\in (0,r/q]$):\\
1. The category of all the pseudocoherent sheaves over the ring $\widetilde{\Pi}_{\mathrm{Spa}(H_0,H_0^+),A}$, carrying the $\varphi^a$-action;\\
2. The category of all the pseudocoherent sheaves over the ring $\widetilde{\Pi}^r_{\mathrm{Spa}(H_0,H_0^+),A}$, carrying the $\varphi^a$-action;\\
3. The category of all the pseudocoherent sheaves over the ring $\widetilde{\Pi}^{[s,r]}_{\mathrm{Spa}(\overline{H}_\infty,\overline{H}^+_\infty),A}$, carrying the $\varphi^a$-action.\\
\indent Then we have the second group of categories which are equvalent:\\
4. The category of all the pseudocoherent quasi-coherent sheaves over corresponding adic Fargues-Fontaine curve in the deformed setting $\mathrm{FF}_{\overline{H}'_\infty,A}$, carrying the corresponding action from the group $\Gamma$ which is assumed to be semilinear and continuous over each section over any affinoid subspace of the whole space which is assumed to be $\Gamma$-invariant;\\
5. The category of all the pseudocoherent modules over the ring $\Pi_{H,A}$, carrying the $(\varphi^a,\Gamma)$-action;\\
6. The category of all the pseudocoherent bundles over the ring $\Pi_{H,A}$, carrying the $(\varphi^a,\Gamma)$-action;\\
7. The category of all the pseudocoherent modules over the ring $\breve{\Pi}_{H,A}$, carrying the $(\varphi^a,\Gamma)$-action;\\
8. The category of all the pseudocoherent bundles over the ring $\breve{\Pi}_{H,A}$, carrying the $(\varphi^a,\Gamma)$-action;\\
9. The category of all the pseudocoherent modules over the ring $\breve{\Pi}^{[s,r]}_{H,A}$, carrying the $(\varphi^a,\Gamma)$-action;\\
10. The category of all the pseudocoherent modules over the ring $\widetilde{\Pi}_{H,A}$, carrying the $(\varphi^a,\Gamma)$-action;\\
11. The category of all the pseudocoherent bundles over the ring $\widetilde{\Pi}_{H,A}$, carrying the $(\varphi^a,\Gamma)$-action;\\
12. The category of all the pseudocoherent modules over the ring $\widetilde{\Pi}^{[s,r]}_{H,A}$, carrying the $(\varphi^a,\Gamma)$-action.\\	
\end{proposition}

\begin{proof}
This is by using the $A$-relative version of \cite[Lemma 5.9.3]{KL16} which states that actually the base change functor along
\begin{displaymath}
\breve{\Pi}_{H,A}^{[s,r]}\rightarrow \widetilde{\Pi}^{[s,r]}_{H,A}	
\end{displaymath}	
is then not only fully faithful on the category of pseudocoherent objects as above but also essential surjective. Indeed by the previous proposition in our situation we have that a surjective covering of any module $\widetilde{\Lambda}$ over $\widetilde{\Pi}^{[s,r]}_{H,A}$:
\begin{displaymath}
\widetilde{\Lambda}'\rightarrow \widetilde{\Lambda}\rightarrow 0	
\end{displaymath}
where $\widetilde{\Lambda}'$ descend to the ring $\breve{\Pi}_{H,A}^{[s,r]}$ carrying $\Gamma$-action. Then by applying the same process above to the corresponding kernel of the covering above we have another exact sequence in the following form:
\begin{displaymath}
\widetilde{\Lambda}''\rightarrow \widetilde{\Lambda}'\rightarrow \widetilde{\Lambda}\rightarrow 0	
\end{displaymath}
Here $\widetilde{\Lambda}'$ and $ \widetilde{\Lambda}$ are finitely presented as in \cite[Lemma 5.9.3]{KL16}. Then we can now consider the corresponding cokernel of the map $\widetilde{\Lambda}''\rightarrow \widetilde{\Lambda}'$ which will present a corresponding desired object carrying $\Gamma$-action over the ring $\breve{\Pi}_{H,A}^{[s,r]}$ which is again pseudocoherent in our context.
\end{proof}

\newpage

\section{Organizaion of Noncommutative Setting}

\subsection{Noncommutative Period Rings in Perfect Setting}

\noindent Now we assume $A$ to be some noncommutative Banach affinoid algebra over the local fields we consider above. Examples of such rings could be coming from the corresponding context of \cite{Soi1}, or the corresponding noncommutative Fr\'echet-Stein algebras as in \cite{ST1}. We will use the notation $E\{Z_1,...,Z_n\}$ to denote the corresponding noncommutative non-noetherian Tate algebra as in the commutative setting.

\begin{remark}
The noncommutative consideration is not quite new, for instance the aspects rooted in \cite{Zah1}, \cite{Wit1} and \cite{Wit2} (and even \cite{KL16}), but obviously the corresponding noncommutative setting is more complicated than the corresponding commutative setting which we discussed extensively in the previous context, therefore we do not have the chance to reach all the corresponding noncommutative version of the results above. 	
\end{remark}

\begin{remark}
We choose to closely in some parallel way present the corresponding construction in the noncommutative setting, which is parallel in some aspects to the commutative setting we presented above. We definitely won't have the chance to see the complete picture as in the commutative setting and even more consideration and effort has to be made during the corresponding development.	
\end{remark}

%
We first deform the corresponding constructions in \cite[Definition 4.1.1]{KL16}:

\begin{definition}
We first consider the corresponding deformation of the above rings over $\mathbb{Q}_p\{Z_1,...,Z_d\}$. We are going to use the notation $W_\pi(R)_{\mathbb{Q}_p\{Z_1,...,Z_d\}}$ to denote the complete tensor product of $W_\pi(R)$ with the Tate algebra $\mathbb{Q}_p\{Z_1,...,Z_d\}$ consisting of all the element taking the form as:
\begin{displaymath}
\sum_{k\geq 0,i_1\geq 0,...,i_d\geq 0}\pi^k[\overline{x}_{k,i_1,...,i_d}]Z_1^{i_1}Z_2^{i_2}...Z_d^{i_d}	
\end{displaymath}
over which we have the Gauss norm $\left\|.\right\|_{\alpha^r,\mathbb{Q}_p\{Z_1,...Z_d\}}$ for any $r>0$ which is defined by:
\begin{displaymath}
\left\|.\right\|_{\alpha^r,\mathbb{Q}_p\{Z_1,...Z_d\}}(\sum_{k\geq 0,i_1\geq 0,...,i_d\geq 0}\pi^k[\overline{x}_{k,i_1,...,i_d}]Z_1^{i_1}Z_2^{i_2}...Z_d^{i_d}):=\sup_{k\geq 0,i_1\geq 0,...,i_d\geq 0}p^{-k}\alpha(\overline{x}_{k,i_1,...,i_d})^r.	
\end{displaymath}
Then we define the corresponding convergent rings:
\begin{displaymath}
\widetilde{\Omega}^\mathrm{int}_{R,\mathbb{Q}_p\{Z_1,...,Z_d\}}=W_\pi(R)_{\mathbb{Q}_p\{Z_1,...,Z_d\}},	\widetilde{\Omega}_{R,\mathbb{Q}_p\{Z_1,...,Z_d\}}=W_\pi(R)_{\mathbb{Q}_p\{Z_1,...,Z_d\}}[1/\pi].	
\end{displaymath}
Then as in the previous setting we define the corresponding $\mathbb{Q}_p\{Z_1,...,Z_d\}$-relative integral Robba ring $\widetilde{\Pi}_{R,\mathbb{Q}_p\{Z_1,...Z_d\}}^{\mathrm{int},r}$ as the completion of the ring $W_{\pi}(R^+)_{\mathbb{Q}_p\{Z_1,...,Z_d\}}[[R]]$ by using the Gauss norm defined above.
Then one can take the corresponding union of all $\widetilde{\Pi}_{R,\mathbb{Q}_p\{Z_1,...,Z_d\}}^{\mathrm{int},r}$ throughout all $r>0$ to define the integral Robba ring $\widetilde{\Pi}_{R,\mathbb{Q}_p\{Z_1,...,Z_d\}}^{\mathrm{int}}$. For the bounded Robba rings we just set $\widetilde{\Pi}_{R,\mathbb{Q}_p\{Z_1,...,Z_d\}}^{\mathrm{bd},r}=\widetilde{\Pi}_{R,\mathbb{Q}_p\{Z_1,...,Z_d\}}^{\mathrm{int},r}[1/\pi]$ and taking the union throughout all $r>0$ to define the corresponding ring $\widetilde{\Pi}_{R,\mathbb{Q}_p\{Z_1,...,Z_d\}}^{\mathrm{bd}}$. Then we define the corresponding Robba ring $\widetilde{\Pi}_{R,\mathbb{Q}_p\{Z_1,...,Z_d\}}^{I}$ with respect to some interval $I\subset (0,\infty)$ by taking the Fr\'echet completion of ${W_\pi(R^+)}_{\mathbb{Q}_p\{Z_1,...,Z_d\}}[[R]][1/\pi]$ with respect to all the norms $\left\|.\right\|_{\alpha^r,\mathbb{Q}_p\{Z_1,...,Z_d\}}$ for all $r\in I$ which means that the corresponding equivalence classes in the completion procedure will be simultaneously Cauchy with respect to all the norms $\left\|.\right\|_{\alpha^r,\mathbb{Q}_p\{Z_1,...,Z_d\}}$ for all $r\in I$. Then we take suitable intervals such as $(0,r]$ and $(0,\infty)$ to define the corresponding Robba rings $\widetilde{\Pi}_{R,\mathbb{Q}_p\{Z_1,...,Z_d\}}^{r}$ and $\widetilde{\Pi}_{R,\mathbb{Q}_p\{Z_1,...,Z_d\}}^{\infty}$, respectively. Then taking the union throughout all $r>0$ one can define the corresponding Robba ring $\widetilde{\Pi}_{R,\mathbb{Q}_p\{Z_1,...,Z_d\}}$. Again as in \cite{KL16} one can define the corresponding integral rings of the similar types.
\end{definition}

\indent Then one can define the corresponding period rings in our context deformed over some affinoid $A$ which is isomorphic to some quotient of $\mathbb{Q}_p\{Z_1,...Z_d\}$, again deforming from the context of \cite[Definition 4.1.1]{KL16}.

\begin{definition}
In the characteristic $0$, we define the following period rings:
\begin{displaymath}
\widetilde{\Omega}^\mathrm{int}_{R,A},\widetilde{\Omega}_{R,A},\widetilde{\Pi}^\mathrm{int,r}_{R,A},\widetilde{\Pi}^\mathrm{bd,r}_{R,A}, \widetilde{\Pi}^I_{R,A},\widetilde{\Pi}^r_{R,A},\widetilde{\Pi}^\infty_{R,A}	
\end{displaymath}
by taking the suitable quotient of the following period rings defined above: 
\begin{align}
\widetilde{\Omega}^\mathrm{int}_{R,\mathbb{Q}_p\{Z_1,...,Z_d\}},&\widetilde{\Omega}_{R,\mathbb{Q}_p\{Z_1,...,Z_d\}},\widetilde{\Pi}^\mathrm{int,r}_{R,\mathbb{Q}_p\{Z_1,...,Z_d\}},\widetilde{\Pi}^\mathrm{bd,r}_{R,\mathbb{Q}_p\{Z_1,...,Z_d\}},\widetilde{\Pi}^I_{R,\mathbb{Q}_p\{Z_1,...,Z_d\}},\\
 &\widetilde{\Pi}^r_{R,\mathbb{Q}_p\{Z_1,...,Z_d\}},\widetilde{\Pi}^\infty_{R,\mathbb{Q}_p\{Z_1,...,Z_d\}}	
\end{align}
with respect to the structure of the affinoid algebra $A$ in the sense of Tate. Therefore they carry the corresponding quotient seminorms of the above Gauss norms defined in the previous definition, note that these are not something induced from the corresponding spectral seminorms from $A$ chosen at the very beginning of our study. We use the notation $\overline{\left\|.\right\|}_{\alpha^r,\mathbb{Q}_p\{Z_1,...,Z_d\}}$ to denote the corresponding quotient Gauss norm which induces then the corresponding spectral seminorm $\left\|.\right\|_{\alpha^r,A}$ for each $r>0$. Then we define the corresponding period rings:
\begin{align}
\widetilde{\Pi}^\mathrm{int}_{R,A},\widetilde{\Pi}^\mathrm{bd}_{R,A},\widetilde{\Pi}_{R,A}	
\end{align}
by taking suitable union throughout all $r>0$.
\end{definition}

\begin{definition}
In positive characteristic situation, when we are working over general affinoid algebra $A$, we use the same notations as in the previous definition, but by using $W(R)_{\mathbb{F}_p[[\eta]]\{Z_1,...,Z_d\}}$ and $W(R)_{A}$ as the starting rings, namely here $A$ is isomorphic to a quotient of $\mathbb{F}_p((\eta))\{Z_1,...,Z_d\}$. Again the corresponding Tate algebra is defined in term of free variables.	
\end{definition}

\indent Then we can define the following affinoid deformations (after \cite[Definition 4.1.1]{KL16} in the flavor as above):

\begin{definition}
Consider now the base change $W_{\pi,\infty,\mathbb{Q}_p\{Z_1,...,Z_d\}}(R)$ which is defined now to be the completed tensor product $(W_{\pi}(R)\widehat{\otimes}_{\mathcal{O}_E}\mathcal{O}_{E_\infty})\widehat{\otimes}_{\mathbb{Q}_p}\mathbb{Q}_p\{Z_1,...,Z_d\}$. Then the point is that each element in this ring admits a unique expression taking the form of 
\begin{center}
$\sum_{n\in \mathbb{Z}[1/p]_{\geq 0},i_1\geq 0,...,i_d\geq 0}\pi^n[\overline{x}_{n,i_1,...,i_d}]Z_1^{i_1}...Z_d^{i_d}$ 	
\end{center}
which allows us to perform the construction mentioned above. First we can define for some radius $r>0$ the corresponding period ring $\widetilde{\Pi}^{\mathrm{int},r}_{R,\infty,\mathbb{Q}_p\{Z_1,...,Z_d\}}$ by taking the completion of the ring
\begin{displaymath}
W_{\pi,\infty,\mathbb{Q}_p\{Z_1,...,Z_d\}}(R^+)[[R]]	
\end{displaymath}
with respect to the following Gauss type norm:
\begin{displaymath}
\left\|.\right\|_{\alpha^r}(\sum_{n\in \mathbb{Z}[1/p]_{\geq 0},i_1\geq 0,...,i_d\geq 0}\pi^n[\overline{x}_{n,i_1,...,i_d}]Z_1^{i_1}...Z_d^{i_d}):=\sup_{n\in \mathbb{Z}[1/p]_{\geq 0},i_1\geq 0,...,i_d\geq 0}\{p^{-n}\alpha(\overline{x}_{n,i_1,...,i_d})^r\}.	
\end{displaymath}
Then we can define the union $\widetilde{\Pi}^{\mathrm{int}}_{R,\infty,\mathbb{Q}_p\{Z_1,...,Z_d\}}$ throughout all the radius $r>0$. Then we just define the bounded Robba ring $\widetilde{\Pi}^{\mathrm{bd},r}_{R,\infty,\mathbb{Q}_p\{Z_1,...,Z_d\}}$ by $\widetilde{\Pi}^{\mathrm{int},r}_{R,\infty,\mathbb{Q}_p\{Z_1,...,Z_d\}}[1/\pi]$ and also we could define the union $\widetilde{\Pi}^{\mathrm{bd}}_{R,\infty,\mathbb{Q}_p\{Z_1,...,Z_d\}}$ throughout all the radius $r>0$. Then for any interval in $(0,\infty)$ which is denoted by $I$ we can define the corresponding Robba rings $\widetilde{\Pi}^{I}_{R,\infty,\mathbb{Q}_p\{Z_1,...,Z_d\}}$ by taking the Fr\'echet completion of 
\begin{displaymath}
W_{\pi,\infty,\mathbb{Q}_p\{Z_1,...,Z_d\}}(R^+)[[R]][1/\pi]	
\end{displaymath}
with respect to all the norms $\left\|.\right\|_{\alpha^t}$ for all $t\in I$. Then by taking suitable specified intervals one can define the rings $\widetilde{\Pi}^{r}_{R,\infty,\mathbb{Q}_p\{Z_1,...,Z_d\}}$ and $\widetilde{\Pi}^{\infty}_{R,\infty,\mathbb{Q}_p\{Z_1,...,Z_d\}}$ as before, and finally one can define the corresponding union $\widetilde{\Pi}_{R,\infty,\mathbb{Q}_p\{Z_1,...,Z_d\}}$ throughout all the radius $r>0$. Again we have the corresponding integral version of the rings defined over $E_\infty$ as \cite{KL16}. Finally over $A$ we can define in the same way as above to deform all the rings over $E_\infty$, and we do not repeat the construction again.

\end{definition}

\subsection{Basic Properties of Period Rings}	
	
\indent Then we do some reality checks over the investigation of the properties of the above period rings in the style taken in \cite[Section 5.2]{KL15} and \cite{KL16}. 

\begin{proposition} \mbox{\bf{(After Kedlaya-Liu \cite[Lemma 5.2.1]{KL15})}}
The function $t\mapsto \left\|x\right\|_{\alpha^t,\mathbb{Q}_p\{Z_1,...,Z_d\}}$ for $x\in \widetilde{\Pi}^r_{R,\mathbb{Q}_p\{Z_1,...,Z_d\}}$ is continuous log convex for the corresponding variable $t\in (0,r]$ for any $r>0$.
\end{proposition}

\begin{proof}
Adapt the corresponding argument in the proof of 5.2.1 of \cite[Lemma 5.2.1]{KL15} to our situation we then first look at the situation where the element is just of the form of $\pi^k[\overline{x}_{k,i_1,...,i_d}]Z_1^{i_1}...Z_d^{i_d}$ where the corresponding norm function in terms of $t>0$ is just affine. Then one focuses on the finite sums of these kind of elements which gives rise to to the log convex directly. Finally by taking the approximation we get the desired result.	
\end{proof}

\begin{proposition}\mbox{\bf{(After Kedlaya-Liu \cite[Lemma 5.2.2]{KL15})}}
For any element $x\in \widetilde{\Pi}_{R,\mathbb{Q}_p\{Z_1,...,Z_d\}}$ we have that $x\in \widetilde{\Pi}^\mathrm{bd}_{R,\mathbb{Q}_p\{Z_1,...,Z_d\}}$ if and only if we have the situation where $x$ actually lives in $\widetilde{\Pi}^r_{R,\mathbb{Q}_p\{Z_1,...,Z_d\}}$ (for some specific $r>0$) such that $x$ itself is bounded under the norm $\left\|.\right\|_{\alpha^t,\mathbb{Q}_p\{Z_1,...,Z_d\}}$ for each $t\in (0,r]$.	
\end{proposition}

\begin{proof}
One direction of the proof is easy, so we only choose to present the proof of the implication in the other direction as in the original proof of 5.2.2 of \cite[Lemma 5.2.2]{KL15} as in the following. First choose some radius $r>0$ such that the element could be assumed to be living in the ring $\widetilde{\Pi}^r_{R,\mathbb{Q}_p\{Z_1,...,Z_d\}}$. The idea is to transfer the original question to the question about showing the integrality of $x$ when we add some hypothesis on the norm by taking suitable powers of $p$ (since the norm is bounded for each $t\in (0,r]$ so we are reduced to the situation where the norm is bounded by $1$). Then we argue as in \cite[Lemma 5.2.2]{KL15} to choose some approximating sequence $\{x_i\}$ living in $\widetilde{\Pi}^\mathrm{bd}_{R,\mathbb{Q}_p\{Z_1,...,Z_d\}}$ of $x$. Therefore we have for any $j\geq 1$ one can find then some integer $N_j\geq 1$ such that for any $i\geq N_j$ we have the estimate:
\begin{displaymath}
\left\|.\right\|_{\alpha^t,\mathbb{Q}_p\{Z_1,...,Z_d\}}(x_i-x)\leq p^{-j}, \forall t\in [p^{-j}r,r].	
\end{displaymath}
Then the idea is to consider the integral decomposition of the element $x_i$ which has the form of $\sum_{k=n(x_i),i_1\geq 0,i_2\geq 0,...,i_d\geq 0}\pi^k[\overline{x}_{i,k,i_1,...,i_d}]$ into the following two parts:
\begin{align}
x_i &:= y_i+z_i\\
	&:=\sum_{k=0,i_1\geq 0,i_2\geq 0,...,i_d\geq 0}\pi^k[\overline{x}_{i,k,i_1,...,i_d}]Z_1^{i_1}...Z_d^{i_d}+z_i
\end{align}
from which we actually have the corresponding estimate over the residual part of the decomposition above:
\begin{displaymath}
\left\|.\right\|_{\alpha^{p^{-j}r},\mathbb{Q}_p\{Z_1,...,Z_d\}}(\pi^k[\overline{x}_{i,k,i_1,...,i_d}]Z_1^{i_1}...Z_d^{i_d})\leq 1, \forall k<0	
\end{displaymath}
which implies by direct computation:
\begin{align}
\alpha^{p^{-j}r}(\overline{x}_{i,k,i_1,...,i_d})&\leq p^k\\
\alpha(\overline{x}_{i,k,i_1,...,i_d})&\leq p^{kp^{j}/r}	
\end{align}
which implies that we have the following estimate: 
\begin{align}
\left\|.\right\|_{\alpha^{r},\mathbb{Q}_p\{Z_1,...,Z_d\}}(x_i-y_i)&\leq p^{-k}p^{kp^{j}r/r}\\
&\leq p^{1-p^j}
\end{align}
which implies that $y_i\rightarrow x$ under the norm $\left\|.\right\|_{\alpha^{r},\mathbb{Q}_p\{Z_1,...,Z_d\}}$ which furthermore under the norm $\left\|.\right\|_{\alpha^{r},\mathbb{Q}_p\{Z_1,...,Z_d\}}$ due the property of the norm, which finishes the proof of the desired result.
\end{proof}

\begin{proposition}\mbox{\bf{(After Kedlaya-Liu \cite[Corollary 5.2.3]{KL15})}} We have the following identity:
\begin{displaymath}
(\widetilde{\Pi}_{R,\mathbb{Q}_p\{Z_1,...,Z_d\}})^\times=(\widetilde{\Pi}^\mathrm{bd}_{R,\mathbb{Q}_p\{Z_1,...,Z_d\}})^\times.
\end{displaymath}
\end{proposition}

\begin{proof}
See 5.2.3 of \cite[Corollary 5.2.3]{KL15}.	
\end{proof}

\begin{proposition}\mbox{\bf{(After Kedlaya-Liu \cite[Lemma 5.2.6]{KL15})}}
For any $0< r_1\leq r_2$ we have the following equality on the corresponding period rings:
\begin{displaymath}
\widetilde{\Pi}^{\mathrm{int},r_1}_{R,\mathbb{Q}_p\{Z_1,...,Z_d\}}\bigcap	\widetilde{\Pi}^{[r_1,r_2]}_{R,\mathbb{Q}_p\{Z_1,...,Z_d\}}=\widetilde{\Pi}^{\mathrm{int},r_2}_{R,\mathbb{Q}_p\{Z_1,...,Z_d\}}.
\end{displaymath}	
\end{proposition}

\begin{proof}
We adapt the argument in \cite[Lemma 5.2.6]{KL15} 5.2 to prove this in the situation where $r_1<r_2$ (otherwise this is trivial), again one direction is easy where we only present the implication in the other direction. We take any element $x\in \widetilde{\Pi}^{\mathrm{int},r_1}_{R,\mathbb{Q}_p\{Z_1,...,Z_d\}}\bigcap	\widetilde{\Pi}^{[r_1,r_2]}_{R,\mathbb{Q}_p\{Z_1,...,Z_d\}}$ and take suitable approximating elements $\{x_i\}$ living in the bounded Robba ring such that for any $j\geq 1$ one can find some integer $N_j\geq 1$ we have whenever $i\geq N_j$ we have the following estimate:
\begin{displaymath}
\left\|.\right\|_{\alpha^{t},\mathbb{Q}_p\{Z_1,...,Z_d\}}(x_i-x) \leq p^{-j}, \forall t\in [r_1,r_2].	
\end{displaymath}
Then we consider the corresponding decomposition of $x_i$ for each $i=1,2,...$ into a form having integral part and the rational part $x_i=y_i+z_i$ by setting
\begin{center}
 $y_i=\sum_{k=0,i_1,...,i_d}\pi^k[\overline{x}_{i,k,i_1,...,i_d}]Z_1^{i_1}...Z_d^{i_d}$ 
\end{center} 
out of
\begin{center} 
$x_i=\sum_{k=n(x_i),i_1,...,i_d}\pi^k[\overline{x}_{i,k,i_1,...,i_d}]Z_1^{i_1}...Z_d^{i_d}$.
\end{center}
Note that by our initial hypothesis we have that the element $x$ lives in the ring $\widetilde{\Pi}^{\mathrm{int},r_1}_{R,\mathbb{Q}_p\{Z_1,...,Z_d\}}$ which further implies that 
\begin{displaymath}
\left\|.\right\|_{\alpha^{r_1},\mathbb{Q}_p\{Z_1,...,Z_d\}}(\pi^k[\overline{x}_{i,k,i_1,...,i_d}]Z_1^{i_1}...Z^{i_d}_d)	\leq p^{-j}.
\end{displaymath}
Therefore we have $\alpha(\overline{x}_{i,k,i_1,...,i_d})\leq p^{(k-j)/r_1}$ directly from this through computation, which implies that then:
\begin{align}
\left\|.\right\|_{\alpha^{r_2},\mathbb{Q}_p\{Z_1,...,Z_d\}}(\pi^k[\overline{x}_{i,k,i_1,...,i_d}]Z_1^{i_1}...Z^{i_d}_d)	&\leq p^{-k}p^{(k-j)r_2/r_1}\\
	&\leq p^{1+(1-j)r_1/r_1}.
\end{align}
Then one can read off the result directly from this estimate since under this estimate we can have the chance to modify the original approximating sequence $\{x_i\}$ by $\{y_i\}$ which are initially chosen to be in the integral Robba ring, which implies that actually the element $x$ lives in the right-hand side of the identity in the statement of the proposition.
\end{proof}

\begin{proposition} \mbox{\bf{(After Kedlaya-Liu \cite[Lemma 5.2.6]{KL15})}}
For any $0< r_1\leq r_2$ we have the following equality on the corresponding period rings:
\begin{displaymath}
\widetilde{\Pi}^{\mathrm{int},r_1}_{R,A}\bigcap	\widetilde{\Pi}^{[r_1,r_2]}_{R,A}=\widetilde{\Pi}^{\mathrm{int},r_2}_{R,A}.
\end{displaymath}	
Here $A$ is some noncommutative Banach affinoid algebra over $\mathbb{Q}_p$.	
\end{proposition}

\begin{proof}
See the proof of \cref{proposition5.7}.	
\end{proof}

\indent Then we have the following analog of the corresponding result of \cite[Lemma 5.2.8]{KL15}:

\begin{proposition} \mbox{\bf{(After Kedlaya-Liu \cite[Lemma 5.2.8]{KL15})}}
Consider now in our situation the radii $0< r_1\leq r_2$, and consider any element $x\in \widetilde{\Pi}^{[r_1,r_2]}_{R,\mathbb{Q}_p\{Z_1,...,Z_d\}}$. Then we have that for each $n\geq 1$ one can decompose $x$ into the form of $x=y+z$ such that $y\in \pi^n\widetilde{\Pi}^{\mathrm{int},r_2}_{R,\mathbb{Q}_p\{Z_1,...,Z_d\}}$ with $z\in \bigcap_{r\geq r_2}\widetilde{\Pi}^{[r_1,r]}_{R,\mathbb{Q}_p\{Z_1,...,Z_d\}}$ with the following estimate for each $r\geq r_2$:
\begin{displaymath}
\left\|.\right\|_{\alpha^r,\mathbb{Q}_p\{Z_1,...,Z_d\}}(z)\leq p^{(1-n)(1-r/r_2)}\left\|.\right\|_{\alpha^{r_2},\mathbb{Q}_p\{Z_1,...,Z_d\}}(z)^{r/r_2}.	
\end{displaymath}

\end{proposition}

\begin{proof}
As in \cite[Lemma 5.2.8]{KL15} and in the proof of our previous proposition we first consider those elements $x$ lives in the bounded Robba rings which could be expressed in general as
\begin{center}
 $\sum_{k=n(x),i_1,...,i_d}\pi^k[\overline{x}_{k,i_1,...,i_d}]Z_1^{i_1}...Z_d^{i_d}$.
 \end{center}	
In this situation the corresponding decomposition is very easy to come up with, namely we consider the corresponding $y_i$ as the corresponding series:
\begin{displaymath}
\sum_{k\geq n,i_1,...,i_d}\pi^k[\overline{x}_{k,i_1,...,i_d}]Z_1^{i_1}...Z_d^{i_d}	
\end{displaymath}
which give us the desired result since we have in this situation when focusing on each single term:
\begin{align}
\left\|.\right\|_{\alpha^r,\mathbb{Q}_p\{Z_1,...,Z_d\}}(\pi^k[\overline{x}_{k,i_1,...,i_d}]Z_1^{i_1}...Z_d^{i_d})&=p^{-k}\alpha(\overline{x}_{k,i_1,...,i_d})^r\\
&=p^{-k(1-r/r_2)}\left\|.\right\|_{\alpha^{r_2},\mathbb{Q}_p\{Z_1,...,Z_d\}}(\pi^k[\overline{x}_{k,i_1,...,i_d}]Z_1^{i_1}...Z_d^{i_d})^{r/r_2}\\
&\leq p^{(1-n)(1-r/r_2)}\left\|.\right\|_{\alpha^{r_2},\mathbb{Q}_p\{Z_1,...,Z_d\}}(\pi^k[\overline{x}_{k,i_1,...,i_d}]Z_1^{i_1}...Z_d^{i_d})^{r/r_2}
\end{align}
for all those suitable $k$. Then to tackle the more general situation we consider the approximating sequence consisting of all the elements in the bounded Robba ring as in the usual situation considered in \cite[Lemma 5.2.8]{KL15}, namely we inductively construct the following approximating sequence just as:
\begin{align}
\left\|.\right\|_{\alpha^r,\mathbb{Q}_p\{Z_1,...,Z_d\}}(x-x_0-...-x_i)\leq p^{-i-1}	\left\|.\right\|_{\alpha^r,\mathbb{Q}_p\{Z_1,...,Z_d\}}(x), i=0,1,..., r\in [r_1,r_2].
\end{align}
Here all the elements $x_i$ for each $i=0,1,...$ are living in the bounded Robba ring, which immediately gives rise to the suitable decomposition as proved in the previous case namely we have for each $i$ the decomposition $x_i=y_i+z_i$ with the desired conditions as mentioned in the statement of the proposition. We first take the series summing all the elements $y_i$ up for all $i=0,1,...$, which first of all converges under the norm $\left\|.\right\|_{\alpha^r,\mathbb{Q}_p\{Z_1,...,Z_d\}}$ for all the radius $r\in [r_1,r_2]$, and also note that all the elements $y_i$ within the infinite sum live inside the corresponding integral Robba ring $\widetilde{\Pi}^{\mathrm{int},r_2}_{R,\mathbb{Q}_p\{Z_1,...,Z_d\}}$, which further implies the corresponding convergence ends up in $\widetilde{\Pi}^{\mathrm{int},r_2}_{R,\mathbb{Q}_p\{Z_1,...,Z_d\}}$. For the elements $z_i$ where $i=0,1,...$ also sum up to a converging series in the desired ring since combining all the estimates above we have:
\begin{displaymath}
\left\|.\right\|_{\alpha^r,\mathbb{Q}_p\{Z_1,...,Z_d\}}(z_i)\leq p^{(1-n)(1-r/r_2)}\left\|.\right\|_{\alpha^{r_2},\mathbb{Q}_p\{Z_1,...,Z_d\}}(x)^{r/r_2}.	
\end{displaymath}
\end{proof}

%

\begin{proposition} \mbox{\bf{(After Kedlaya-Liu \cite[Lemma 5.2.10]{KL15})}}
We have the following identity:
\begin{displaymath}
\widetilde{\Pi}^{[s_1,r_1]}_{R,\mathbb{Q}_p\{Z_1,...,Z_d\}}\bigcap\widetilde{\Pi}^{[s_2,r_2]}_{R,\mathbb{Q}_p\{Z_1,...,Z_d\}}=\widetilde{\Pi}^{[s_1,r_2]}_{R,\mathbb{Q}_p\{Z_1,...,Z_d\}},
\end{displaymath}
here the radii satisfy $<s_1\leq s_2 \leq r_1 \leq r_2$.
\end{proposition}

\begin{proof}
In our situation one direction is obvious while on the other hand we consider any element $x$ in the intersection on the left, then by the previous proposition we	have the decomposition $x=y+z$ where $y\in \widetilde{\Pi}^{\mathrm{int},r_1}_{R,\mathbb{Q}_p\{Z_1,...,Z_d\}}$ and $z\in \widetilde{\Pi}^{[s_1,r_2]}_{R,\mathbb{Q}_p\{Z_1,...,Z_d\}}$. Then as in \cite[Lemma 5.2.10]{KL15} section 5.2 we look at $y=x-z$ which lives in the intersection:
\begin{displaymath}
\widetilde{\Pi}^{\mathrm{int},r_1}_{R,\mathbb{Q}_p\{Z_1,...,Z_d\}}\bigcap	\widetilde{\Pi}^{[s_2,r_2]}_{R,\mathbb{Q}_p\{Z_1,...,Z_d\}}=\widetilde{\Pi}^{\mathrm{int},r_2}_{R,\mathbb{Q}_p\{Z_1,...,Z_d\}}
\end{displaymath}
which finishes the proof.
\end{proof}

\begin{proposition}\mbox{\bf{(After Kedlaya-Liu \cite[Lemma 5.2.10]{KL15})}}
We have the following identity:
\begin{displaymath}
\widetilde{\Pi}^{[s_1,r_1]}_{R,A}\bigcap\widetilde{\Pi}^{[s_2,r_2]}_{R,A}=\widetilde{\Pi}^{[s_1,r_2]}_{R,A},
\end{displaymath}
here the radii satisfy $<s_1\leq s_2 \leq r_1 \leq r_2$.
	
\end{proposition}

\begin{proof}
See the proof of \cref{proposition5.7}.		
\end{proof}

\begin{remark}
Again this subsection is finished so far only for the situation where $E$ is of mixed characteristic. But everything uniformly carries over for our original assumption on the field $E$ and $A$. We will not repeat the proof again.	
\end{remark}

\subsection{Noncommutative Period Rings and Period Sheaves}

\begin{setting}
We will work in the categories of the pseudocoherent, fpd and finite projective modules over the period rings defined above. First we specify the Frobenius in our setting. The rings involved are:
\begin{align}
\widetilde{\Omega}^\mathrm{int}_{R,A},\widetilde{\Omega}_{R,A}, \widetilde{\Pi}^\mathrm{int}_{R,A}, \widetilde{\Pi}^{\mathrm{int},r}_{R,A},\widetilde{\Pi}^\mathrm{bd}_{R,A},\widetilde{\Pi}^{\mathrm{bd},r}_{R,A}, \widetilde{\Pi}_{R,A}, \widetilde{\Pi}^r_{R,A},\widetilde{\Pi}^+_{R,A}, \widetilde{\Pi}^\infty_{R,A},\widetilde{\Pi}^I_{R,A}.
\end{align}
We are going to endow these rings with the Frobenius induced by continuation from the Witt vector part only, which is to say the corresponding Frobenius induced by the $p^h$-power absolute Frobenius over $R$. Note all the rings above are defined by taking the product of $\triangle$ where each $\triangle$ representing one of the following rings (over $E$):
\begin{align}
\widetilde{\Omega}^\mathrm{int}_{R},\widetilde{\Omega}_{R}, \widetilde{\Pi}^\mathrm{int}_{R}, \widetilde{\Pi}^{\mathrm{int},r}_{R},\widetilde{\Pi}^\mathrm{bd}_{R},\widetilde{\Pi}^{\mathrm{bd},r}_{R}, \widetilde{\Pi}_{R}, \widetilde{\Pi}^r_{R},\widetilde{\Pi}^+_{R}, \widetilde{\Pi}^\infty_{R},\widetilde{\Pi}^I_{R}
\end{align}
with the affinoid ring $A$. The Frobenius acts on $A$ trivially and we assume that the action is $A$-linear.  
\end{setting}

\indent First we consider the following sheafification as in \cite[Definition 4.4.2]{KL16}:

\begin{setting} 
Consider the space $X=\mathrm{Spa}(R,R^+)$, over this perfectoid space there were sheaves:
\begin{align}
\widetilde{\Omega}^\mathrm{int}_{},\widetilde{\Omega}_{}, \widetilde{\Pi}^\mathrm{int}_{}, \widetilde{\Pi}^{\mathrm{int},r}_{},\widetilde{\Pi}^\mathrm{bd}_{},\widetilde{\Pi}^{\mathrm{bd},r}_{}, \widetilde{\Pi}_{}, \widetilde{\Pi}^r_{},\widetilde{\Pi}^+_{}, \widetilde{\Pi}^\infty_{},\widetilde{\Pi}^I_{}.
\end{align}
defined over this space through the corresponding adic, \'etale, pro-\'etale or $v$-topology, we consider the corresponding sheaves defined over the same Grothendieck sites but with further deformed consideration:
\begin{align}
\widetilde{\Omega}^\mathrm{int}_{*,A},\widetilde{\Omega}_{*,A}, \widetilde{\Pi}^\mathrm{int}_{*,A}, \widetilde{\Pi}^{\mathrm{int},r}_{*,A},\widetilde{\Pi}^\mathrm{bd}_{*,A},\widetilde{\Pi}^{\mathrm{bd},r}_{*,A}, \widetilde{\Pi}_{*,A}, \widetilde{\Pi}^r_{*,A},\widetilde{\Pi}^+_{*,A}, \widetilde{\Pi}^\infty_{*,A},\widetilde{\Pi}^I_{*,A}.
\end{align}
\end{setting}

\begin{remark}
This is well defined, since we have orthogonal basis for the ring $A$.	
\end{remark}

\begin{definition}
In our situation we consider the corresponding Frobenius action on the following period rings and sheaves:
\begin{align}
\widetilde{\Omega}^\mathrm{int}_{R,A},\widetilde{\Omega}^\mathrm{int}_{R,A},\widetilde{\Pi}^\mathrm{int}_{R,A},\widetilde{\Pi}^\mathrm{bd}_{R,A},\widetilde{\Pi}_{R,A},\widetilde{\Pi}^\infty_{R,A}, \widetilde{\Omega}^\mathrm{int}_{*,A}, \widetilde{\Pi}^\mathrm{int}_{*,A}, \widetilde{\Pi}^+_{*,A}, \widetilde{\Omega}_{*,A}, \widetilde{\Pi}^\mathrm{bd}_{*,A}, \widetilde{\Pi}^\infty_{*,A}, \widetilde{\Pi}^+_{*,A}, \widetilde{\Pi}_{*,A},   
\end{align}
which is defined by considering the corresponding lift of the absolute Frobenius of $p^h$-power in characteristic $p>0$ induced from $R$, which will be denoted by $\varphi$. We then introduce more general consideration by taking $E_a$ to be some unramified extension of $E$ of degree $a$ divisible by $h$, the corresponding Frobenius will be denoted by $\varphi^a$.
\end{definition}

\indent Then we generalize the corresponding Frobenius modules in \cite[Definition 4.4.4]{KL16} to our situation as in the following.

\begin{definition}
Over the period rings and sheaves (each is denoted by $\triangle$ in this definition) defined in the previous definition we define as in \cite[Definition 4.4.4]{KL16} the corresponding $\varphi^a$-modules over $\triangle$ which are respectively projective to be the corresponding finite projetive modules over $\triangle$ with further assigned semilinear action of the operator $\varphi^a$. Here we define in our situation the corresponding $\varphi^a$-cohomology to be the (hyper)-cohomology of the following complex:
\[
\xymatrix@R+0pc@C+0pc{
0\ar[r]\ar[r]\ar[r] &M
\ar[r]^{\varphi-1}\ar[r]\ar[r] &M
\ar[r]\ar[r]\ar[r] &0.
}
\]
All modules are right over the noncommutative rings. 
\end{definition}

\indent Now we define the corresponding modules over the rings which are the domains in the following morphisms induced from the Frobenius map $\varphi^a$:
\begin{align} 
\widetilde{\Pi}^{\mathrm{int},r}_{R,A}\rightarrow \widetilde{\Pi}^{\mathrm{int},rp^{-ha}}_{R,A},\widetilde{\Pi}^{\mathrm{bd},r}_{R,A}\rightarrow \widetilde{\Pi}^{\mathrm{bd},rp^{-ha}}_{R,A},\widetilde{\Pi}^{r}_{R,A}\rightarrow \widetilde{\Pi}^{rp^{-ha}}_{R,A}\\
\widetilde{\Pi}^{\mathrm{int},r}_{*,A}\rightarrow \widetilde{\Pi}^{\mathrm{int},rp^{-ha}}_{*,A},\widetilde{\Pi}^{\mathrm{bd},r}_{*,A}\rightarrow \widetilde{\Pi}^{\mathrm{bd},rp^{-ha}}_{*,A},\widetilde{\Pi}^{r}_{*,A}\rightarrow \widetilde{\Pi}^{rp^{-ha}}_{*,A}.	
\end{align}

\begin{definition}
Over each rings $\triangle$ which are the domains in the morphisms as mentioned just before this definition, we define the corresponding projective $\varphi^a$-module over any $\triangle$ listed above to be the corresponding finite projective module $M$ over $\triangle$ with additionally endowed semilinear Frobenius action from $\varphi^a$ such that we have the isomorphism $\varphi^{a*}M\overset{\sim}{\rightarrow}M\otimes \square$ where the ring $\square$ is one of the targets listed above. Also the cohomology of any module under this definition will be defined to be the (hyper)cohomology of the complex in the following form:
\[
\xymatrix@R+0pc@C+0pc{
0\ar[r]\ar[r]\ar[r] &M
\ar[r]^{\varphi-1}\ar[r]\ar[r] &M\otimes_\triangle \square
\ar[r]\ar[r]\ar[r] &0.
}
\]
All modules are right over the noncommutative rings.
\end{definition}

\indent Then we consider the following morphisms of specific period rings induced by the Frobenius. 
\begin{align}
\widetilde{\Pi}_{R,A}^{[s,r]}	\rightarrow \widetilde{\Pi}_{R,A}^{[sp^{-ah},rp^{-ah}]}\\
\widetilde{\Pi}_{*,A}^{[s,r]}	\rightarrow \widetilde{\Pi}_{*,A}^{[sp^{-ah},rp^{-ah}]}
\end{align}

with the corresponding morphisms in the following:

\begin{align}
\widetilde{\Pi}_{R,A}^{[s,r]}	\rightarrow \widetilde{\Pi}_{R,A}^{[s,rp^{-ah}]}\\
\widetilde{\Pi}_{*,A}^{[s,r]}	\rightarrow \widetilde{\Pi}_{*,A}^{[s,rp^{-ah}]}
\end{align}

\begin{definition}
Again as in \cite[Definition 4.4.4]{KL16}, we define the corresponding projective $\varphi^a$-modules over the domain rings or sheaves of rings in the morphisms just before this definition to be the finite projective modules (which will be denoted by $M$) over the domain rings in the morphism just before this definition additionally endowed with semilinear Frobenius action from $\varphi^a$ with the following isomorphisms:
\begin{align}
\varphi^{a*}M\otimes_{\widetilde{\Pi}_{R,A}^{[sp^{-ah},rp^{-ah}]}}\widetilde{\Pi}_{R,A}^{[s,rp^{-ah}]}\overset{\sim}{\rightarrow}M\otimes_{\widetilde{\Pi}_{R,A}^{[s,r]}}\widetilde{\Pi}_{R,A}^{[s,rp^{-ah}]},\\
\varphi^{a*}M\otimes_{\widetilde{\Pi}_{*,A}^{[sp^{-ah},rp^{-ah}]}}\widetilde{\Pi}_{*,A}^{[s,rp^{-ah}]}\overset{\sim}{\rightarrow}M\otimes_{\widetilde{\Pi}_{*,A}^{[s,r]}}\widetilde{\Pi}_{*,A}^{[s,rp^{-ah}]}.
\end{align}
All modules are right over the noncommutative rings.
\end{definition}

\noindent Also one can further define the corresponding bundles carrying semilinear Frobenius in our context as in the situation of \cite[Definition 4.4.4]{KL16}:

\begin{definition}
Over the ring $\widetilde{\Pi}_{R,A}$ we define a corresponding projective Frobenius bundle to be a family $(M_I)_I$ of finite projective modules over each $\widetilde{\Pi}^I_{R,A}$ carrying the natural Frobenius action coming from the operator $\varphi^a$ such that for any two involved intervals having the relation $I\subset J$ we have:
\begin{displaymath}
M_J\otimes_{\widetilde{\Pi}^J_{R,A}}\widetilde{\Pi}^I_{R,A}\overset{\sim}{\rightarrow}	M_I
\end{displaymath}
with the obvious cocycle condition. Here we have to propose condition on the intervals that for each $I=[s,r]$ involved we have $s\leq rp^{ah}$. All modules are right over the noncommutative rings.
\end{definition}

%
%
%

\indent Then we generalize the comparison of local systems and Frobenius modules in \cite[Theorem 4.5.7]{KL16} to our noncommutative context as in the following:

\begin{theorem} \mbox{\bf{(After Kedlaya-Liu \cite[Theorem 4.5.7]{KL16})}}
1. There is a fully faithful embedding functor from the category of the projective right $\mathcal{O}_{E_a^{\varphi^a}}\widehat{\otimes}\mathcal{O}_A$-local systems over $X_\text{pro\'et}$ to the following two categories:\\
1(a).  The category of projective right Frobenius modules over the period sheaf $\widetilde{\Omega}_{*,A}^\mathrm{int}$ over $X$, $X_\text{\'et}$, $X_\text{pro\'et}$;\\
1(b).  The category of projective right Frobenius modules over the period sheaf $\widetilde{\Pi}_{*,A}^\mathrm{int}$ over $X$, $X_\text{\'et}$, $X_\text{pro\'et}$.
\end{theorem}

\begin{proof}
See \cref{theorem4.4}.	
\end{proof}

\indent Then as in \cite[Corollary 4.5.8]{KL16} we consider the setting of more general adic spaces:

\begin{proposition} \mbox{\bf{(After Kedlaya-Liu \cite[Corollary 4.5.8]{KL16})}}
1. Let $X$	be a preadic space over $E_a$. Then we have that there is fully faithful embedding functor from the category of all the right $\mathcal{O}_{E_{a}^{\varphi^a}}\widehat{\otimes}\mathcal{O}_A$-local systems (projective) over $X$, $X_{\text{\'et}}$ and $X_\text{pro\'et}$ to the category of all the projective right Frobenius modules over the sheaves over $\widetilde{\Pi}_{R,A}^\mathrm{int}$ over the corresponding sites;\\
2. Let $X$	be a preadic space over $E_{\infty,a}$. Then we have that there is fully faithful embedding functor from the category of all the right $\mathcal{O}_{E_{\infty,a}^{\varphi^a}}\widehat{\otimes}\mathcal{O}_A$-local systems (projective) over $X$, $X_{\text{\'et}}$ and $X_\text{pro\'et}$ to the category of all the projective right Frobenius modules over the sheaves over $\widetilde{\Pi}_{R,\infty,A}^\mathrm{int}$ over the corresponding sites.
\end{proposition}

\begin{proof}
For one, apply the previous theorem. For two, repeat the argument in the proof of the previous theorem.	
\end{proof}

\indent The following definition is kind of generalization of the corresponding one in \cite[Definition 4.5.9]{KL16}:

\begin{definition}
Now over the ring $\widetilde{\Pi}_{R,A}$ or $\widetilde{\Pi}^\mathrm{bd}_{R,A}$	we call the corresponding Frobenius modules globally \'etale if they arise from the Frobenius modules over the ring $\widetilde{\Pi}^\mathrm{int}_{R,A}$. Now over the ring $\widetilde{\Pi}_{R,\infty,A}$ or $\widetilde{\Pi}^\mathrm{bd}_{R,\infty,A}$	we call the corresponding Frobenius modules globally \'etale if they arise from the Frobenius modules over the ring $\widetilde{\Pi}^\mathrm{int}_{R,\infty,A}$. \\
Now over the sheaf $\widetilde{\Pi}_{*,A}$ or $\widetilde{\Pi}^\mathrm{bd}_{*,A}$	we call the corresponding Frobenius modules globally \'etale if they arise from the Frobenius modules over the sheaf $\widetilde{\Pi}^\mathrm{int}_{*,A}$. Now over the ring $\widetilde{\Pi}_{*,\infty,A}$ or $\widetilde{\Pi}^\mathrm{bd}_{*,\infty,A}$	we call the corresponding Frobenius modules globally \'etale if they arise from the Frobenius modules over the sheaf $\widetilde{\Pi}^\mathrm{int}_{*,\infty,A}$.

\end{definition}

\begin{proposition} \mbox{\bf{(After Kedlaya-Liu \cite[Theorem 4.5.11]{KL16})}}
1. Let $X$ be a preadic space over $E_{a}$. Then we have that there is a fully faithful embedding of the category of the corresponding right $E^{\varphi^a}_{a}\widehat{\otimes}A$-local systems (in the projective setting) into the category of the corresponding projective right Frobenius $\varphi^a$-modules over $\widetilde{\Pi}_{*,A}$;\\
2. Let $X$ be a preadic space over $E_{\infty,a}$. Then we have that there is a fully faithful embedding of the category of the corresponding right $E^{\varphi^a}_{\infty,a}\widehat{\otimes}A$-local systems (in the projective setting) into the category of the corresponding projective right  Frobenius $\varphi^a$-modules over $\widetilde{\Pi}_{*,A}$;\\
\end{proposition}

\begin{proof}
As in \cite[Theorem 4.5.11]{KL16}, we consider the corresponding base change of the corresponding exact sequence in \cref{proposition4.1} which reflects an exact sequence on the sheaves.
\end{proof}

\begin{lemma} \mbox{\bf{(After Kedlaya-Liu \cite[6.2.2-6.2.4]{KL15}, \cite[Lemma 4.6.9]{KL16}})}\\ \mbox{\bf{(And also see \cite[Proposition 2.11]{T1})}}
For any Frobenius $\varphi^a$-bundle $M$ over $\widetilde{\Pi}_{R,A}$, then we have that for any interval $I=[s,r]$ where $0<s\leq r$ the map $\varphi^a-1: M_{[s,rq]}\rightarrow M_{[s,r]}$ is surjective after taking some Frobenius twist as in \cite[6.2.2]{KL15} and our previous work namely the new morphism $\varphi^a-1: M(n)_{[s,rq]}\rightarrow M(n)_{[s,r]}$ for sufficiently large $n\geq 1$. As in the previous established version one may have the chance to take the number to be $1$ is the bundle initially comes from the corresponding base change from the integral Robba ring. 
\end{lemma}

\begin{proof}
See the proof of \cite[Proposition 2.11]{T1}.

\end{proof}

%

\begin{lemma} \mbox{\bf{(After Kedlaya-Liu \cite[6.2.2-6.2.4]{KL15}, \cite[Lemma 4.6.9]{KL16}})}\\ \mbox{\bf{(And also see \cite[Proposition 2.11, Proposition 2.14]{T1})}}
For any Frobenius $\varphi^a$-module $M$ over $\widetilde{\Pi}_{R,A}$, we have that for suifficiently large number $n\geq 1$ the space $H^1_{\varphi^a}(M(n))$ vanishes.	
\end{lemma}

\begin{proof}
See the proof of \cite[Proposition 2.11, Proposition 2.14]{T1}.	
\end{proof}

\indent Then we have the following key corollary which is analog of \cite[Corollary 6.2.3, Lemma 6.3.3]{KL15}, \cite[Corollary 4.6.10]{KL16} and \cite[Corollary 2.13, Corollary 3.4, Proposition 3.9]{T1}:

\begin{corollary} \label{2coro4.16}
I. When $M_\alpha,M,M_\beta$ are three Frobenius $\varphi^a$-bundles over the ring $\widetilde{\Pi}_{R,A}$ then we have that for sufficiently large $n\geq 0$ we have the following exact sequence:
\[
\xymatrix@R+0pc@C+0pc{
0\ar[r]\ar[r]\ar[r] &M_{\alpha,I}(n)^{\varphi^a}
\ar[r]\ar[r]\ar[r] &M_I(n)^{\varphi^a}
\ar[r]\ar[r]\ar[r] &M_{\beta,I}(n)^{\varphi^a} \ar[r]\ar[r]\ar[r] &0,
}
\]	
for each interval $I$.	\\
II. When $M_\alpha,M,M_\beta$ are three Frobenius $\varphi^a$-modules over the ring $\widetilde{\Pi}_{R,A}$ then we have that for sufficiently large $n\geq 0$ we have the following exact sequence:
\[
\xymatrix@R+0pc@C+0pc{
0\ar[r]\ar[r]\ar[r] &M_{\alpha,I}(n)^{\varphi^a}
\ar[r]\ar[r]\ar[r] &M_I(n)^{\varphi^a}
\ar[r]\ar[r]\ar[r] &M_{\beta,I}(n)^{\varphi^a} \ar[r]\ar[r]\ar[r] &0,
}
\]	
for each interval $I$.\\
III. For a Frobenius $\varphi^a$-bundle $M$ over $\widetilde{\Pi}_{R,A}$, and for each module $M_I$ over some $\widetilde{\Pi}_{R,A}^{I}$, we put for any element $f$ such that $\varphi^af=p^df$:
\begin{displaymath}
M_{I,f}:=\bigcup_{n\in \mathbb{Z}}f^{-n}M_I(dn)^{\varphi^a}.	
\end{displaymath}
Then with this convention suppose we have three Frobenius $\varphi^a$-bundles taking the form of $M_\alpha,M,M_\beta$ over
$\widetilde{\Pi}_{R,A}$, then for each closed interval $I$ we have the following is an exact sequence:
\[
\xymatrix@R+0pc@C+0pc{
0\ar[r]\ar[r]\ar[r] &M_{\alpha,I,f}\ar[r]\ar[r]\ar[r] &M_{I,f}
\ar[r]\ar[r]\ar[r] &M_{\beta,I,f} \ar[r]\ar[r]\ar[r] &0,
}
\]
where each module in the exact sequence is now a module over $\widetilde{\Pi}_{R,A}[1/f]^{\varphi^a}$;\\
IV. For a Frobenius $\varphi^a$-module $M$ over $\widetilde{\Pi}_{R,A}$, we put for any element $f$ such that $\varphi^af=p^df$:
\begin{displaymath}
M_{f}:=\bigcup_{n\in \mathbb{Z}}f^{-n}M(dn)^{\varphi^a}.	
\end{displaymath}
Then with this convention suppose we have three Frobenius $\varphi^a$-modules taking the form of $M_\alpha,M,M_\beta$ over
$\widetilde{\Pi}_{R,A}$, then we have the following is an exact sequence:
\[
\xymatrix@R+0pc@C+0pc{
0\ar[r]\ar[r]\ar[r] &M_{\alpha,f}\ar[r]\ar[r]\ar[r] &M_{f}
\ar[r]\ar[r]\ar[r] &M_{\beta,f} \ar[r]\ar[r]\ar[r] &0,
}
\]
where each module in the exact sequence is now a module over $\widetilde{\Pi}_{R,A}[1/f]^{\varphi^a}$.\\

\end{corollary}

\begin{proof}
See the proof of \cite[Corollary 6.2.3, Lemma 6.3.3]{KL15}, \cite[Corollary 4.6.10]{KL16} and \cite[Corollary 2.13, Corollary 3.4, Proposition 3.9]{T1}.	
\end{proof}

\subsection{Noncommutative Imperfect Setting}

\noindent We now consider the corresponding imperfectization of the corresponding constructions we considered above, after \cite{KL16}. We will consider the corresponding towers in \cite[Chapter 5]{KL16}, so we keep the assumption on the towers as in the commutative setting, namely:

\begin{assumption}
In this section, we are going to assume that $k$ is just $\mathbb{F}_{p^h}$.
\end{assumption}


\indent Now we describe the corresponding imperfect period rings which we will deform. These rings are those introduced in \cite[Definition 5.2.1]{KL16} by using series of imperfection processes. Recall in more detail in our noncommutative setting we have:

\begin{setting} \label{2setting6.3}
Fix a perfectoid tower $(H_\bullet,H^+_\bullet)$ Recall from \cite[Definition 5.2.1]{KL16} we have the following different imperfect constructions:\\
A. First we have the ring $\overline{H'}_\infty$, which could give us the corresponding ring $\widetilde{\Omega}^{\mathrm{int}}_{\overline{H'}_\infty}$. Taking the corresponding product with $A$ we have the corresponding deformed period ring;\\
B. We then have the ring $\widetilde{\Pi}^{\mathrm{int},r}_{\overline{H'}_\infty}$ coming from $\overline{H'}_\infty$. Taking the corresponding product with $A$ we have the corresponding deformed period ring;\\
C. We then have the ring $\widetilde{\Pi}^{\mathrm{int}}_{\overline{H'}_\infty}$ coming from $\overline{H'}_\infty$. Taking the corresponding product with $A$ we have the corresponding deformed period ring;\\
D. We then have the ring $\widetilde{\Pi}^{\mathrm{bd},r}_{\overline{H'}_\infty}$ coming from $\overline{H'}_\infty$. Taking the corresponding product with $A$ we have the corresponding deformed period ring;\\
E. We then have the ring $\widetilde{\Pi}^{\mathrm{bd}}_{\overline{H'}_\infty}$ coming from $\overline{H'}_\infty$. Taking the corresponding product with $A$ we have the corresponding deformed period ring;\\
F. We then have the ring $\widetilde{\Pi}^{r}_{\overline{H'}_\infty}$ coming from $\overline{H'}_\infty$. Taking the corresponding product with $A$ we have the corresponding deformed period ring;\\
G. We then have the ring $\widetilde{\Pi}^{[s,r]}_{\overline{H'}_\infty}$ coming from $\overline{H'}_\infty$. Taking the corresponding product with $A$ we have the corresponding deformed period ring;\\
H. We then have the ring $\widetilde{\Pi}_{\overline{H'}_\infty}$ coming from $\overline{H'}_\infty$. Taking the corresponding product with $A$ we have the corresponding deformed period ring;\\
I. $\Pi^{\mathrm{int},r}_{H}$ comes from the ring $\widetilde{\Pi}^{\mathrm{int},r}_{\overline{H'}_\infty}$ consisting of those elements of $\widetilde{\Pi}^{\mathrm{int},r}_{\overline{H'}_\infty}$ with the requirement that whenever we have $n$ an integer such that $nh>-\log_pr$ we have then $\theta(\varphi^{-n}(x))\in H_n$. Taking the corresponding product with $A$ we have the corresponding deformed period ring;\\
J. $\Pi^{\mathrm{int},\dagger}_{H}$ is defined to be the corresponding union of the rings in $I$. Taking the corresponding product with $A$ we have the corresponding deformed period ring;\\
K. $\Omega^\mathrm{int}_H$ is defined to be the corresponding period ring coming from the corresponding $\pi$-adic completion of the ring $\Pi^{\mathrm{int},\dagger}_{H}$ in $J$. Taking the corresponding product with $A$ we have the corresponding deformed period ring;\\
L. $\breve{\Omega}^{\mathrm{int}}_{H}$ is the ring which is defined to be the union of all the $\varphi^{-n}\Omega^\mathrm{int}_H$. Taking the corresponding product with $A$ we have the corresponding deformed period ring;\\
M. $\breve{\Pi}^{\mathrm{int},r}_H$ is then the ring which is defined to be the union of all the $\varphi^{-n}\Pi^{\mathrm{int},p^{hn}r}_H$. Taking the corresponding product with $A$ we have the corresponding deformed period ring;\\
N. $\breve{\Pi}^{\mathrm{int},\dagger}_H$ is define to be union of all the $\varphi^{-n}\breve{\Pi}^{\mathrm{int},\dagger}_H$. Taking the corresponding product with $A$ we have the corresponding deformed period ring;\\
O. $\widehat{\Omega}^{\mathrm{int}}_{H}$ is defined to be the $\pi$-completion of $\breve{\Omega}^{\mathrm{int}}_{H}$. Taking the corresponding product with $A$ we have the corresponding deformed period ring;\\
P. $\widehat{\Pi}^{\mathrm{int},r}_H$ is defined to be the $\mathrm{max}\{\|.\|_{\overline{\alpha}_\infty^r},\|.\|_{\pi-\text{adic}}\}$-completion of $\breve{\Pi}^{\mathrm{int},r}_H$. Taking the corresponding product with $A$ we have the corresponding deformed period ring;\\
Q. We then have $\widehat{\Pi}^{\mathrm{int},\dagger}_H$ by taking the union over $r>0$. Taking the corresponding product with $A$ we have the corresponding deformed period ring;\\
R. Correspondingly we have $\breve{\Omega}_{H}$, $\breve{\Pi}^{\mathrm{bd},r}_H$, $\breve{\Pi}^{\mathrm{bd},\dagger}_H$ by inverting the element $\pi$. Taking the corresponding product with $A$ we have the corresponding deformed period ring;\\
S. Correspondingly we also have $\widehat{\Omega}_{H}$, $\widehat{\Pi}^{\mathrm{bd},r}_H$, $\widehat{\Pi}^{\mathrm{bd},\dagger}_H$ by inverting the corresponding element $\pi$. Taking the corresponding product with $A$ we have the corresponding deformed period ring;\\
T. We also have $\Omega_{H}$, $\Pi^{\mathrm{bd},r}_H$, $\Pi^{\mathrm{bd},\dagger}_H$ again by inverting the element $\pi$. Taking the corresponding product with $A$ we have the corresponding deformed period ring;\\
U. Taking the $\max\{\|.\|_{\overline{\alpha}_\infty^s},\|.\|_{\overline{\alpha}_\infty^r}\}$ (for $0<s\leq r$) completion of the ring $\Pi^{\mathrm{bd},r}_H$ we have the ring $\Pi^{[s,r]}_{H}$, while taking the Fr\'echet completion with respect to the norm $\|.\|_{\overline{\alpha}^t_\infty}$ for each $0<t\leq r$ we have the ring $\Pi^r_{H}$. Taking the corresponding product with $A$ we have the corresponding deformed period ring;\\
V. Taking the union we have the ring $\Pi_H$. Taking the corresponding product with $A$ we have the corresponding deformed period ring;\\
W. We use the notation $\breve{\Pi}_H$ to denote the corresponding union of all the $\varphi^{-n}\Pi_H$. Taking the corresponding product with $A$ we have the corresponding deformed period ring;\\
X. We use the notation $\breve{\Pi}^{[s,r]}_{H}$ to be the corresponding union of all the $\varphi^{-n}\Pi^{[p^{-hn}s,p^{-hn}r]}_H$. Taking the corresponding product with $A$ we have the corresponding deformed period ring;\\
Y. We use the notation $\breve{\Pi}_H^r$ to be the union of all the $\varphi^{-n}\breve{\Pi}_H^{p^{-hn}r}$. Taking the corresponding product with $A$ we have the corresponding deformed period ring.\\
\end{setting}

\indent Then we have the following direct analog of the relative version of the ring defined above (here as before the ring $A$ denotes a Banach affinoid algebra in the noncommutative setting):

\begin{setting} \label{2setting6.4}
Now we consider the deformation of the rings above:\\
I. We have the first group of the period rings in the deformed setting:
\begin{align}
\Pi^{\mathrm{int},r}_{H,A},\Pi^{\mathrm{int},\dagger}_{H,A},\Omega^\mathrm{int}_{H,A}, \Omega_{H,A}, \Pi^{\mathrm{bd},r}_{H,A}, \Pi^{\mathrm{bd},\dagger}_{H,A},\Pi^{[s,r]}_{H,A}, \Pi^r_{H,A}, \Pi_{H,A}.
\end{align}
II. We also have the second group of desired rings in the desired setting:
\begin{align}
\breve{\Pi}^{\mathrm{int},r}_{H,A},\breve{\Pi}^{\mathrm{int},\dagger}_{H,A},\breve{\Omega}^\mathrm{int}_{H,A}, \breve{\Omega}_{H,A}, \breve{\Pi}^{\mathrm{bd},r}_{H,A}, \breve{\Pi}^{\mathrm{bd},\dagger}_{H,A},\breve{\Pi}^{[s,r]}_{H,A}, \breve{\Pi}^r_{H,A}, \breve{\Pi}_{H,A}.	
\end{align}
III. We also have the third group:
\begin{align}
\widehat{\Pi}^{\mathrm{int},r}_{H,A},\widehat{\Pi}^{\mathrm{int},\dagger}_{H,A},\widehat{\Omega}^\mathrm{int}_{H,A}, \widehat{\Omega}_{H,A}, \widehat{\Pi}^{\mathrm{bd},r}_{H,A}, \widehat{\Pi}^{\mathrm{bd},\dagger}_{H,A}.	
\end{align}
\end{setting}

\indent Now we discuss some properties of the corresponding deformed version of the imperfect rings in our context, which is parallel to the corresponding discussion we made in the perfect setting and generalizing the corresponding discussion in \cite{KL16}, again in our noncommutative setting:

\begin{proposition}\mbox{\bf{(After Kedlaya-Liu \cite[Lemma 5.2.10]{KL16})}}
For any $0< r_1\leq r_2$ we have the following equality on the corresponding period rings:
\begin{displaymath}
\Pi^{\mathrm{int},r_1}_{H,\mathbb{Q}_p\{Z_1,...,Z_d\}}\bigcap	\Pi^{[r_1,r_2]}_{H,\mathbb{Q}_p\{Z_1,...,Z_d\}}=\Pi^{\mathrm{int},r_2}_{H,\mathbb{Q}_p\{Z_1,...,Z_d\}}.
\end{displaymath}	
\end{proposition}

\begin{proof}
We adapt the argument in \cite[Lemma 5.2.10]{KL16} to prove this in the situation where $r_1<r_2$ (otherwise this is trivial), again one direction is easy where we only present the implication in the other direction. We take any element $x\in \Pi^{\mathrm{int},r_1}_{H,\mathbb{Q}_p\{Z_1,...,Z_d\}}\bigcap	\Pi^{[r_1,r_2]}_{H,\mathbb{Q}_p\{Z_1,...,Z_d\}}$ and take suitable approximating elements $\{x_i\}$ living in the bounded Robba ring such that for any $j\geq 1$ one can find some integer $N_j\geq 1$ we have whenever $i\geq N_j$ we have the following estimate:
\begin{displaymath}
\left\|.\right\|_{\overline{\alpha}_\infty^{t},\mathbb{Q}_p\{Z_1,...,Z_d\}}(x_i-x) \leq p^{-j}, \forall t\in [r_1,r_2].	
\end{displaymath}
Then we consider the corresponding decomposition of $x_i$ for each $i=1,2,...$ into a form having integral part and the rational part $x_i=y_i+z_i$ by setting
\begin{center}
 $y_i=\sum_{k=0,i_1,...,i_d}\pi^kx_{i,k,i_1,...,i_d}Z_1^{i_1}...Z_d^{i_d}$ 
\end{center} 
out of
\begin{center} 
$x_i=\sum_{k=n(x_i),i_1,...,i_d}\pi^kx_{i,k,i_1,...,i_d}Z_1^{i_1}...Z_d^{i_d}$.
\end{center}
Note that by our initial hypothesis we have that the element $x$ lives in the ring $\Pi^{\mathrm{int},r_1}_{H,\mathbb{Q}_p\{Z_1,...,Z_d\}}$ which further implies that 
\begin{displaymath}
\left\|.\right\|_{\overline{\alpha}_\infty^{r_1},\mathbb{Q}_p\{Z_1,...,Z_d\}}(\pi^kx_{i,k,i_1,...,i_d}Z_1^{i_1}...Z^{i_d}_d)	\leq p^{-j}.
\end{displaymath}
Therefore we have ${\overline{\alpha}_\infty}(\overline{x}_{i,k,i_1,...,i_d})\leq p^{(k-j)/r_1},\forall k< 0$ directly from this through computation, which implies that then:
\begin{align}
\left\|.\right\|_{\overline{\alpha}_\infty^{r_2},\mathbb{Q}_p\{Z_1,...,Z_d\}}(\pi^kx_{i,k,i_1,...,i_d}Z_1^{i_1}...Z^{i_d}_d)	&\leq p^{-k}p^{(k-j)r_2/r_1}\\
	&\leq p^{1+(1-j)r_1/r_1}.
\end{align}
Then one can read off the result directly from this estimate since under this estimate we can have the chance to modify the original approximating sequence $\{x_i\}$ by $\{y_i\}$ which are initially chosen to be in the integral Robba ring, which implies that actually the element $x$ lives in the right-hand side of the identity in the statement of the proposition.
\end{proof}

\begin{remark}
The following result cannot be derived from the previous proposition but the proof could be made essentially in the fashion, which is the same as the corresponding situation we encountered in the commutative situation, for instance see  \cref{proposition5.7}.	
\end{remark}

\begin{proposition}\mbox{\bf{(After Kedlaya-Liu \cite[Lemma 5.2.10]{KL16})}}
For any $0< r_1\leq r_2$ we have the following equality on the corresponding period rings:
\begin{displaymath}
\Pi^{\mathrm{int},r_1}_{H,A}\bigcap	\Pi^{[r_1,r_2]}_{H,A}=\Pi^{\mathrm{int},r_2}_{H,A}.
\end{displaymath}	
Here $A$ is some Banach affinoid algebra over $\mathbb{Q}_p$ in the noncommutative setting.	
\end{proposition}

\begin{proof}
See the proof of \cref{proposition5.7}.		
\end{proof}


\begin{proposition} \mbox{\bf{(After Kedlaya-Liu \cite[Lemma 5.2.8]{KL15} and \cite{KL16})}}
Consider now in our situation the radii $0< r_1\leq r_2$, and consider any element $x\in \Pi^{[r_1,r_2]}_{H,\mathbb{Q}_p\{Z_1,...,Z_d\}}$. Then we have that for each $n\geq 1$ one can decompose $x$ into the form of $x=y+z$ such that $y\in \pi^n\Pi^{\mathrm{int},r_2}_{H,\mathbb{Q}_p\{Z_1,...,Z_d\}}$ with $z\in \bigcap_{r\geq r_2}\Pi^{[r_1,r]}_{H,\mathbb{Q}_p\{Z_1,...,Z_d\}}$ with the following estimate for each $r\geq r_2$:
\begin{displaymath}
\left\|.\right\|_{\overline{\alpha}_\infty^r,\mathbb{Q}_p\{Z_1,...,Z_d\}}(z)\leq p^{(1-n)(1-r/r_2)}\left\|.\right\|_{\overline{\alpha}_\infty^{r_2},\mathbb{Q}_p\{Z_1,...,Z_d\}}(z)^{r/r_2}.	
\end{displaymath}

\end{proposition}

\begin{proof}
As in \cite[Lemma 5.2.8]{KL15} and \cite{KL16} and in the proof of our previous proposition we first consider those elements $x$ lives in the bounded Robba rings which could be expressed in general as
\begin{center}
 $\sum_{k=n(x),i_1,...,i_d}\pi^kx_{k,i_1,...,i_d}Z_1^{i_1}...Z_d^{i_d}$.
 \end{center}	
In this situation the corresponding decomposition is very easy to come up with, namely we consider the corresponding $y_i$ as the corresponding series:
\begin{displaymath}
\sum_{k\geq n,i_1,...,i_d}\pi^kx_{k,i_1,...,i_d}Z_1^{i_1}...Z_d^{i_d}	
\end{displaymath}
which give us the desired result since we have in this situation when focusing on each single term:
\begin{align}
\left\|.\right\|_{\overline{\alpha}_\infty^r,\mathbb{Q}_p\{Z_1,...,Z_d\}}(\pi^kx_{k,i_1,...,i_d}Z_1^{i_1}...Z_d^{i_d})&=p^{-k}\overline{\alpha}_\infty(\overline{x}_{k,i_1,...,i_d})^r\\
&=p^{-k(1-r/r_2)}\left\|.\right\|_{\overline{\alpha}_\infty^{r_2},\mathbb{Q}_p\{Z_1,...,Z_d\}}(\pi^kx_{k,i_1,...,i_d}Z_1^{i_1}...Z_d^{i_d})^{r/r_2}\\
&\leq p^{(1-n)(1-r/r_2)}\left\|.\right\|_{\overline{\alpha}_\infty^{r_2},\mathbb{Q}_p\{Z_1,...,Z_d\}}(\pi^kx_{k,i_1,...,i_d}Z_1^{i_1}...Z_d^{i_d})^{r/r_2}
\end{align}
for all those suitable $k$. Then to tackle the more general situation we consider the approximating sequence consisting of all the elements in the bounded Robba ring as in the usual situation considered in \cite[Lemma 5.2.8]{KL15} and \cite{KL16}, namely we inductively construct the following approximating sequence just as:
\begin{align}
\left\|.\right\|_{\overline{\alpha}_\infty^r,\mathbb{Q}_p\{Z_1,...,Z_d\}}(x-x_0-...-x_i)\leq p^{-i-1}	\left\|.\right\|_{\overline{\alpha}_\infty^r,\mathbb{Q}_p\{Z_1,...,Z_d\}}(x), i=0,1,..., r\in [r_1,r_2].
\end{align}
Here all the elements $x_i$ for each $i=0,1,...$ are living in the bounded Robba ring, which immediately gives rise to the suitable decomposition as proved in the previous case namely we have for each $i$ the decomposition $x_i=y_i+z_i$ with the desired conditions as mentioned in the statement of the proposition. We first take the series summing all the elements $y_i$ up for all $i=0,1,...$, which first of all converges under the norm $\left\|.\right\|_{\overline{\alpha}_\infty^r,\mathbb{Q}_p\{Z_1,...,Z_d\}}$ for all the radius $r\in [r_1,r_2]$, and also note that all the elements $y_i$ within the infinite sum live inside the corresponding integral Robba ring $\Pi^{\mathrm{int},r_2}_{H,\mathbb{Q}_p\{Z_1,...,Z_d\}}$, which further implies the corresponding convergence ends up in $\Pi^{\mathrm{int},r_2}_{H,\mathbb{Q}_p\{Z_1,...,Z_d\}}$. For the elements $z_i$ where $i=0,1,...$ also sum up to a converging series in the desired ring since combining all the estimates above we have:
\begin{displaymath}
\left\|.\right\|_{\overline{\alpha}_\infty^r,\mathbb{Q}_p\{Z_1,...,Z_d\}}(z_i)\leq p^{(1-n)(1-r/r_2)}\left\|.\right\|_{\overline{\alpha}_\infty^{r_2},\mathbb{Q}_p\{Z_1,...,Z_d\}}(x)^{r/r_2}.	
\end{displaymath}
\end{proof}

%

\begin{proposition} \mbox{\bf{(After Kedlaya-Liu \cite[Lemma 5.2.10]{KL15})}}
We have the following identity:
\begin{displaymath}
\Pi^{[s_1,r_1]}_{H,\mathbb{Q}_p\{Z_1,...,Z_d\}}\bigcap\Pi^{[s_2,r_2]}_{H,\mathbb{Q}_p\{Z_1,...,Z_d\}}=\Pi^{[s_1,r_2]}_{H,\mathbb{Q}_p\{Z_1,...,Z_d\}},
\end{displaymath}
here the radii satisfy $<s_1\leq s_2 \leq r_1 \leq r_2$.
\end{proposition}

\begin{proof}
In our situation one direction is obvious while on the other hand we consider any element $x$ in the intersection on the left, then by the previous proposition we	have the decomposition $x=y+z$ where $y\in \Pi^{\mathrm{int},r_1}_{H,\mathbb{Q}_p\{Z_1,...,Z_d\}}$ and $z\in \Pi^{[s_1,r_2]}_{H,\mathbb{Q}_p\{Z_1,...,Z_d\}}$. Then as in \cite[Lemma 5.2.10]{KL15} section 5.2 we look at $y=x-z$ which lives in the intersection:
\begin{displaymath}
\Pi^{\mathrm{int},r_1}_{H,\mathbb{Q}_p\{Z_1,...,Z_d\}}\bigcap	\Pi^{[s_2,r_2]}_{H,\mathbb{Q}_p\{Z_1,...,Z_d\}}=\Pi^{\mathrm{int},r_2}_{H,\mathbb{Q}_p\{Z_1,...,Z_d\}}
\end{displaymath}
which finishes the proof.
\end{proof}

\begin{proposition} \mbox{\bf{(After Kedlaya-Liu \cite[Lemma 5.2.10]{KL15})}}
We have the following identity:
\begin{displaymath}
\Pi^{[s_1,r_1]}_{H,A}\bigcap \Pi^{[s_2,r_2]}_{H,A}=\Pi^{[s_1,r_2]}_{H,A},
\end{displaymath}
here the radii satisfy $<s_1\leq s_2 \leq r_1 \leq r_2$.
	
\end{proposition}

\begin{proof}
See the proof of \cref{proposition5.7}.		
\end{proof}

\begin{remark}
Again one can follow the same strategy to deal with the corresponding equal-characteristic situation.	
\end{remark}

\subsection{Modules and Bundles over Noncommutative Rings}

Now we consider the modules and bundles over the rings introduced in the previous subsection. First we make the following assumption:

\begin{setting}
Recall that from \cite[Definition 5.2.3]{KL16} any tower $(H_\bullet,H_\bullet^+)$ is called weakly decompleting if we have that first the density of the perfection of $H_{\infty}$ in $\overline{H}_\infty$. Here the ring $H_\infty$ is the ring coming from the mod-$\pi$ construction of the ring $\Omega^\mathrm{int}_{H}$, also at the same time one can find some $r>0$ such that the corresponding modulo $\pi$ operation from the ring $\Omega^\mathrm{int}_{H}$ to the ring $H_\infty$ is actually surjective strictly. 
\end{setting}

\begin{assumption} \label{assumption642}
We now assume that we are basically in the situation where $(H_\bullet,H_\bullet^+)$ is actually weakly decompleting. Also as in \cite[Lemma 5.2.7]{KL16} we assume we fix some radius $r_0>0$, for instance this will correspond to the corresponding index in the situation we have the corresponding noetherian tower. Recall that a tower is called noetherian if we have some specific radius as above such that we have the strongly noetherian property on the ring $\Pi^{[s,r]}_{H}$ with $[s,r]\subset (0,r_0]$. We now assume that the tower is then noetherian and any ring $\Pi^{[s,r]}_{H,A}$ with $[s,r]\subset (0,r_0]$ is noetherian in the noncommutative setting. 
\end{assumption}

\begin{example}
This assumption could be achieved as in the following. First naively we can consider a noncommutative affinoid algebra which is finite over some commutative affinoid algebra over $\mathbb{Q}_p$, then tensor with the corresponding the Robba ring $\Pi^{[s,r]}_{H}$ which is assumed to be strongly noetherian. Then for some deeper case, we can take the Robba ring $D_{[a,b]}(\mathbb{Z}_p,\mathbb{Q}_p)$ attached to the group $\mathbb{Z}_p$ in the sense of \cite[Proposition 4.1]{Z1} (namely the usual Robba ring over $\mathbb{Q}_p$, here as in \cite[Proposition 4.1]{Z1} we assume that $a,b\in p^\mathbb{Q}$), and we take the corresponding complete tensor product with the local chart ring of the distribution algebra attached to a uniform $p$-adic Lie group $G$ as in \cite[Proposition 4.1]{Z1}:
\begin{displaymath}
D_\rho(G,\mathbb{Q}_p)	
\end{displaymath}
with some radius $\rho>0$ living in $p^\mathbb{Q}$. The whole product:
\begin{displaymath}
D_{[a,b]}(\mathbb{Z}_p,\mathbb{Q}_p)\widehat{\otimes}D_\rho(G,\mathbb{Q}_p)	
\end{displaymath}
is noetherian. Indeed, as in \cite[Proposition 4.1]{Z1} we apply the criterion in \cite[Proposition I.7.1.2]{LVO} to check that actually the graded ring induced by the product norm:
\begin{displaymath}
gr^\bullet_{\|.\|_\otimes}(D_{[a,b]}(\mathbb{Z}_p,\mathbb{Q}_p)\widehat{\otimes}D_\rho(G,\mathbb{Q}_p)	)\overset{\sim}{\rightarrow} gr^\bullet_{\|.\|_\otimes}(D_{[a,b]}(\mathbb{Z}_p,\mathbb{Q}_p)\otimes D_\rho(G,\mathbb{Q}_p)	)
\end{displaymath}
admits a surjection map from:
\begin{align}
gr^\bullet_{\|.\|_{[a,b]}}(D_{[a,b]}(\mathbb{Z}_p,\mathbb{Q}_p))\otimes_{gr^\bullet\mathbb{Q}_p} gr^\bullet_{\|.\|_\rho}(D_\rho(G,\mathbb{Q}_p)	)
\end{align}
As in the proof of \cite[Proposition 4.1]{Z1} one can show that this is a tensor product of a Laurent polynomial ring and a polynomial ring, which is noetherian.
\end{example}

\indent Then we can start to discuss the corresponding modules over the rings in our deformed setting, first as in \cite[Lemma 5.3.3]{KL16} the following result should be derived from our construction:

\indent Then we deform the basic notation of bundles in \cite[Definition 5.3.6]{KL16}:

\begin{definition} \mbox{\bf{(After Kedlaya-Liu \cite[Definition 5.3.6]{KL16})}}
We define the bundle over the ring $\Pi^{r_0}_{H,A}$ to be a collection $(M_I)_I$ of finite projective right modules over each $\Pi_{H,A}^{I}$ with $I\subset (0,r_0]$ closed subintervals of $(0,r_0]$ such that we have the following requirement in the glueing fashion. First for any $I_1\subset I_2$ two closed intervals we have $M_{I_2}\otimes_{\Pi_{H,A}^{I_2}}\Pi_{H,A}^{I_1}\overset{\sim}{\rightarrow} M_{I_1}$ with the obvious cocycle condition with respect to three closed subintervals of $(0,r_0]$ namely taking the form of $I_1\subset I_2\subset I_3$.\\
\indent We define the pseudocoherent sheaf over the ring $\Pi^{r_0}_{H,A}$ to be a collection $(M_I)_I$ of pseudocoherent right modules over each $\Pi_{H,A}^{I}$ with $I\subset (0,r_0]$ closed subintervals of $(0,r_0]$ such that we have the following requirement in the glueing fashion. First for any $I_1\subset I_2$ two closed intervals we have $M_{I_2}\otimes_{\Pi_{H,A}^{I_2}}\Pi_{H,A}^{I_1}\overset{\sim}{\rightarrow} M_{I_1}$ with the obvious cocycle condition with respect to three closed subintervals of $(0,r_0]$ namely taking the form of $I_1\subset I_2\subset I_3$. 	
\end{definition}

\indent We make the following discussion around the corresponding module and sheaf structures defined above.

\begin{lemma} \mbox{\bf{(After Kedlaya-Liu \cite[Lemma 5.3.8]{KL16})}}
We have the isomorphism between the ring $\Pi^r_{H,A}$ and the inverse limit of the ring $\Pi^{[s,r]}_{H,A}$ with respect to the radius $s$ by the map $\Pi^r_{H,A}\rightarrow \Pi^{[s,r]}_{H,A}$.	
\end{lemma}
 
\begin{proof}
As in \cite[Lemma 5.3.8]{KL16} it is injective by the isometry, and then use the corresponding elements $x_n,n=0,1,...$ in the dense ring $\Pi^{r}_{H,A}$ to approximate any element $x$ in the ring $\Pi^{[s,r]}_{H,A}$ in the same way as in \cite[Lemma 5.3.8]{KL16}:
\begin{displaymath}
\|.\|_{\overline{\alpha}_\infty^t,A}(x-x_n)\leq p^{n}	
\end{displaymath}
for any radius $t$ now living in the corresponding interval $[r2^{-n},r]$. This will establish Cauchy sequence which finishes the proof as in \cite[Lemma 5.3.8]{KL16}.
\end{proof}

\begin{proposition} \mbox{\bf{(After Kedlaya-Liu \cite[Lemma 5.3.9]{KL16})}}
For some radius $r\in (0,r_0]$. Suppose we have that $M$ is a vector bundle in the general setting or $M$ is a pseudocoherent sheaf in the setting where the tower is noetherian. Then we have that the corresponding global section is actually dense in each section with respect to some closed interval. And then we have the corresponding vanishing result of the first derived inverse limit functor.	
\end{proposition}

\begin{proof}
See \cite[Lemma 5.3.9]{KL16}. In our setting, note that we are actually in the formalism of the corresponding (noncommutative) Fr\'echet-Stein algebras as in \cite{ST1}.	
\end{proof}

\indent The interesting issue here as in \cite{KL16} is the corresponding finitely generated of the global section of a pseudocoherent sheaf which is actually not guaranteed in general. Therefore as in \cite{KL16} we have to distinguish the corresponding well-behaved sheaves out from the corresponding category of all the corresponding pseudocoherent sheaves. But we are not quite for sure how more complicated the noncommutative situation is than the commutative setting. For the completeness of the presentation we present some conjectures in our mind.

\begin{conjecture} \mbox{\bf{(After Kedlaya-Liu \cite[Lemma 5.3.10]{KL16})}}
As in the previous proposition we choose some $r\in (0,r_0]$. Now assume that the corresponding tower $(H_\bullet,H_\bullet^+)$ is noetherian. Now for any pseudocoherent sheaf $M$ defined above we have the following three statements are equivalent. A. The first statement is that one can find a sequence of positive integers $x_1,x_2,...$ such that for any closed interval living inside $(0,r]$ the section of the sheaf with respect this closed interval admits a projective resolution of modules with corresponding ranks bounded by the sequence of integer $x_1,x_2,...$. B. The second statement is that for any locally finite covering of the corresponding interval $(0,r]$ which takes the corresponding form of $\{I_i\}$ one can find a sequence of positive integers $x_1,x_2,...$ such that for any closed interval living inside $\{I_i\}$ the section of the sheaf with respect this closed interval admits a projective resolution of modules with corresponding ranks bounded by the sequence of integer $x_1,x_2,...$. C. Lastly the third statement is that the corresponding global section is a pseudocoherent module over the ring $\Pi^r_{H,A}$. 	
\end{conjecture}

%
%

\indent As in \cite[Definition 5.3.11]{KL16} we call the sheaf satisfies the corresponding equivalent conditions in the proposition above uniform pseudocoherent sheaf. Then we have the following analog of \cite[Lemma 5.3.12]{KL16}:

\begin{conjecture} \mbox{\bf{(After Kedlaya-Liu \cite[Lemma 5.3.12]{KL16})}}
The global section functor defines the corresponding equivalence between the categories of the following two sorts of objects. The first ones are the corresponding uniform pseudocoherent sheaves over $\Pi^r_{H,A}$. The second ones are those pseudocoherent modules over the ring $\Pi^r_{H,A}$. 	
\end{conjecture}


\begin{conjecture} \mbox{\bf{(After Kedlaya-Liu \cite[Lemma 5.3.13]{KL16})}}
The global section functor defines the corresponding equivalence between the categories of the following two sorts of objects. The first ones are the corresponding finite projective uniform pseudocoherent sheaves over $\Pi^r_{H,A}$. The second ones are those finite projective modules over the ring $\Pi^r_{H,A}$. 	
\end{conjecture}

\subsection{$\varphi$-Module Structures and $\Gamma$-Module Structures over Noncommutative Rings}

\indent Now we consider the corresponding Frobenius actions over the corresponding imperfect rings we defined before, note that the corresponding Frobenius actions are induced from the corresponding imperfect rings in the undeformed situation from \cite{KL16} which is to say that the Frobenius action on the ring $A$ is actually trivial.

\begin{assumption}
We now drop the corresponding requirement on $A$ which makes the noncommutative deformation of the Robba ring over some interval both left and right noetherian, in \cref{assumption642}. 	
\end{assumption}

\indent First we consider the corresponding Frobenius modules:

\begin{definition} \mbox{\bf{(After Kedlaya-Liu \cite[Definition 5.4.2]{KL16})}}
Over the period rings $\Pi_{H,A}$ or $\breve{\Pi}_{H,A}$  (which is denoted by $\triangle$ in this definition) we define the corresponding $\varphi^a$-modules over $\triangle$ which are respectively projective to be the corresponding finite projective right modules over $\triangle$ with further assigned semilinear action of the operator $\varphi^a$. 
\end{definition}

\begin{definition} \mbox{\bf{(After Kedlaya-Liu \cite[Definition 5.4.2]{KL16})}}
Over each rings $\triangle=\Pi^r_{H,A},\breve{\Pi}^r_{H,A}$ we define the corresponding projective $\varphi^a$-module over any $\triangle$ to be the corresponding finite projective right module $M$ over $\triangle$ with additionally endowed semilinear Frobenius action from $\varphi^a$ such that we have the isomorphism $\varphi^{a*}M\overset{\sim}{\rightarrow}M\otimes \square$ where the ring $\square$ is one $\triangle=\Pi^{r/p}_{H,A},\breve{\Pi}^{r/p}_{H,A}$.   
\end{definition}

\begin{definition} \mbox{\bf{(After Kedlaya-Liu \cite[Definition 5.4.2]{KL16})}}
Again as in \cite[Definition 5.4.2]{KL16}, we define the corresponding projective $\varphi^a$-modules over ring $\Pi^{[s,r]}_{H,A}$ or $\breve{\Pi}^{[s,r]}_{H,A}$  to be the finite projective right modules (which will be denoted by $M$) over $\Pi^{[s,r]}_{H,A}$ or $\breve{\Pi}^{[s,r]}_{H,A}$ respectively additionally endowed with semilinear Frobenius action from $\varphi^a$ with the following isomorphisms:
\begin{align}
\varphi^{a*}M\otimes_{\Pi_{H,A}^{[sp^{-ah},rp^{-ah}]}}\Pi_{H,A}^{[s,rp^{-ah}]}\overset{\sim}{\rightarrow}M\otimes_{\Pi_{H,A}^{[s,r]}}\Pi_{H,A}^{[s,rp^{-ah}]}
\end{align}
and
\begin{align}
\varphi^{a*}M\otimes_{\breve{\Pi}_{R,A}^{[sp^{-ah},rp^{-ah}]}}\breve{\Pi}_{R,A}^{[s,rp^{-ah}]}\overset{\sim}{\rightarrow}M\otimes_{\breve{\Pi}_{R,A}^{[s,r]}}\breve{\Pi}_{R,A}^{[s,rp^{-ah}]}.\\
\end{align} 
\end{definition}

\noindent Also one can further define the corresponding bundles carrying semilinear Frobenius in our context as in the situation of \cite[Definition 5.4.10]{KL16}:

\begin{definition} \mbox{\bf{(After Kedlaya-Liu \cite[Definition 5.4.10]{KL16})}}
Over the ring $\Pi^t_{H,A}$ or $\breve{\Pi}^t_{H,A}$ we define a corresponding projective Frobenius bundle to be a family $(M_I)_I$ of finite projective right modules over each $\Pi^I_{H,A}$ or $\breve{\Pi}^I_{H,A}$ respectively carrying the natural Frobenius action coming from the operator $\varphi^a$ such that for any two involved intervals having the relation $I\subset J$ we have:
\begin{displaymath}
M_J\otimes_{\Pi^J_{H,A}}\Pi^I_{H,A}\overset{\sim}{\rightarrow}	M_I
\end{displaymath}
and 
\begin{displaymath}
M_J\otimes_{\breve{\Pi}^J_{H,A}}\breve{\Pi}^I_{H,A}\overset{\sim}{\rightarrow}	M_I
\end{displaymath}
with the obvious cocycle condition. Here we have to propose condition on the intervals that for each $I=[s,r]$ involved we have $s\leq r/p^{ah}$. As in \cite[Definition 5.4.10]{KL16} one can take the corresponding 2-limit in the direct sense to define the corresponding objects over the full Robba rings.
\end{definition}

\indent We can then compare the corresponding objects defined above:

\begin{conjecture} \mbox{\bf{(After Kedlaya-Liu \cite[Lemma 5.4.11]{KL16})}}\\
I. Consider the following objects for some radius $r_0$ in our situation. The first group of objects are those finite projective $\varphi^a$-modules over the Robba ring $\Pi^{r_0}_{H,A}$. The second group of objects are those finite projective $\varphi^a$-bundles over the Robba ring $\Pi^{r_0}_{H,A}$. The third group of objects are those finite projective $\varphi^a$-modules over the Robba ring $\Pi^{[s,r]}_{H,A}$ for some radii $0<s\leq r\leq r_0$. Then we have that the corresponding categories of the two groups of objects are equivalent.
\end{conjecture}

%

\indent Now we define the corresponding $\Gamma$-modules over the period rings attached to the tower $(H_\bullet,H_\bullet^+)$. The corresponding structures are actually abstractly defined in the same way as in \cite{KL16} and parallel to our commutative setting. First we consider the deformation of the corresponding complex $*_{H^\bullet}$ for any ring $*$ in \cref{2setting6.4}.

\begin{assumption}
Recall that the corresponding tower is called decompleting if it is weakly decompleting, finite \'etale on each finite level and having the exact sequence $\overline{\varphi}^{-1}H'_{H^\bullet}/H'_{H^\bullet}$ is exact. We now assume that the tower $(H_\bullet,H^+_\bullet)$ is then decompleting.	
\end{assumption}


\begin{setting}
Assume now $\Gamma$ is a topological group as in \cite[Definition 5.5.5]{KL16} acting on the corresponding period rings in the \cref{2setting6.3} in the original context of \cref{2setting6.3}. Then we consider the corresponding induced continous action over the corresponding deformed version in our context namely in \cref{2setting6.4}. Assume now that the tower is Galois with the corresponding Galois group $\Gamma$.	
\end{setting}

\begin{definition}
We now consider the corresponding inhomogeneous continuous cocycles of the group $\Gamma$, as in \cite[Definition 5.5.5]{KL16} we use the following notation to denote the corresponding complex extracted from a single tower for a given period ring $*_{H,A}$ in \cref{2setting6.4} for each $k>0$:
\begin{displaymath}
*_{H^k,A}:=C_\mathrm{con}(\Gamma^k,*_{H})\widehat{\otimes}_{\sharp} ?
\end{displaymath}
where $\sharp=\mathbb{Q}_p,\mathbb{Z}_p$ and $?=A,\mathfrak{o}_A$ respectively, which forms the corresponding complex $(*_{H^\bullet,A},d^\bullet)$ with the corresponding differential as in \cite[Definition 5.5.5]{KL16} in the sense of continuous group cohomology.	
\end{definition}

\begin{definition}
Having established the corresponding meaning of the $\Gamma$-structure we now consider the corresponding definition of $\Gamma$-modules. Such modules called the corresponding right $\Gamma$-modules are defined over the corresponding rings in \cref{2setting6.4}. Again we allow that the corresponding right modules to be finite projective, or pseudocoherent or fpd over the rings in \cref{2setting6.4}. And the modules are defined to be carrying the corresponding continuous semilinear action from the group $\Gamma$.	
\end{definition}

\begin{proposition} \mbox{\bf{(After Kedlaya-Liu \cite[Corollary 5.6.5]{KL16})}}
The complex $\varphi^{-1}\Pi^{[sp^{-h},rp^{-h}]}_{H^\bullet_{\geq n},A}/\Pi^{[s,r]}_{H^\bullet_{\geq n},A}$ and the complex $\widetilde{\Pi}^{[s,r]}_{H^\bullet_{\geq n},A}/\Pi^{[s,r]}_{H^\bullet_{\geq n},A}$ are strictly exact for any truncation index $n$. The corresponding radii satisfy the corresponding relation $0<s\leq r\leq r_0$.
\end{proposition}

\begin{proof}
See \cite[Corollary 5.6.5]{KL16}, and consider the corresponding Schauder basis of $A$.	
\end{proof}

\begin{proposition} \mbox{\bf{(After Kedlaya-Liu \cite[Lemma 5.6.6]{KL16})}}
The complex 
\begin{displaymath}
M\otimes_{\Pi^{[s,r]}_{H,A}} \varphi^{-(\ell+1)}\Pi^{[sp^{-h(\ell+1)},rp^{-h(\ell+1)}]}_{H^\bullet,A}/ \varphi^{-\ell}\Pi^{[sp^{-h\ell},rp^{-h\ell}]}_{H^\bullet,A}	
\end{displaymath}
and the complex 
\begin{displaymath}
M\otimes_{\Pi^{[s,r]}_{H,A}} \widetilde{\Pi}^{[s,r]}_{H^\bullet,A}/\varphi^{-\ell}\Pi^{[sp^{-h(\ell)},rp^{-h(\ell)}]}_{H^\bullet,A}
\end{displaymath}
are strictly exact for any truncation index $n$. The corresponding radii satisfy the corresponding relation $0<s\leq r\leq r_0$, and $\ell$ is bigger than some existing truncated integer $\ell_0\geq  0$.
\end{proposition}

\begin{proof}
See \cite[Lemma 5.6.6]{KL16}.	
\end{proof}

\begin{lemma} \mbox{\bf{(After Kedlaya-Liu \cite[Lemma 5.6.8]{KL16})}} Suppose we have a direct system of Banach rings $(B_i,\iota_i)$ where the corresponding map $\iota_i$ as in \cite[Lemma 5.6.8]{KL16} is submetric for each $i$. As in \cite[Lemma 5.6.8]{KL16} we assume the direct limit is endowed with the corresponding infimum seminorm. Now assume that $B\rightarrow C$ is a isometry and of the corresponding image which is assumed to be dense into a Banach ring $C$. Then we have that one can have the chance to have that any finitely generated right projective module over the ring $C$ could come from the corresponding one finitely generated right projective module over some ring $B_i$ after the corresponding base change.

\end{lemma}

\begin{proof}
We need to work in the corresponding noncommutative setting. We adapt the argument in the proof of \cite[Lemma 5.6.8]{KL16} with some possible modification due to the fact that the rings are noncommutative. Let $O$ be the matrix attached to the corresponding projector associated to the module over the ring $C$. Now choose $P$ such that $\|O-P\|< \|O\|^{-3}$ (and we need to guarantee that this $P$ could live over some $B_i$). Now we consider the corresponding iteration where $Q_0:=P$ and set $Q_{k+1}=3Q_k^2-2Q_k^3$, so we have:
\begin{align}
Q_{k+1}-Q_k&=(1-2Q_k)(Q_k^2-Q_k),\\
Q^2_{k+1}-Q_{k+1}&=(4Q_k^2-4Q_k-3)(Q_k^2-Q_k)^2,	
\end{align}
which implies that by induction (see the estimate on $P^2-P$ below):
\begin{align}
\|Q_{k+1}-O\|&\leq \|O\|\|P^2-P\|,\\
\|Q^2_{k}-Q_{k}\|&\leq (\|P^2-P\|\|O\|^2)^{2^k}\|O\|^{-2}.	
\end{align}
%
We estimate $P^2-P$ in the following way:
\begin{align}
P^2-P &=P^2-P-O^2+O	\\
      &=(P-O)(P+O-1)+OP-PO\\
      &=(P-O)(P+O-1)+OP-O^2+O^2-PO\\
      &=(P-O)(P+O-1)+O(P-O)+(O-P)O
\end{align}
which implies that actually:
\begin{align}
\|P^2-P\| &\leq \sup\{\|(P-O)(P+O-1)\|,\|O(P-O)\|,\|(O-P)O\|\}\\
	      &\leq \sup\{\|(P-O)(P+O-1)\|,\|O\|\|(P-O)\|,\|(O-P)\|\|O\|\}\\
	      &< \|O\|^{-3}\|O\|\\
	      &< \|O\|^{-2}
\end{align}
which by taking the limit gives rise to the desired matrix $Q$ which proves the result. Indeed one can consider the corresponding projection induced by $Q$ from some free module over $B_i$, which gives rise to the desired projective module over $B_i$ whose base change to $C$ gives the initially chosen finite projective module 
over $C$ since one can compute:
\begin{align}
OQ+(1-O)(1-Q)-1&=OQ-O^2-Q^2+OQ\\
	           &=O(Q-O)-(Q-O)Q
\end{align}
which implies that $\|OQ+(1-O)(1-Q)-1\|\leq 1$.

\end{proof}

\begin{remark}
The bound for $Q_k-O$ is not the same as in \cite[Lemma 5.6.8]{KL16}, we would like to thank Professor Kedlaya for letting us know this correction.	
\end{remark}

\begin{proposition} \mbox{\bf{(After Kedlaya-Liu \cite[Lemma 5.6.9]{KL16})}}
With the corresponding notations as above we have that the corresponding base change from the ring taking the form of $\breve{\Pi}^{[s,r]}_{H,A}$ to the ring taking the form of $\widetilde{\Pi}^{[s,r]}_{H,A}$ establishes the corresponding equivalence on the corresponding categories of $\Gamma$-modules.	
\end{proposition}

\begin{proof}
This is a relative version of the corresponding result in \cite[Lemma 5.6.9]{KL16} we adapt the corresponding argument here. Indeed the corresponding fully faithfulness comes from the previous proposition the idea to prove the corresponding essential surjectivity comes from writing the module $M$ over the ring taking form of $\widetilde{\Pi}^{[s,r]}_{H,A}$ as the base change from $\varphi^{-k}(\Pi^{[s,r]}_{H,A})$ of a module $M_0$ after the corresponding analog of \cite[Lemma 5.6.8]{KL16} as above.	Then as in \cite[Lemma 5.6.9]{KL16} we consider the corresponding norms on the corresponding $M$ (the corresponding differentials) and the base change of $M_0$ (the corresponding differentials) which could be controlled up to some constant which could be further modified to be zero by reducing each time positive amount of constant from the constant represented by the difference of the norms.
\end{proof}

\begin{definition} \mbox{\bf{(After Kedlaya-Liu \cite[Definition 5.7.2]{KL16})}}
Over the period rings $\Pi_{H,A}$ or $\breve{\Pi}_{H,A}$  (which is denoted by $\triangle$ in this definition) we define the corresponding $(\varphi^a,\Gamma)$-modules over $\triangle$ which are respectively projective to be the corresponding finite projective right $\Gamma$-modules over $\triangle$ with further assigned semilinear action of the operator $\varphi^a$ with the isomorphism defined by using the Frobenius.  
\end{definition}

\begin{definition} \mbox{\bf{(After Kedlaya-Liu \cite[Definition 5.7.2]{KL16})}}
Over each rings $\triangle=\Pi^r_{H,A},\breve{\Pi}^r_{H,A}$ we define the corresponding projective $(\varphi^a,\Gamma)$-module over any $\triangle$ to be the corresponding finite projective right $\Gamma$-module $M$ over $\triangle$ with additionally endowed semilinear Frobenius action from $\varphi^a$ such that we have the isomorphism $\varphi^{a*}M\overset{\sim}{\rightarrow}M\otimes \square$ where the ring $\square$ is one $\triangle=\Pi^{r/p}_{H,A},\breve{\Pi}^{r/p}_{H,A}$.  
\end{definition}

\begin{definition} \mbox{\bf{(After Kedlaya-Liu \cite[Definition 5.7.2]{KL16})}}
Again as in \cite[Definition 5.7.2]{KL16}, we define the corresponding projective $(\varphi^a,\Gamma)$-modules over ring $\Pi^{[s,r]}_{H,A}$ or $\breve{\Pi}^{[s,r]}_{H,A}$ to be the finite projective right $\Gamma$-modules (which will be denoted by $M$) over $\Pi^{[s,r]}_{H,A}$ or $\breve{\Pi}^{[s,r]}_{H,A}$ respectively additionally endowed with semilinear Frobenius action from $\varphi^a$ with the following isomorphisms:
\begin{align}
\varphi^{a*}M\otimes_{\Pi_{H,A}^{[sp^{-ah},rp^{-ah}]}}\Pi_{H,A}^{[s,rp^{-ah}]}\overset{\sim}{\rightarrow}M\otimes_{\Pi_{H,A}^{[s,r]}}\Pi_{H,A}^{[s,rp^{-ah}]}
\end{align}
and
\begin{align}
\varphi^{a*}M\otimes_{\breve{\Pi}_{R,A}^{[sp^{-ah},rp^{-ah}]}}\breve{\Pi}_{R,A}^{[s,rp^{-ah}]}\overset{\sim}{\rightarrow}M\otimes_{\breve{\Pi}_{R,A}^{[s,r]}}\breve{\Pi}_{R,A}^{[s,rp^{-ah}]}.
\end{align}
\end{definition}

\begin{definition} \mbox{\bf{(After Kedlaya-Liu \cite[Definition 5.7.2]{KL16})}}
Over the ring $\Pi^t_{H,A}$ or $\breve{\Pi}^t_{H,A}$ we define a corresponding projective $(\varphi^a,\Gamma)$ bundle to be a family $(M_I)_I$ of finite projective right $\Gamma$-modules over each $\widetilde{\Pi}^I_{H,A}$ carrying the natural Frobenius action coming from the operator $\varphi^a$ such that for any two involved intervals having the relation $I\subset J$ we have:
\begin{displaymath}
M_J\otimes_{\Pi^J_{H,A}}\Pi^I_{H,A}\overset{\sim}{\rightarrow}	M_I
\end{displaymath}
and 
\begin{displaymath}
M_J\otimes_{\breve{\Pi}^J_{H,A}}\breve{\Pi}^I_{H,A}\overset{\sim}{\rightarrow}	M_I
\end{displaymath}
with the obvious cocycle condition. Here we have to propose condition on the intervals that for each $I=[s,r]$ involved we have $s\leq r/p^{ah}$. We require the corresponding topological conditions as we did for the corresponding Frobenius bundles. one can take the corresponding 2-limit in the direct sense to define the corresponding objects over the full Robba rings.
\end{definition}

\begin{definition} \mbox{\bf{(After Kedlaya-Liu \cite[Definition 5.7.2]{KL16})}}
Over the period rings $\widetilde{\Pi}_{H,A}$ (which is denoted by $\triangle$ in this definition) we define the corresponding $(\varphi^a,\Gamma)$-modules over $\triangle$ which are respectively projective to be the corresponding finite projective right $\Gamma$-modules over $\triangle$ with further assigned semilinear action of the operator $\varphi^a$ with the isomorphism defined by using the Frobenius. 
\end{definition}

\begin{definition} \mbox{\bf{(After Kedlaya-Liu \cite[Definition 5.7.2]{KL16})}}
Over each ring $\triangle=\widetilde{\Pi}^r_{H,A}$ we define the corresponding projective $(\varphi^a,\Gamma)$-module over any $\triangle$ to be the corresponding finite projective right $\Gamma$-module $M$ over $\triangle$ with additionally endowed semilinear Frobenius action from $\varphi^a$ such that we have the isomorphism $\varphi^{a*}M\overset{\sim}{\rightarrow}M\otimes \square$ where the ring $\square$ is one $\triangle=\widetilde{\Pi}^{r/p}_{H,A}$.

\end{definition}

\begin{definition} \mbox{\bf{(After Kedlaya-Liu \cite[Definition 5.7.2]{KL16})}}
Again as in \cite[Definition 5.7.2]{KL16}, we define the corresponding projective $(\varphi^a,\Gamma)$-modules over ring $\widetilde{\Pi}^{[s,r]}_{H,A}$ to be the finite projective right $\Gamma$-modules (which will be denoted by $M$) over $\widetilde{\Pi}^{[s,r]}_{H,A}$ additionally endowed with semilinear Frobenius action from $\varphi^a$ with the following isomorphisms:
\begin{align}
\varphi^{a*}M\otimes_{\widetilde{\Pi}_{H,A}^{[sp^{-ah},rp^{-ah}]}}\widetilde{\Pi}_{H,A}^{[s,rp^{-ah}]}\overset{\sim}{\rightarrow}M\otimes_{\widetilde{\Pi}_{H,A}^{[s,r]}}\widetilde{\Pi}_{H,A}^{[s,rp^{-ah}]}.
\end{align}
\end{definition}

\begin{definition} \mbox{\bf{(After Kedlaya-Liu \cite[Definition 5.7.2]{KL16})}}
Over the ring $\widetilde{\Pi}^t_{H,A}$ we define a corresponding projective $(\varphi^a,\Gamma)$ bundle to be a family $(M_I)_I$ of finite projective right $\Gamma$-modules over each $\widetilde{\Pi}^I_{H,A}$ carrying the natural Frobenius action coming from the operator $\varphi^a$ such that for any two involved intervals having the relation $I\subset J$ we have:
\begin{displaymath}
M_J\otimes_{\widetilde{\Pi}^J_{H,A}}\widetilde{\Pi}^I_{H,A}\overset{\sim}{\rightarrow}	M_I
\end{displaymath}
with the obvious cocycle condition. Here we have to propose condition on the intervals that for each $I=[s,r]$ involved we have $s\leq r/p^{ah}$. We require the corresponding topological conditions as we did for the corresponding Frobenius bundles. one can take the corresponding 2-limit in the direct sense to define the corresponding objects over the full Robba rings.
\end{definition}


\begin{conjecture} \mbox{\bf{(After Kedlaya-Liu \cite[Theorem 5.7.5]{KL16})}} \label{conjecture672}
We have now the following categories are equivalence for the corresponding radii $0< s\leq r\leq r_0$ (with the further requirement as in \cite[Theorem 5.7.5]{KL16} that $s\in (0,r/q]$):\\
1. The category of all the finite projective sheaves over the ring $\widetilde{\Pi}_{\mathrm{Spa}(H_0,H_0^+),A}$, carrying the $\varphi^a$-action;\\
2. The category of all the finite projective sheaves over the ring $\widetilde{\Pi}^r_{\mathrm{Spa}(H_0,H_0^+),A}$, carrying the $\varphi^a$-action;\\
3. The category of all the finite projective sheaves over the ring $\widetilde{\Pi}^{[s,r]}_{\mathrm{Spa}(H_0,H_0^+),A}$, carrying the $\varphi^a$-action;\\
4. The category of all the finite projective modules over the ring $\Pi_{H,A}$, carrying the $(\varphi^a,\Gamma)$-action;\\
5. The category of all the finite projective bundles over the ring $\Pi_{H,A}$, carrying the $(\varphi^a,\Gamma)$-action;\\
6. The category of all the finite projective modules over the ring $\breve{\Pi}_{H,A}$, carrying the $(\varphi^a,\Gamma)$-action;\\
7. The category of all the finite projective bundles over the ring $\breve{\Pi}_{H,A}$, carrying the $(\varphi^a,\Gamma)$-action;\\
8. The category of all the finite projective modules over the ring $\breve{\Pi}^{[s,r]}_{H,A}$, carrying the $(\varphi^a,\Gamma)$-action;\\
9. The category of all the finite projective modules over the ring $\widetilde{\Pi}_{H,A}$, carrying the $(\varphi^a,\Gamma)$-action;\\
10. The category of all the finite projective bundles over the ring $\widetilde{\Pi}_{H,A}$, carrying the $(\varphi^a,\Gamma)$-action;\\
11. The category of all the finite projective modules over the ring $\widetilde{\Pi}^{[s,r]}_{H,A}$, carrying the $(\varphi^a,\Gamma)$-action.
\end{conjecture}


\begin{proposition} \mbox{\bf{(After Kedlaya-Liu \cite[Theorem 5.7.5]{KL16})}}
We have now the corresponding equivalence among categories described as below:\\
1. The category of all the finite projective modules over the ring $\Pi_{H,A}$, carrying the $(\varphi^a,\Gamma)$-action;\\
2. The category of all the finite projective modules over the ring $\breve{\Pi}_{H,A}$, carrying the $(\varphi^a,\Gamma)$-action;\\
3. The category of all the finite projective modules over the ring $\widetilde{\Pi}_{H,A}$, carrying the $(\varphi^a,\Gamma)$-action.

\end{proposition}

\begin{proof}
These are proved exactly the same as \cite[Theorem 5.7.5]{KL16} by using our development.
\end{proof}

\subsection{Noncommutative Descent}

\indent Motivated by the corresponding comparison between the bundles and modules carrying the corresponding action coming from the Frobenius and the group $\Gamma$ we study some noncommutative descent, which takes the form which is more representation theoretic and topos theoretic. Note that the corresponding localization and descent are very complicated issues within the corresponding noncommutative geometry.

\indent Here we study the corresponding noncommutative descent in the style coming from \cite[Section 1.3]{KL15}. The corresponding story in our setting will be a noncommutative analog of \cite[Section 1.3]{KL15}. Also we mention that \cite{DLLZ1} has already considered some interesting generalization of \cite[Section 1.3]{KL15} along the other direction within the study of the corresponding logarithmic spaces.

\begin{setting}\mbox{\bf{(Noncommutaive Glueing)}}
Consider the following noncommutative analog of \cite[Section 1.3]{KL15}. First we have a square diagram taking the following form:
\[
\xymatrix@R+2pc@C+2pc{
\Pi \ar[r]\ar[r]\ar[r]\ar[d]\ar[d]\ar[d] &\Pi_2 \ar[d]\ar[d]\ar[d]\\
\Pi_1 \ar[r]\ar[r]\ar[r] &\Pi_{12},\\
}
\]
which expands to the corresponding short exact sequence:
\[
\xymatrix@R+1pc@C+1pc{
0 \ar[r]\ar[r]\ar[r] &\Pi \ar[r]\ar[r]\ar[r] &\Pi_1\oplus \Pi_2 \ar[r]^{*-*}\ar[r]\ar[r] &\Pi_{12}  \ar[r]\ar[r]\ar[r] &0
}
\]	
in the corresponding category of $\Pi$-modules. The corresponding datum consists also of three right modules $M_1,M_2,M_{12}$ over the corresponding rings $\Pi_1,\Pi_2,\Pi_{12}$ respectively, with the corresponding base change isomorphisms from the module $M_1,M_2$ to the module $M_{12}$. Take the corresponding kernel we have the corresponding exact sequence:
\[
\xymatrix@R+1pc@C+1pc{
0 \ar[r]\ar[r]\ar[r] &M \ar[r]\ar[r]\ar[r] &M_1\oplus M_2 \ar[r]\ar[r]\ar[r] &M_{12}.  
}
\]
We are going to call the corresponding datum in our situation to be $coherent$, $pseudocoherent$, $ finite$, $finitely~presented$, $finite~projective$ if the corresponding modules involved are so over the corresponding rings $\Pi_1,\Pi_2,\Pi_{12}$, which is to say $coherent$, $pseudocoherent$, $ finite$, $finitely~presented$, $finite~projective$.
\end{setting}


\begin{lemma}\mbox{\bf{(After Kedlaya-Liu \cite[Lemma 1.3.8]{KL15})}}\label{lemma677} 
For finite datum defined in the corresponding sense, suppose we further assume that the corresponding map $M\otimes\Pi_1\rightarrow M_1$ is surjective. Then we have that the corresponding maps $M\otimes\Pi_2\rightarrow M_2$ and $f_1-f_2:M_1\oplus M_2\rightarrow M_{12}$ are also being surjective, and one can find a finite submodule of $M$ such that the corresponding base change to the $R_1$ or $R_2$ maps through surjective maps to the corresponding module $M_1$ or $M_2$ respectively.
\end{lemma}

\begin{proof}
The proof in our noncommutative setting is actually parallel to \cite[Lemma 1.3.8]{KL15}. To be more precise the map $f_1-f_2:M_1\oplus M_2\rightarrow M_{12}$ is then surjective is just because we can derive this from the corresponding surjectivity of $M\otimes\Pi_1\rightarrow M_1$. The surjectivity of $M\otimes\Pi_1\rightarrow M_1$ implies that the corresponding surjectivity of $M\otimes (\Pi_1\oplus \Pi_2)\rightarrow M_{12}$ which factor through the corresponding desired map in our situation. Then after proving this one can show the corresponding surjectivity of the map $M\otimes\Pi_2\rightarrow M_2$. Indeed suppose we start from some element $m\in M_2$ then by taking the corresponding base change to $\Pi_{12}$ we can have corresponding elements $m_1$ and $m_2$ living in the image of the maps:
\begin{align}
g_1:M\otimes \Pi_1\rightarrow M_1,\\
g_2:M\otimes \Pi_2\rightarrow M_2,	
\end{align}
which gives the corresponding element $f_1(m_1)-f_2(m_2)$ which is just the corresponding image of the element $m$ in $\Pi_{12}$. Then look at the element $(m_2,m_2+m)$ in the direct sum, then one can see that the image of the map in our mind will contain actually the element $m_2$ and $m+m_2$ at the same times, which implies the result. Finally the last statement then follows from this.
\end{proof}

\begin{lemma} \mbox{\bf{(After Kedlaya-Liu \cite[Lemma 1.3.9]{KL15})}} \label{lemma678}
For finitely projective datum, suppose that we have the surjectivity of the map $M\otimes \Pi_1\rightarrow M_1$. Then we have the kernel $M$ is finitely presented and we have the isomorphisms $M\otimes \Pi_1\rightarrow M_1$ and $M\otimes \Pi_2 \rightarrow M_2$.
\end{lemma}

\begin{proof}
By applying the previous lemma one can have the chance to find a submodule $M'$ of $M$ which is finite which admits a covering from some finite free module $G$. Then set $G_1,G_2,G_{12}$ to be the corresponding base change of $G$ to the corresponding rings $\Pi_1,\Pi_2,\Pi_{12}$. Then we have the following commutative diagram:
\[
\xymatrix@R+2pc@C+2pc{
     &0\ar[d]\ar[d]\ar[d] &0\ar[d]\ar[d]\ar[d] &0\ar[d]\ar[d]\ar[d] \\ 
0 \ar[r]\ar[r]\ar[r] &F \ar[r]\ar[r]\ar[r]\ar[d]\ar[d]\ar[d] &F_1\oplus F_2 \ar[d]\ar[d]\ar[d] \ar[r]\ar[r]\ar[r] &F_{12} \ar[d]\ar[d]\ar[d]\\
0 \ar[r]\ar[r]\ar[r] &G \ar[r]\ar[r]\ar[r]\ar[d]\ar[d]\ar[d] &G_1\oplus G_2 \ar[d]\ar[d]\ar[d] \ar[r]\ar[r]\ar[r] &G_{12} \ar[d]\ar[d]\ar[d] \ar[r]\ar[r]\ar[r] \ar[d]\ar[d]\ar[d] &0\\
0 \ar[r]\ar[r]\ar[r] &M \ar[r]\ar[r]\ar[r]  &M_1\oplus M_2  \ar[r]\ar[r]\ar[r] \ar[d]\ar[d]\ar[d] &M_{12} \ar[d]\ar[d]\ar[d] \ar[r]\ar[r]\ar[r] \ar[d]\ar[d]\ar[d] &0  \\
&  &0  &0.
}
\]
Here the module $F,F_1,F_2,F_{12}$ are the corresponding kernel of the corresponding maps in the diagram, namely $G\rightarrow M$, $G_1\rightarrow M_1$, $G_2\rightarrow M_2$, $G_{12}\rightarrow M_{12}$ respectively. Now for each $j=\{1,2\}$ we have the corresponding exact sequence:
\[
\xymatrix@R+1pc@C+1pc{
0 \ar[r]\ar[r]\ar[r] &F_j\otimes_{\Pi_j}\Pi_{12} \ar[r]\ar[r]\ar[r] &G_{12} \ar[r]\ar[r]\ar[r] &M_{12}  \ar[r]\ar[r]\ar[r] &0
}
\]
since the corresponding module $M_j$ ($j\in\{1,2\}$) is finite projective. We then have the situation that $F_j\otimes_{\Pi_j}\Pi_{12}\overset{\sim}{\rightarrow} F_{12}$ for each $j\in \{1,2\}$ and the corresponding modules $F_1,F_2$ are finitely generated and projective. Now apply the previous lemma we have that $F_1\oplus F_2\rightarrow F_{12}$ is then surjective which gives the following commutative diagram:
\[
\xymatrix@R+2pc@C+2pc{
     &0\ar[d]\ar[d]\ar[d] &0\ar[d]\ar[d]\ar[d] &0\ar[d]\ar[d]\ar[d] \\ 
0 \ar[r]\ar[r]\ar[r] &F \ar[r]\ar[r]\ar[r]\ar[d]\ar[d]\ar[d] &F_1\oplus F_2 \ar[d]\ar[d]\ar[d] \ar[r]^{\gamma}\ar[r]\ar[r] &F_{12} \ar[d]\ar[d]\ar[d] \ar[r]\ar[r]\ar[r] \ar[d]\ar[d]\ar[d] &0\\
0 \ar[r]\ar[r]\ar[r] &G \ar[r]\ar[r]\ar[r]\ar[d]\ar[d]\ar[d] &G_1\oplus G_2 \ar[d]\ar[d]\ar[d] \ar[r]^{\beta}\ar[r]\ar[r] &G_{12} \ar[d]\ar[d]\ar[d] \ar[r]\ar[r]\ar[r] \ar[d]\ar[d]\ar[d] &0\\
0 \ar[r]\ar[r]\ar[r] &M \ar[r]\ar[r]\ar[r]  &M_1\oplus M_2  \ar[r]^{\alpha}\ar[r]\ar[r] \ar[d]\ar[d]\ar[d] &M_{12} \ar[d]\ar[d]\ar[d] \ar[r]\ar[r]\ar[r] \ar[d]\ar[d]\ar[d] &0  \\
&  &0  &0.
}
\]
Then we claim that then the map $G\rightarrow M$ is surjective. Indeed this is by direct diagram chasing. First take any element in $M_1\oplus M_2$, which is denoted by $(m_1,m_2)$ for which we assume that $(m_1,m_2)\in \mathrm{Ker}\alpha$, then take any element in $(f_1,f_2)\in G_1\oplus G_2$ lifting this element. Let $\overline{f_1-f_2}$ be the image of $(f_1,f_2)$ under $\beta$. By the commutativity of the diagram we have $\overline{f_1-f_2}$ dies in $M_{12}$, so we can find $g_{12}\in F_{12}$ whose image in $G_{12}$ is $\overline{f_1-f_2}$. Then take any element element $(g_1,g_2)$ in $F_1\oplus F_2$ which lifts $f_{12}$. Then since the diagram is commutative we then have that $(g_1,g_2)$ takes image in $G_1\oplus G_2$ which must be living in same equivalence class with $(f_1,f_2)$ with respect to $G$, so we have $(f_1,f_2)=(g_1,g_2)\oplus f$ where $f\in G$. But then we have $f$ which is an element in the kernel of $\beta$ maps to $(m_1,m_2)$ which finishes the proof of the claim. Then repeat this we can have the chance to derive the finiteness of the module $F$. Then to finish we look at the corresponding commutative diagram:
\[
\xymatrix@R+2pc@C+2pc{
 &\mathrm{Ker}(G\rightarrow M)\otimes \Pi_i \ar[r]\ar[r]\ar[r]\ar[d]\ar[d]\ar[d] &G_i \ar[d]\ar[d]\ar[d] \ar[r]\ar[r]\ar[r] &M\otimes \Pi_i \ar[d]\ar[d]\ar[d] \ar[r]\ar[r]\ar[r] \ar[d]\ar[d]\ar[d] &0\\
0 \ar[r]\ar[r]\ar[r] &\mathrm{Ker}(G_i\rightarrow M_i) \ar[r]\ar[r]\ar[r]&G_i  \ar[r]\ar[r]\ar[r] &M_i  \ar[r]\ar[r]\ar[r]  &0,
}
\]
which implies that the corresponding map from $M\otimes \Pi_i$ to $M_i$ is injective as well by five lemma (note that the corresponding map $\mathrm{Ker}(G\rightarrow M)\otimes \Pi_i$ to $\mathrm{Ker}(G_i\rightarrow M_i)$ is surjective). So we have the corresponding desired isomorphisms.	
\end{proof}

\indent The following is a similar noncommutative version of \cite[Lemma A.3]{DLLZ1}:

\begin{lemma}\mbox{\bf{(After Kedlaya-Liu \cite[Lemma 1.3.9]{KL15})}} \label{lemma679}
For finitely presented datum, suppose that we have the surjectivity of the map $M\otimes \Pi_1\rightarrow M_1$. And suppose we have the corresponding additional requirement that $\Pi_1\rightarrow \Pi_{12}$ and $\Pi_2\rightarrow \Pi_{12}$ are now assumed to be flat. Then we have the kernel $M$ is finitely presented and we have the isomorphisms $M\otimes \Pi_1\rightarrow M_1$ and $M\otimes \Pi_2 \rightarrow M_2$.
\end{lemma}

\begin{proof}
By applying the previous lemma one can have the chance to find a submodule $M'$ of $M$ which is finite which admits a covering from some finite free module $G$. Then set $G_1,G_2,G_{12}$ to be the corresponding base change of $G$ to the corresponding rings $\Pi_1,\Pi_2,\Pi_{12}$. Then we have the following commutative diagram:
\[
\xymatrix@R+2pc@C+2pc{
     &0\ar[d]\ar[d]\ar[d] &0\ar[d]\ar[d]\ar[d] &0\ar[d]\ar[d]\ar[d] \\ 
0 \ar[r]\ar[r]\ar[r] &F \ar[r]\ar[r]\ar[r]\ar[d]\ar[d]\ar[d] &F_1\oplus F_2 \ar[d]\ar[d]\ar[d] \ar[r]\ar[r]\ar[r] &F_{12} \ar[d]\ar[d]\ar[d]\\
0 \ar[r]\ar[r]\ar[r] &G \ar[r]\ar[r]\ar[r]\ar[d]\ar[d]\ar[d] &G_1\oplus G_2 \ar[d]\ar[d]\ar[d] \ar[r]\ar[r]\ar[r] &G_{12} \ar[d]\ar[d]\ar[d] \ar[r]\ar[r]\ar[r] \ar[d]\ar[d]\ar[d] &0\\
0 \ar[r]\ar[r]\ar[r] &M \ar[r]\ar[r]\ar[r]  &M_1\oplus M_2  \ar[r]\ar[r]\ar[r] \ar[d]\ar[d]\ar[d] &M_{12} \ar[d]\ar[d]\ar[d] \ar[r]\ar[r]\ar[r] \ar[d]\ar[d]\ar[d] &0  \\
&  &0  &0.
}
\]
Here the module $F,F_1,F_2,F_{12}$ are the corresponding kernel of the corresponding maps in the diagram, namely $G\rightarrow M$, $G_1\rightarrow M_1$, $G_2\rightarrow M_2$, $G_{12}\rightarrow M_{12}$ respectively. Now for each $j=\{1,2\}$ we have the corresponding exact sequence:
\[
\xymatrix@R+1pc@C+1pc{
0 \ar[r]\ar[r]\ar[r] &F_j\otimes_{\Pi_j}\Pi_{12} \ar[r]\ar[r]\ar[r] &G_{12} \ar[r]\ar[r]\ar[r] &M_{12}  \ar[r]\ar[r]\ar[r] &0
}
\]
since we have the corresponding additional requirement that $\Pi_1\rightarrow \Pi_{12}$ and $\Pi_2\rightarrow \Pi_{12}$ are now assumed to be flat. We then have the situation that $F_j\otimes_{\Pi_j}\Pi_{12}\overset{\sim}{\rightarrow} F_{12}$ for each $j\in \{1,2\}$. The corresponding modules $F_1,F_2$ are finitely generated by assumption. Now applying the previous lemma we have that $F_1\oplus F_2\rightarrow F_{12}$ is then surjective which gives the following commutative diagram:
\[
\xymatrix@R+2pc@C+2pc{
     &0\ar[d]\ar[d]\ar[d] &0\ar[d]\ar[d]\ar[d] &0\ar[d]\ar[d]\ar[d] \\ 
0 \ar[r]\ar[r]\ar[r] &F \ar[r]\ar[r]\ar[r]\ar[d]\ar[d]\ar[d] &F_1\oplus F_2 \ar[d]\ar[d]\ar[d] \ar[r]^{\gamma}\ar[r]\ar[r] &F_{12} \ar[d]\ar[d]\ar[d] \ar[r]\ar[r]\ar[r] \ar[d]\ar[d]\ar[d] &0\\
0 \ar[r]\ar[r]\ar[r] &G \ar[r]\ar[r]\ar[r]\ar[d]\ar[d]\ar[d] &G_1\oplus G_2 \ar[d]\ar[d]\ar[d] \ar[r]^{\beta}\ar[r]\ar[r] &G_{12} \ar[d]\ar[d]\ar[d] \ar[r]\ar[r]\ar[r] \ar[d]\ar[d]\ar[d] &0\\
0 \ar[r]\ar[r]\ar[r] &M \ar[r]\ar[r]\ar[r]  &M_1\oplus M_2  \ar[r]^{\alpha}\ar[r]\ar[r] \ar[d]\ar[d]\ar[d] &M_{12} \ar[d]\ar[d]\ar[d] \ar[r]\ar[r]\ar[r] \ar[d]\ar[d]\ar[d] &0  \\
&  &0  &0.
}
\]
Then we claim that then the map $G\rightarrow M$ is surjective. Indeed this is by direct diagram chasing. First take any element in $M_1\oplus M_2$, which is denoted by $(m_1,m_2)$ for which we assume that $(m_1,m_2)\in \mathrm{Ker}\alpha$, then take any element in $(f_1,f_2)\in G_1\oplus G_2$ lifting this element. Let $\overline{f_1-f_2}$ be the image of $(f_1,f_2)$ under $\beta$. By the commutativity of the diagram we have $\overline{f_1-f_2}$ dies in $M_{12}$, so we can find $g_{12}\in F_{12}$ whose image in $G_{12}$ is $\overline{f_1-f_2}$. Then take any element element $(g_1,g_2)$ in $F_1\oplus F_2$ which lifts $f_{12}$. Then since the diagram is commutative we then have that $(g_1,g_2)$ takes image in $G_1\oplus G_2$ which must be living in same equivalence class with $(f_1,f_2)$ with respect to $G$, so we have $(f_1,f_2)=(g_1,g_2)\oplus f$ where $f\in G$. But then we have $f$ which is an element in the kernel of $\beta$ maps to $(m_1,m_2)$ which finishes the proof of the claim. Then repeat this we can have the chance to derive the finiteness of the module $F$. Then to finish we look at the corresponding commutative diagram:
\[
\xymatrix@R+2pc@C+2pc{
 &\mathrm{Ker}(G\rightarrow M)\otimes \Pi_i \ar[r]\ar[r]\ar[r]\ar[d]\ar[d]\ar[d] &G_i \ar[d]\ar[d]\ar[d] \ar[r]\ar[r]\ar[r] &M\otimes \Pi_i \ar[d]\ar[d]\ar[d] \ar[r]\ar[r]\ar[r] \ar[d]\ar[d]\ar[d] &0\\
0 \ar[r]\ar[r]\ar[r] &\mathrm{Ker}(G_i\rightarrow M_i) \ar[r]\ar[r]\ar[r]&G_i  \ar[r]\ar[r]\ar[r] &M_i  \ar[r]\ar[r]\ar[r]  &0,
}
\]
which implies that the corresponding map from $M\otimes \Pi_i$ to $M_i$ is injective as well by five lemma (note that the corresponding map $\mathrm{Ker}(G\rightarrow M)\otimes \Pi_i$ to $\mathrm{Ker}(G_i\rightarrow M_i)$ is surjective). So we have the corresponding desired isomorphisms.	
	
\end{proof}

\indent Now we consider the corresponding Banach structures as well. In this case we consider the following glueing process. First we have the following analog of \cite[Lemma 2.7.1]{KL15}:

\begin{lemma}\mbox{\bf{(After Kedlaya-Liu \cite[Lemma 2.7.1]{KL15})}}
Suppose we have the following two homomorphism of Banach rings (which are not necessarily commutative in our situation):
\begin{align}
f_1: \Pi_1\rightarrow \Gamma\\
f_2: \Pi_2\rightarrow \Gamma	
\end{align}
such that they are bounded and we have the following sum map induced which is surjective in the strict sense:
\begin{displaymath}
\Pi_1\oplus \Pi_2\rightarrow \Gamma.	
\end{displaymath}
Then one can find some constant $c$ with the fact that for any positive integer $n\geq 1$, we have the corresponding decomposition of any invertible matrix $I\in \Gamma^{n\times n}$ such that $\|I-1\|\leq c$ in the following form:
\begin{align}
f_1(I_1)f_2(I_2)	
\end{align}
with invertible matrices $I_1,I_2$ living in $\Pi_1^{n\times n}$ and $\Pi_2^{n\times n}$	respectively.
\end{lemma}

\begin{proof}
Following 	\cite[Lemma 2.7.1]{KL15}, we consider the corresponding lifting process. First by the conditions of the lemma we have that one can lift $I-1$ to a pair matrices $A,B$ living in $\mathrm{GL}(\Pi_1)$ and $\mathrm{GL}(\Pi_2)$ with the corresponding estimate on the entries:
\begin{displaymath}
\|A_{ij}\|\leq c_0 \|(I-1)_{ij}\|\\
\|B_{ij}\|\leq c_0 \|(I-1)_{ij}\|\\	
\end{displaymath}
for some constant $c_0$. Then consider the corresponding iteration process, namely from what we have set $I'$ to be $f_1(1-A)If_2(1-B)$ such that we have the corresponding estimate in the form of $\|I'-1\|\leq d\|I-1\|^2$. Then to finish run the construction process above step by step, we will be done.
\end{proof}

\indent Now in the following setting, we set all rings involved to be Banach:

\begin{setting}\mbox{\bf{(Noncommutaive Glueing)}}
Consider the following noncommutative analog of \cite[Section 1.3]{KL15}. First we have a square diagram taking the following form which is commutative:
\[
\xymatrix@R+2pc@C+2pc{
\Pi \ar[r]\ar[r]\ar[r]\ar[d]\ar[d]\ar[d] &\Pi_2 \ar[d]\ar[d]\ar[d]\\
\Pi_1 \ar[r]\ar[r]\ar[r] &\Pi_{12},\\
}
\]
which expands to the corresponding short exact sequence:
\[
\xymatrix@R+1pc@C+1pc{
0 \ar[r]\ar[r]\ar[r] &\Pi \ar[r]\ar[r]\ar[r] &\Pi_1\oplus \Pi_2 \ar[r]^{*-*}\ar[r]\ar[r] &\Pi_{12}  \ar[r]\ar[r]\ar[r] &0
}
\]	
in the corresponding category of $\Pi$-modules. The corresponding datum consists also of three right modules $M_1,M_2,M_{12}$ over the corresponding rings $\Pi_1,\Pi_2,\Pi_{12}$ respectively, with the corresponding base change isomorphisms from the module $M_1,M_2$ to the module $M_{12}$. Take the corresponding kernel we have the corresponding exact sequence:
\[
\xymatrix@R+1pc@C+1pc{
0 \ar[r]\ar[r]\ar[r] &M \ar[r]\ar[r]\ar[r] &M_1\oplus M_2 \ar[r]\ar[r]\ar[r] &M_{12}.  
}
\]
We are going to call the corresponding datum in our situation to be $coherent$, $pseudocoherent$, $ finite$, $finitely~presented$, $finite~projective$ if the corresponding modules involved are so over the corresponding rings $\Pi_1,\Pi_2,\Pi_{12}$, which is to say $coherent$, $pseudocoherent$, $finite$, $finitely~presented$, $finite~projective$, and we require that the corresponding exact sequence:
\[
\xymatrix@R+1pc@C+1pc{
0 \ar[r]\ar[r]\ar[r] &\Pi \ar[r]\ar[r]\ar[r] &\Pi_1\oplus \Pi_2 \ar[r]^{*-*}\ar[r]\ar[r] &\Pi_{12}  \ar[r]\ar[r]\ar[r] &0
}
\]
is surjective in the strict sense, and we require that the map:
\begin{displaymath}
\Pi_2\rightarrow \Pi_{12}	
\end{displaymath}
has the image which is assumed in our situation to be dense.
\end{setting}

\begin{lemma} \mbox{\bf{(After Kedlaya-Liu \cite[Lemma 2.7.3]{KL15})}}
Consider a finite glueing datum as above, we then have the fact that 
$M\otimes \Pi_1 \rightarrow M_{1}$ and $M\otimes \Pi_2 \rightarrow M_{2}$ are surjective and the map $M_1\oplus M_2\rightarrow M_{12}$ is also surjective.
\end{lemma}

\begin{proof}
This is a noncommutative version of \cite[Lemma 2.7.3]{KL15}. We just need to check that in our situation the condition of \cref{lemma677} holds. Now consider the bases of $M_1$ and $M_2$ respectively:
\begin{align}
e_1,e_2,...,e_m,\\
f_1,f_2,...,f_m.	
\end{align}
And we consider the corresponding expansion with respect to these bases after taking the image under the maps $s_1:M_1\otimes \Pi_{12}\rightarrow M_{12}$ and $s_2:M_2\otimes \Pi_{12}\rightarrow M_{12}$:
\begin{displaymath}
s_2(f_j)=\sum_{i} I_{ij} s_1(e_i), s_1(f_j)=\sum_{i} J_{ij} s_2(f_i).	
\end{displaymath}
Then we have from the assumption on the dense image some matrix $J'$ such that $I(J'-J)$ could be satisfying the condition $\|I(J'-J)\|\leq c$ where $c$ is the constant in the previous lemma. Then by using the previous lemma we have the following decomposition:
\begin{displaymath}
1+I(J'-J)=K_1K_2^{-1}	
\end{displaymath}
with $K_1\in \mathrm{GL}(\Pi_1),K_2\in \mathrm{GL}(\Pi_2)$. Then we consider the element:
\begin{displaymath}
(a_{j},b_j):=(\sum_{i} K_{1,ij} e_i,\sum_{i} (J'K_2)_{ij} f_i).	
\end{displaymath}
We compute as in the following:
\begin{align}
s_1(a_{j})-s_2(b_j)&=\sum_{i} K_{1,ij} s_1(e_i)-	\sum_{i} J'K_{2,ij} s_2(f_i))\\
                   &=\sum_i K_{1,ij} s_1(e_i)-(IJ'K_{2})_{ij}s_1(e_i))\\
                   &=\sum_i ((1-IJ)K_2)_{ij} s_1(e_i))\\
                   &=0.
\end{align}
Therefore the element we defined is now living in the kernel $M$. Consider just the first component we can see that the map $M\otimes \Pi_1\rightarrow M_1$ is now surjective.	
\end{proof}

%
%

\begin{lemma}\mbox{\bf{(After Kedlaya-Liu \cite[Lemma 2.7.4]{KL15})}}
For finitely projective datum, we then in our situation have the surjectivity of the map $M\otimes \Pi_1\rightarrow M_1$. Then we have the kernel $M$ is finitely presented and we have the isomorphisms $M\otimes \Pi_1\rightarrow M_1$ and $M\otimes \Pi_2 \rightarrow M_2$.
\end{lemma}

\begin{proof}
This is by the previous lemma and \cref{lemma678}.	
\end{proof}

\indent The following is a similar noncommutative version of \cite[Lemma A.7]{DLLZ1}:

\begin{lemma}\mbox{\bf{(After Kedlaya-Liu \cite[Lemma 2.7.4]{KL15})}}
For finitely presented datum, we then in our situation have the surjectivity of the map $M\otimes \Pi_1\rightarrow M_1$. Furthermore suppose we are in the situation where $\Pi_1\rightarrow \Pi_{12}$ and $\Pi_2\rightarrow \Pi_{12}$ are simultaneously assumed to be flat. Then we have the kernel $M$ is finitely presented and we have the isomorphisms $M\otimes \Pi_1\rightarrow M_1$ and $M\otimes \Pi_2 \rightarrow M_2$.
\end{lemma}

\begin{proof}
See lemma above and \cref{lemma679}.	
\end{proof}


\begin{definition} \mbox{\bf{(After Kedlaya-Liu \cite[Definition 5.7.2]{KL16})}}
Over the period rings $\Pi_{H,A}$ (which is denoted by $\triangle$ in this definition) we define the corresponding $(\varphi^a,\Gamma)$-modules over $\triangle$ which are pseudocoherent or fpd to be the corresponding pseudocoherent or fpd right $\Gamma$-modules over $\triangle$ with further assigned semilinear action of the operator $\varphi^a$ with the isomorphism defined by using the Frobenius. We also require that the modules are complete for the natural topology involved in our situation and for any module over $\Pi_{H,A}$ to be some base change of some module over $\Pi^r_{H,A}$ (which will be defined in the following).
\end{definition}

\begin{definition} \mbox{\bf{(After Kedlaya-Liu \cite[Definition 5.7.2]{KL16})}}
Over each ring $\triangle=\Pi^r_{H,A}$ we define the corresponding pseudocoherent or fpd $(\varphi^a,\Gamma)$-module over any $\triangle$ to be the corresponding pseudocoherent or fpd right $\Gamma$-module $M$ over $\triangle$ with additionally endowed semilinear Frobenius action from $\varphi^a$ such that we have the isomorphism $\varphi^{a*}M\overset{\sim}{\rightarrow}M\otimes \square$ where the ring $\square$ is one $\triangle=\Pi^{r/p}_{H,A}$. Also as in \cite[Definition 5.7.2]{KL16} we assume that the module over $\Pi^r_{H,A}$ is then complete for the natural topology and the corresponding base change to $\Pi^I_{H,A}$ for any interval which is assumed to be closed $I\subset [0,r)$ gives rise to a module over $\Pi^I_{H,A}$ with specified conditions which will be specified below.

\end{definition}

\begin{definition} \mbox{\bf{(After Kedlaya-Liu \cite[Definition 5.7.2]{KL16})}}
Again as in \cite[Definition 5.7.2]{KL16}, we define the corresponding pseudocoherent and fpd $(\varphi^a,\Gamma)$-modules over ring $\Pi^{[s,r]}_{H,A}$ to be the pseudocoherent and fpd right $\Gamma$-modules (which will be denoted by $M$) over $\Pi^{[s,r]}_{H,A}$ additionally endowed with semilinear Frobenius action from $\varphi^a$ with the following isomorphisms:
\begin{align}
\varphi^{a*}M\otimes_{\Pi_{H,A}^{[sp^{-ah},rp^{-ah}]}}\Pi_{H,A}^{[s,rp^{-ah}]}\overset{\sim}{\rightarrow}M\otimes_{\Pi_{H,A}^{[s,r]}}\Pi_{H,A}^{[s,rp^{-ah}]}.
\end{align}
We now assume that the modules are complete with respect to the natural topology. 
\end{definition}

\begin{definition} \mbox{\bf{(After Kedlaya-Liu \cite[Definition 5.7.2]{KL16})}}
Over the ring $\Pi^t_{H,A}$ we define a corresponding pseudocoherent and fpd $(\varphi^a,\Gamma)$ bundle to be a family $(M_I)_I$ of pseudocoherent and fpd right $\Gamma$-modules over each $\Pi^I_{H,A}$ carrying the natural Frobenius action coming from the operator $\varphi^a$ such that for any two involved intervals having the relation $I\subset J$ we have:
\begin{displaymath}
M_J\otimes_{\Pi^J_{H,A}}\Pi^I_{H,A}\overset{\sim}{\rightarrow}	M_I
\end{displaymath}
with the obvious cocycle condition. Here we have to propose condition on the intervals that for each $I=[s,r]$ involved we have $s\leq r/p^{ah}$. We require the corresponding topological conditions as above. one can take the corresponding 2-limit in the direct sense to define the corresponding objects over the full Robba rings.
\end{definition}

\begin{remark}
We have the obvious notion of coherent objects in the noetherian setting. Also note that in the noetherian setting we do not have to worry about the corresponding completion issue, which to say the finiteness is all enough. In our current situation and furthermore also in the noetherian setting, by using the corresponding open mapping, one can also mimick the proof in \cite[3.7.3/2,3]{BGR} to prove this by considering the corresponding presentation and take the corresponding kernel.
\end{remark}

\begin{proposition}
Keeping \cref{assumption642} and under the assumption of \cref{conjecture672}, consider the following categories:\\
1. The category of all the pseudocoherent bundles over the ring $\Pi_{H,A}$, carrying the $(\varphi^a,\Gamma)$-action;\\
2. The category of all the pseudocoherent modules over the ring $\Pi^{[s,r]}_{H,A}$, carrying the $(\varphi^a,\Gamma)$-action.\\
\indent Then in our situation they are equivalent, if we have for any intervals permitted $I\subset I'$ we have $\Pi^{I'}_{H,A}\rightarrow \Pi^I_{H,A}$ is flat.	
\end{proposition}

\begin{proof}
First note that we are working over noetherian rings. The functor from the category of the Frobenius bundles to that of the Frobenius modules is just projection of the family of Frobenius modules to the corresponding Frobenius module with respect to the specific interval. In our our situation we have 
\[
\xymatrix@R+1pc@C+1pc{
0 \ar[r]\ar[r]\ar[r] &\Pi \ar[r]\ar[r]\ar[r] &\Pi_1\oplus \Pi_2 \ar[r]^{*-*}\ar[r]\ar[r] &\Pi_{12}  \ar[r]\ar[r]\ar[r] &0
}
\]
is exact in the strict sense due to the existence of Schauder basis, and we have the image of $\Pi_2\rightarrow \Pi_{12}$ is dense as well in the same way. Therefore we can see that the corresponding categories are equivalent since we can show the corresponding essential surjectivity by using the Frobenius action and the glueing process established above to reach any closed interval involved.	
\end{proof}


\newpage

\newpage

\subsection*{Acknowledgements} 

\indent This is our second paper on the Hodge-Iwasawa theory. Largely the initial motivation comes from the work \cite{KP}, \cite{KL15} and \cite{KL16}. Our development is also largely driven by the corresponding deep philosophy rooted in the work pioneered by Burns-Flach-Fukaya-Kato on the noncommutative interaction among motives. The author would like to thank Professor Kedlaya for helpful discussion throughout the whole preparation. During the preparation, under Professor Kedlaya the author was supported by NSF Grant DMS-1844206.

\newpage

\bibliographystyle{ams}

\end{document}